\documentclass[10 pt, reqno]{amsart}
\usepackage{fullpage, amsthm, amsmath, amsfonts, amssymb, url, mathrsfs, stmaryrd, verbatim}
\usepackage{pxfonts}

\author{Casey Rodriguez}
\address{Department of Mathematics\\
  University of Chicago\\
  5734 S. University Avenue \\
  Chicago, IL, 60637}
\email{c-rod216@math.uchicago.edu}

\numberwithin{equation}{section}
\DeclareMathOperator{\arcsinh}{arcsinh}
\newcommand{\R}{\mathbb R}
\newcommand{\C}{\mathbb C}
\newcommand{\N}{\mathbb N}
\newcommand{\Z}{\mathbb Z}
\newcommand{\lims}{\varlimsup}

\newcommand{\M}{\mathcal M}
\newcommand{\s}{\mathbb S}

\newcommand{\vphi}{\varphi}
\newcommand{\ra}{\rangle}
\newcommand{\la}{\langle}

\newcommand{\rar}{\rightarrow}
\newcommand{\p}{\partial}
\newcommand{\al}{\alpha}
\newcommand{\de}{\delta}
\newcommand{\e}{\epsilon}
\newcommand{\tht}{\theta}
\newcommand{\lam}{\lambda}
\newcommand{\h}{\mathcal H}

\newcommand{\om}{\omega}
\newcommand{\cl}{\mathcal}

\newcommand{\supp}{\mbox{supp }}

\newtheorem{lem}{Lemma}[section]
\newtheorem{thm}[lem]{Theorem}
\newtheorem{ppn}[lem]{Proposition}
\newtheorem{defn}[lem]{Definition}
\newtheorem{clm}[lem]{Claim}
\newtheorem{cor}[lem]{Corollary}

\title{Soliton Resolution for Equivariant \\ Wave Maps on a Wormhole: I}

\begin{document}

\begin{abstract}
In this paper, we initiate the study of finite energy equivariant wave maps from the $(1+3)$--dimensional spacetime $\R \times
(\R \times \s^2) \rar \s^3$ where the metric on $\R \times (\R \times \s^2)$ is given by  
\begin{align*} 
ds^2 = -dt^2 + dr^2 + (r^2 + 1) \left ( d \theta^2 + \sin^2 \theta d \varphi^2 \right ), \quad t,r \in \R, 
(\theta,\varphi) \in \s^2. 
\end{align*}
The constant time slices are each given by the Riemannian manifold $\M := \R \times \s^2$
with metric 
\begin{align*}
ds^2 = dr^2 + (r^2 + 1) \left ( d \theta^2 + \sin^2 \theta d \varphi^2 \right ).
\end{align*}
The Riemannian manifold $\M$ contains two asymptotically Euclidean ends at $r \rightarrow \pm \infty$ that
are connected by a spherical throat of area $4 \pi^2$ at $r = 0$.  The spacetime $\R \times \M$ is a simple example 
of a wormhole geometry in general relativity.  In this work we will consider 1--equivariant or corotational wave maps.  Each corotational wave map can be indexed by its topological degree $n$.  For each $n$, there exists 
a unique energy minimizing corotational harmonic map $Q_{n} : \M \rightarrow \s^3$ of degree $n$.  In this work, we show that modulo a free 
radiation term, every corotational wave map of degree $n$ converges strongly to $Q_{n}$.  This resolves 
a conjecture made by Bizon and Kahl in \cite{biz2} in the corotational case.

\end{abstract}

\maketitle




\section{Introduction}

There has been an increased interest in recent years in the study of geometric nonlinear wave equations.  One of the fundamental 
models considered is the following wave map model.  Let $(M,g)$ be a $(1+d)$--dimensional Lorentzian spacetime, and let $(N,h)$ be a 
Riemannian manifold.  A wave map $U: M \rar N$ is a formal critical point of the action functional
\begin{align}\label{s01}
\cl S(U,\p U) = \frac{1}{2} \int_M g^{\mu \nu} \la \p_{\mu} U, \p_{\nu} U \ra_h dg.
\end{align}
In local coordinates, the Euler--Lagrange equations associated to $\cl S$ is the following system of semilinear wave equations 
\begin{align}\label{s02}
\Box_g U^i + \Gamma^i_{jk}(U)  \p_{\mu} U^j \p_{\nu} U^k g^{\mu \nu} = 0, 
\end{align}
where $\Box_g := \frac{1}{\sqrt{-g}} \p_\mu(g^{\mu \nu} \sqrt{-g} \p_\nu)$ is the wave operator associated to the 
background spacetime $(M,g)$ and $\Gamma^{i}_{jk}$
are the Christoffel symbols associated to the target $(N,h)$.  The system  is collectively referred to as the 
\emph{wave map system} and is also known in the physics literature as the classical nonlinear $\sigma$--model.  
A particular case that has been intensely studied is the case when $M$ is $(1+d)$--dimensional 
Minkowski space $\R^{1+d}$ with the flat metric (see the classical reference \cite{shst} and the recent review \cite{sch}).  From a mathematical point of view, a wave map $U: \R^{1+d} \rar N$ 
can be considered as a geometric generalization of the
free wave equation on Minkowski space.  Indeed, if we take $N = \R$ with , then the wave map equations \eqref{s02} reduce to the free wave equation on Minkowski space
\begin{align*}
\p_t^2 U - \Delta U = 0, \quad (t,x) \in \R^{1+d}.
\end{align*}
From a physical point 
of view, wave maps $U: \R^{1+3} \rar \s^3$ describe fields which approximate a low energy regime of QCD (see \cite{geba1} and 
\cite{geba2} for nice introductions to this perspective).  

The case of a curved spacetime is relatively unexplored. In this work, we study corotational wave maps on a \emph{curved} background.  In particular, we consider wave maps $U: \R \times (\R \times \s^2) \rar \s^3$ where the background metric is
given by 
\begin{align}\label{s02b}
ds^2 = -dt^2 + dr^2 + (r^2 + 1)(d\theta^2 + \sin^2 \tht d\vphi^2 ), \quad t,r \in \R, (\theta,\varphi) \in \s^2.
\end{align}
The constant time slices correspond to the Riemannian manifold $\M := \R \times \s^2$ with metric 
\begin{align}\label{s02c}
ds^2 = dr^2 + (r^2 + 1)(d\theta^2 + \sin^2 \tht d\vphi^2 ), \quad r \in \R, (\theta,\varphi) \in \s^2.
\end{align}
Heuristically, $\M$ has two asymptotically Euclidean ends at $r = \pm \infty$ connected by a spherical throat 
at $r = 0$.  In the general relativity literature, the spacetime $\R \times \M$ is a simple example of a \lq wormhole geometry.' 
A corotational wave map $U: \R \times \M \rar \s^3$ is given by the ansatz
\begin{align}\label{s03}
U(t,r,\theta,\vphi) = (\psi(t,r),\theta,\varphi) \in \s^3,
\end{align}
where $\psi$ is the azimuth angle on $\s^3$.  For $U$ given by \eqref{s03}, the action $\cl S$ given by \eqref{s01} reduces
to 
\begin{align*}
\cl S(\psi, \p \psi) = \frac{1}{2} \int_\R \int_\R \left [ -|\p_t \psi(t,r)|^2 + |\p_r \psi(t,r)|^2 + 
\frac{2\sin^2 \psi(t,r)}{r^2 + 1} \right ] (r^2 + 1) dr dt,
\end{align*}
and the wave map equation \eqref{s02} reduces to the single semilinear wave equation
\begin{align}
\begin{split}\label{s04}
&\p_t^2 \psi - \p_r^2 \psi - \frac{2r}{r^2 + 1} \p_r \psi + \frac{\sin 2 \psi}{r^2 + 1} = 0, \quad (t,r) \in \R \times \R,\\
&\vec \psi(0) = (\psi_0,\psi_1).
\end{split}
\end{align}
Here we use the notation $\vec \psi(t) = (\psi(t,\cdot), \p_t \psi(t,\cdot))$.  In this work, solutions $\psi$ to \eqref{s04}
will be referred to as corotational \emph{wave maps on a wormhole}. The equation \eqref{s04} has the following conserved energy along the flow:
\begin{align*}
\mathcal E(\vec \psi(t)) := \frac{1}{2} \int_\R \left [ |\p_t \psi(t,r)|^2 + |\p_r \psi(t,r)|^2 + 
\frac{2\sin^2 \psi(t,r)}{r^2 + 1}  \right ] (r^2 + 1) dr = \mathcal E(\vec \psi(0)). 
\end{align*}
In order for the initial data to have finite energy, we must have for some $m,n \in \Z$, 
\begin{align*}
\psi_0(-\infty) = m\pi \quad \mbox{and} \quad \psi_0(\infty) = n\pi. 
\end{align*}
For a finite energy solution $\vec \psi(t)$ to \eqref{s04} to depend continuously on $t$, we must have that $\psi(t,-\infty) = m
\pi$ and $\psi(t,\infty) = n\pi$ for all $t$.  In this work, we will, without loss of generality, fix $m = 0$ and assume 
$n \in \N \cup \{0\}$.  Thus, we only consider wave maps which send the left Euclidean end at $r = -\infty$ to the 
north pole of $\s^3$.  The integer $n$ is referred to as the topological degree of the map $\psi$ and, heuristically, 
represents the number of times $\M$ gets wrapped around $\s^3$ by $\psi$.  For each $n \in \N \cup \{0\}$, we denote the 
space of finite energy pairs of degree $n$ by 
\begin{align*}
\mathcal E_n := \left \{ (\psi_0,\psi_1) : \mathcal E(\psi_0,\psi_1) < \infty, \quad \psi_0(-\infty) = 0, \quad \psi_0(\infty) = n\pi
\right \}.
\end{align*}
In this work, we classify the long time dynamics of all finite energy corotational wave maps on a wormhole.

There has been widespread belief in the mathematical physics community that for most globally well--posed dispersive equations, a solution
asymptotically decouples into a coherent element and a purely dispersive element.  The coherent element is nonlinear in nature
and is determined by the static solutions and symmetries of the equation (i.e. solitons).  The purely dispersive element
is a solution to the underlying linear equation.  This heuristic belief goes by the name of the \emph{soliton resolution 
	conjecture}.  There are features that wave maps on a wormhole exhibit that make it an interesting model in which to study this phenomenon.
The first feature is that showing global well--posedness, i.e. every solution $\vec \psi(t)$ to \eqref{s04} exists for all $t \in \R$, 
is simple.  The geometry of the domain removes the possibility of singularity formation at the origin and renders the equation 
essentially energy--subcritical.  Another feature of this model is the abundance of finite energy static solutions to 
\eqref{s04} which also go by the name
of harmonic maps.  In particular, it can be shown that for every $n \in \N \cup \{0\}$, there exists a unique solution $Q_n$
to the static equation 
\begin{align}
\begin{split}\label{s04b}
&\p_r^2 F + \frac{2r}{r^2 + 1} \p_r F - \frac{\sin 2 F}{r^2 + 1} = 0, \\
&F(-\infty) = 0, \quad F(\infty) = n\pi,
\end{split}
\end{align}
and $Q_n$ has finite energy (see Section 2).  Moreover, each $Q_n$ is linear stable (see Section 5).  In \cite{biz2}, Bizon and Kahl gave numerical evidence for the following soliton resolution 
conjecture for this model: for every $n \in \N \cup \{0\}$ and for any $(\psi_0,\psi_1) \in \mathcal E_n$, there exist 
a unique globally defined solution $\psi$ to \eqref{s04} and solutions 
$\vphi^{\pm}_L$ to the underlying linear equation 
\begin{align}\label{s05}
\p_t ^2 \vphi - \p_r^2 \vphi - \frac{2r}{r^2 + 1} \p_r \vphi + \frac{2}{r^2 + 1} \vphi = 0, 
\end{align}
such that 
\begin{align*}
\vec \psi(t) = (Q_n, 0) + \vec \vphi_L^{\pm}(t) + o(1),
\end{align*}
as $t \rar \pm \infty$.  In this work we verify this conjecture.  As alluded to in the initial description of the background
$\R \times \M$, the spacetime $\R \times \M$ and Riemannian manifold $\M$ have appeared in contexts outside of 
this work.  For example, $\R \times \M$ has been considered in 
the general relativity as a prototype geometry representing a wormhole
since it was first introduced by Ellis in the 1970's and later popularized by Morris and Thorne in the 1980's 
(cf. \cite{mt} and \cite{fjtt} and the references therein).  Also, the two dimensional version of $\M$ given by 
$\M^2 = \R \times \s^1$ with metric
\begin{align*}
ds^2 = dr^2 + (r^2 + 1) d\vphi^2, 
\end{align*}
is simply an intrinsic description of the classical catenoid surface.

We now turn to stating our main result.  In what follows we use the following notation.  If $r_0 \geq -\infty$ and 
$w(r)$ is a positive continuous function on $[r_0,\infty)$, then we define
\begin{align*}
\| (\psi_0,\psi_1) \|_{\h([r_0,\infty); w(r)dr)}^2 :=
\int_{r_0}^\infty \left [ |\psi_0(r)|^2 + |\psi_1(r)|^2 dr \right ] w(r) dr. 
\end{align*}
The Hilbert space $\h([r_0,\infty); w(r)dr)$ is then defined to be the completion of pairs of smooth functions with compact 
support in $(r_0,\infty)$ under the norm previously defined.  Let $n \in \N \cup \{0\}$ be a fixed topological degree. 
In the $n=0$ case, the natural space to place the solution $\vec \psi(t)$ to \eqref{s04} in is the \emph{energy space}
$\h_0 := \h((-\infty,\infty); (r^2 + 1)dr)$.  Indeed, it is easy to show that $\| \vec \psi \|_{\mathcal E_0} \simeq
\| \vec \psi \|_{\h_0}$.  For $n \geq 1$, we measure distance relative to $(Q_n,0)$ and define $\h_n := \mathcal E_n - (Q_n,0)$
with \lq norm'
\begin{align*}
\| \vec \psi \|_{\h_n} := \| \vec \psi - (Q_n,0) \|_{\h_0}.
\end{align*}
Note that $\psi(r) - Q_n(r) \rar 0$ as $r \rar \pm \infty$. The main result of this work is the following.  

\begin{thm}\label{t01}
	For any energy data $(\psi_0,\psi_1) \in \mathcal E_n$, there exists a unique globally defined solution $\vec \psi(t) 
	\in C(\R; \h_n)$ which scatters forwards and backwards in time to the harmonic map $(Q_n,0)$, i.e. there exist
	solutions $\varphi^{\pm}_L$ to the linear equation \eqref{s05} such that  
	\begin{align*}
	\vec \psi(t) = (Q_n, 0) + \vec \vphi_L^{\pm}(t) + o_{\h_0}(1),
	\end{align*}
	as $t \rar \pm \infty$. 
\end{thm}

We remark that in \cite{biz2} Bizon and Kahl gave numerical evidence that 
soliton resolution holds in the more general $\ell$--equivariant setting (here corotational corresponds to $\ell = 1$).  In the companion work \cite{cr2} we prove this and completely resolve the soliton resolution conjecture for all equivariant wave maps on a wormhole.  

We point out that an equation with properties similar to the model considered in this paper was studied in 
\cite{ls}, \cite{kls1}, and \cite{klls2} and served as a road map for the work carried out here.  In these works, the
authors studied $\ell$--equivariant wave maps $U: \R \times (\R \backslash B(0,1)) \rar \s^3$ such that $U(\p B(0,1)) = \{ (0,0,0,1) \}$.  An $\ell$--equivariant 
wave map $U$ is 
determined by the associated azimuth angle $\psi(t,r)$ which satisfies the equation 
\begin{align}
\begin{split}\label{s05a}
&\p_t^2 \psi - \p_r^2 \psi - \frac{2}{r} \p_r \psi + \frac{\ell(\ell+1)}{2(r^2 + 1)}\sin 2 \psi = 0, \quad t \in \
\R, r \geq 1,\\
&\psi(t,1) = 0, \quad \psi(t,\infty) = n \pi, \quad \forall t.
\end{split}
\end{align}
Such wave maps were called $\ell$--equivariant \emph{exterior wave maps}.  Similar to wave maps on a wormhole, global well--posedness
and an abundance of harmonic maps hold for the exterior wave map equation \eqref{s05a}.  In the works 
\cite{ls}, \cite{kls1}, and \cite{klls2}, the authors proved the soliton resolution conjecture for $\ell$--equivariant exterior
wave maps for arbitrary $\ell \geq 1$.  However, the geometry of the background $\R \times (\R \backslash B(0,1))$
is still flat and could be considered artificial. On the other hand, the wormhole geometry considered in this
work contains curvature which make wave maps on a wormhole more geometric in nature
while still retaining the properties that make them ideal for studying the soliton resolution conjecture.  We remark here that, 
to the author's knowledge, Theorem \ref{t01} is the first result that establishes the soliton resolution conjecture
for arbitrary corotational finite energy wave maps on a curved background.  See \cite{los} for soliton resolution for 
corotational wave maps from $\mathbb \R \times \mathbb H^2 \rar \mathbb H^2$ with a restriction on the behavior at infinity. 

The method of proof used in the works \cite{ls}, \cite{kls1}, and \cite{klls2} to establish the soliton resolution 
conjecture for \eqref{s05a} was the celebrated concentration--compactness/rigidity theorem method pioneered by Kenig and 
Merle in \cite{km06} and \cite{km08}.  In \cite{kls1} and \cite{klls2}, the authors used a \lq channels of energy'
argument based on exterior energy estimates for free waves on $\R^{1+d}$ with $d$ odd to close 
the argument (see \cite{cks} and \cite{klls1}
for these estimates).  The proof of our main result, Theorem \ref{t01}, uses a similar methodology which we now
briefly overview.  The proof is by contradiction
and is split into three main steps.  In the first step, we establish a small data theory for \eqref{s04}, i.e. if 
$\| \vec \psi(0) \|_{\h_n}$ is sufficiently small, then the solution $\psi$ to \eqref{s04} is global and 
scatters to $(Q_n,0)$.  In the second step, using concentration--compactness arguments and the first
step we show that if Theorem \ref{t01} \emph{fails}, then there exists a solution $\vec \psi_* \neq (Q_n,0)$ to \eqref{s04} such that 
the trajectory $\{\vec \psi_*(t) : t \in \R \}$ is precompact in $\h_n$.  In the third and final step, we establish the following
rigidity theorem: if $\psi$ is a solution to \eqref{s04} such that $\{ \vec \psi(t) : t \in \R \}$ is precompact in $\h_n$, then 
$\vec \psi = (Q_n,0)$.  This rigidity theorem contradicts the second step which implies that Theorem \ref{t01} must hold.

We now give an outline of the paper and provide a few more details of the previously sketched steps.  
Section 2, Section 3, and Section 4 contain preliminaries necessary to carry out the concentration--compactness/
rigidity theorem methodology for wave maps on a wormhole.  In Section 2, we establish
various properties of the harmonic maps $Q_n$ needed throughout the work.  In particular, we establish existence, uniqueness, and asymptotics.  Establishing these properties
in the exterior wave map model is considerably simpler since the static solutions to 
\eqref{s05a} (in the corotational case) are governed by the well--known equation for a damped pendulum
\begin{align*}
\frac{d^2 F}{dx^2} + \frac{dF}{dx} = \sin 2F, \quad x = \log r.  
\end{align*}
The properties needed can then be derived from a simple phase plane analysis.  However, in our setting there is no such change of variables 
that renders \eqref{s04b} autonomous.  We instead use classical ODE arguments inspired by the work 
on corotational Skyrmions \cite{mctr} to derive the properties we need.  In Section 3 and Section 4, we
establish results needed to carry out the first two steps in the concentration--compactness
/rigidity theorem methodology.  We first reformulate Theorem \ref{t01} as the statement 
that all radial solutions to a certain semilinear wave equation of the form 
\begin{align}\label{s06}
\p_t^2 u - \Delta_g u + V(r) u = N(r,u), \quad (t,r) \in \R \times \R. 
\end{align}
scatter to free waves as $t \rar \pm \infty$ (see Theorem \ref{t41} for the exact statement).  Here $u$ is related to $\psi$ by
$u = \frac{1}{(r^2 + 1)^{1/2}} (\psi - Q_n)$, $V(r)$ is a smooth potential arising from linearizing about $Q_n$, and $-\Delta_g$ is the Laplace operator on the 
$5d$ wormhole $\M^5 = \R \times \s^4$ with metric
\begin{align*}
ds^2 = dr^2 + (r^2 + 1) d\Omega_{\s^4}^2,
\end{align*}
where $d\Omega_{\s^4}^2$ is the round metric on the sphere $\s^4$.  In the remainder of the paper we carry out the concentration--compactness/
rigidity theorem method in the equivalent \lq $u$--formulation.'  Establishing the first two steps in the $u$--formulation follows
from fairly standard arguments once Strichartz estimates for radial solutions to the free wave equation $\p_t^2 u - \Delta_g u = 0$
are established.  In the exterior wave map model, these estimates follow from previously known results on Strichartz
estimates for free waves on Riemannian manifolds.  However, Strichartz estimates for free waves on a wormhole fall outside of 
the literature devoted to free waves on Riemannian manifolds because of the trapping that occurs at the throat $r = 0$. 
In the works \cite{sss1} and \cite{sss2}, the authors established dispersive estimates
in geometries with trapping which are asymptotic to wormholes as $r \rar \pm \infty$ as long as the initial data is localized to a fixed
spherical harmonic (i.e. angular momentum).  Since
we are only interested in radial free waves on a wormhole, in Section 3 we are able to refine the dispersive estimates from \cite{sss1} and \cite{sss2} 
in the zero angular momentum case and obtain the Strichartz estimates we need.  In fact, we establish Strichartz
estimates for radial free waves on $d$--dimensional wormholes for arbitrary $d \geq 3$.  This section is independent
of all other sections and may be of interest in its own right. In Section 4, we make the reduction previously described and
transfer the Strichartz estimates established in Section 3 for $\p_t^2 - \Delta_g$ to the perturbed operator $\p_t^2 - 
\Delta_g + V$.  The fact that the Strichartz estimates for the free wave operator carry over to the perturbed operator
hinges on spectral information for the Schr\"odinger operator $-\Delta_g + V$.  In Section 5
and Section 6, we use the concentration--compactness/rigidity theorem method to prove our main result.  In Section 5 we carry out the 
first two steps of the concentration--compactness/rigidity theorem methodology in the $u$--formulation.  The main result of this 
section is that if Theorem \ref{t01} fails so that not all solutions to \eqref{s06} scatter, then there exists a nonzero solution $u_*$ to \eqref{s06} such that 
$\{ \vec u_*(t) : t \in \R \}$ is precompact in $\h := \h((-\infty,\infty); (r^2 + 1)^2 dr)$.  In Section 6, we show that a solution
$u$ to \eqref{s06} such that $\{ \vec u(t) : t \in \R \}$ is precompact in $\h$ must be identically 0 which completes the proof.
In particular, we show that $u$ is zero by showing it must be a static solution to \eqref{s06} with finite energy.  This is 
achieved using a change of variables valid in the exterior regions $|r| \gtrsim 1$ that transforms \eqref{s06} into an 
`exterior wave map equation'.  We then use channels of energy arguments similar
to those used in \cite{kls1} and \cite{klls2} to show that $u$ is a static solution to \eqref{s06}.  This then implies that 
$\psi = Q_n + (r^2 + 1)^{1/2} u$ satisfies \eqref{s04b}.  By the uniqueness of harmonic maps, we deduce that $u \equiv 0$
and conclude the proof.   

\textbf{Acknowledgments}: This work was completed during the author’s doctoral studies at the University of Chicago. The author
would like to thank his adviser, Carlos Kenig, for his invaluable guidance and careful reading of the original manuscript.  The author would also like to thank Wilhelm Schlag, Andrew Lawrie, and 
Piotr Bizon for helpful discussions and encouragement during the completion of this work.




\section{Harmonic Maps}

For the remainder of the paper, we fix a topological degree $n \in \N \cup \{ 0 \}$.  
In this section, we study static solutions to \eqref{s04}.  In particular, we prove the following. 

\begin{ppn}\label{harm}
	There exists a unique smooth solution $Q_{n}$ to the equation   
	\begin{align*}
	&F'' + \frac{2r}{r^2 + 1}F' - \frac{\sin2F}{r^2 + 1} = 0, \quad r \in \R, \\
	&F(-\infty) = 0, \quad F(\infty) = n\pi.
	\end{align*}
	In the case $n = 0$, $Q_{0} = 0$.  For $n \in \N$, $Q_{n}$ is increasing on $\R$, satisfies $Q(r) + Q(-r) = n\pi$ for all $r$  
	and there exists  
	$\alpha_{n} > 0$ such that, 
	\begin{align*}
	Q_{n}(r) &= n \pi - \alpha_{n} r^{-2} + O(r^{-4}), \quad \mbox{as } r \rightarrow \infty, \\
	Q_{n}(r) &= \alpha_{n} r^{-2} + O(r^{-4}), \quad \mbox{as } r \rightarrow -\infty.
	\end{align*}
	The $O(\cdot)$ terms also satisfy the natural derivative bounds. 
\end{ppn}

The proof of existence follows from a simple shooting argument sketched in \cite{biz2}.  The proof of uniqueness
and properties needed are inspired by the work on the equivariant Skyrme equation \cite{mctr}. 
The proof of Proposition \ref{harm} will be contained in the following various lemmas.

\subsection{Existence of Harmonic Maps}

In this section we prove the existence part of Proposition \ref{harm}.  In order to achieve this and, in fact, uniqueness
of the harmonic map constructed, 
we will need to study general solutions to 
\begin{align}
F'' + \frac{2r}{r^2 + 1}F' - \frac{ \sin 2F}{r^2 + 1} = 0, \quad r \in \R. \label{ode}
\end{align}

We begin with the following simple lemma. 

\begin{lem}\label{odeprop1}
	If $F$ is a solution to \eqref{ode}, then $F$ exists on all of $\R$.  Moreover, $F$ has limits at $\pm \infty$ in 
	$\Z \pi \cup \left (\Z + \frac{1}{2} \right ) \pi$. 
\end{lem}

\begin{proof}
	Suppose that $F$ solves \eqref{ode}.  Due to the sublinear growth in $F, F'$ in \eqref{ode}, it follows from standard ODE theory that 
	$F$ is globally defined.  Because of the invariance of the equation under the change $r \leftrightarrow -r$, 
	we need only show that $F$ has a limit at $\infty$.  
	
	Define the following auxiliary function 
	\begin{align*}
	Q(r) = (r^2 + 1) \frac{(F')^2}{2} -  \sin^2 F. 
	\end{align*}
	Using that $F$ solves \eqref{ode}, we have that 
	\begin{align}
	Q'(r) = -r (F')^2. \label{hm1}
	\end{align}
	Thus, $Q$ is nonincreasing on $r \geq 0$ and by definition is also bounded below.  Thus, $Q(r) \rightarrow c \in [-1,\infty)$ as 
	$r \rightarrow \infty$. Moreover, we note that 
	\begin{align*}
	((r^2 + 1)Q)' = -2r \sin^2 F \leq 0,
	\end{align*}
	so that 
	\begin{align*}
	Q(r) \leq \frac{Q(0)}{r^2 + 1}, \quad r \geq 0. 
	\end{align*}
	This implies $c \leq 0$. 
	
	The previous bound on $Q$ implies that 
	\begin{align*}
	F'(r)^2 &= \frac{Q(r)}{r^2 + 1} + \frac{\sin^2 F(r)}{r^2 + 1} \sin^2 F(r)
	= O \left ( \frac{1}{r^2} \right ). 
	\end{align*}
	We now claim that $F'$ isn't just $O(r^{-1})$ but in fact satisfies
	\begin{align*}
	F'(r) = o \left ( \frac{1}{r} \right ). 
	\end{align*}
	Suppose towards a contradiction that this is not the case. Then there exist $\delta > 0$ and a sequence $r_n \rightarrow \infty$ with the property 
	\begin{align*}
	\frac{\delta}{r_n} \leq |F'(r_n)|.
	\end{align*}
	Since $F$ solves \eqref{ode}, we have that 
	\begin{align*}
	|F''(r)| \leq \frac{K}{r^2}. 
	\end{align*}
	Thus, for $r_n \leq r \leq (1 + \delta/2K ) r_n$, we have 
	\begin{align*}
	| F'(r) - F'(r_n) | \leq K \int_{r_n}^r \rho^{-2} d \rho \leq K \left ( \frac{1}{r_n} - \frac{1}{r} \right ) 
	\leq \frac{\delta}{2r_n},
	\end{align*}
	so that 
	\begin{align*}
	|F'(r)| \geq \frac{\delta}{2r_n}, \quad r_n \leq r \leq (1 + \delta/2K) r_n. 
	\end{align*}
	Hence
	\begin{align*}
	- Q'(r) = r F'(r)^2 \geq \frac{\delta^2}{4r_n}, \quad r_n \leq r \leq (1 + \delta/2K) r_n.
	\end{align*}
	The previous estimate implies that 
	\begin{align*}
	|Q(r_n) - Q((1 + \delta/2K) r_n)| = \int_{r_n}^{(1+\delta/2K)r_n} -Q(r) dr \geq \frac{\delta^3}{8K},
	\end{align*}
	which contradicts the fact that $\lim_{r \rightarrow \infty} Q(r)$ exists in $[-1,0]$.  Thus, the claim 
	$F'(r) = o(r^{-1})$ holds. 
	
	We now show that as $r \rightarrow \infty$, $F(r)$ tends to $k\pi$ or $\left ( k + \frac{1}{2} \right ) \pi$ 
	for some $k \in \Z$.   
	Since $F'(r) = o(r^{-1})$ and $Q(r) \rightarrow c \in [-1,0]$, we have that 
	\begin{align*}
	\sin^2 F(r) \rightarrow  \tilde c \in [0,1], \quad r \rightarrow \infty.
	\end{align*}
	Thus, $F(r)$ tends to some limit $F_\infty \in \R$ as $r \rightarrow \infty$. Since $F$ solves \eqref{ode} and satisfies
	$F'(r) = o(r^{-1})$, we see that 
	\begin{align*}
	(r^2 + 1) F''(r) \rightarrow \sin 2F_\infty, \quad \mbox{as } r \rightarrow \infty. 
	\end{align*}
	If $\sin 2 F_\infty \neq 0$, then for large $r$ we have 
	\begin{align*}
	F'(r) = \int_r^{\infty} F''(\rho) d\rho  \sim \sin 2 F_\infty \int_r^{\infty} \frac{1}{\rho^2 + 1} d\rho
	\sim \sin 2 F_\infty \frac{1}{r}, 
	\end{align*}
	which contradicts $F'(r) = o(r^{-1})$. Thus, we must have that $\sin 2 F_\infty = 0$ as desired. 
\end{proof}

We remark that we will now be interested solely in solutions to \eqref{ode} which satisfy $F(\pm \infty) \in \Z \pi$.  This is because these solutions are the only solutions that have the potential to have finite energy $\cl E(F,0)< \infty$.  Using Lemma \ref{odeprop1} we can establish the following asymptotics for solutions to \eqref{ode}.

\begin{lem}\label{odeprop2}
	Suppose that $F$ solves \eqref{ode} and there exists $k \in \N \cup \{ 0 \}$ such that $
	F(\infty) = k \pi.$
	Then there exists $\alpha \in \R$ such that 
	\begin{align}\label{odeprop22}
	F(r) = k\pi + \alpha r^{-2} + O(r^{-4}), 
	\end{align}
	as $r \rightarrow \infty$ where the $O(\cdot)$ term satisfies the natural derivative bounds.  A similar statement holds as $r \rightarrow -\infty$ if $F(-\infty) = k \pi$. 
\end{lem}

We note that Lemma \ref{odeprop2} provides the asymptotics stated in Proposition \ref{harm}. 

\begin{proof}
	The proof of Lemma \ref{odeprop2} follows in almost exactly the same way as the proof of Case 1 of Theorem 2.3 in \cite{mctr}.  The idea is to 
	make the change of variables $x = \arcsinh r$ and use the fact that $dF/dx = rdF/dr = o(1)$ to write \eqref{ode} 
	as 
	\begin{align}\label{odeprop21}
	\frac{d^2 F}{dx^2} + \frac{dF}{dx} - \sin 2F + O(e^{-2x} ) = 0.  
	\end{align}
	The ODE \eqref{odeprop21} is asymptotically the autonomous ODE $F'' + F' - \sin 2F = 0$ (the damped pendulum) near $x = \infty$ for which 
	the desired expansion \eqref{odeprop22} holds in the $x$ variable.  We omit the details 
	and refer the reader to the proof of Case 1 in Theorem 2.3 in \cite{mctr} for the full details of the argument. 
\end{proof}

A fact that will be useful in Section 6 is that one can obtain a solution to \eqref{ode} with prescribed 
asymptotics as $r \rar \infty$.  

\begin{ppn}\label{prescribe}
	Let $k \in \N \cup \{ 0 \}$, and let $\alpha \in \R$.  Then there exists a unique solution $F_\alpha$ to \eqref{ode} such that 
	\begin{align}
	F_\alpha(r) = k \pi + \alpha r^{-2} + O ( r^{-4} )\label{pre1}
	\end{align}
	as $r \rar \infty$ where the $O(\cdot)$ term satisfies the natural derivative bounds. 
\end{ppn}

Before giving the proof, we note that the symmetry $r \mapsto -r$ of \eqref{ode} allows us 
to conclude from Proposition \ref{prescribe} that given $k \in \N \cup \{0\}$ and $\beta \in \R$, there exists a solution 
$F_{\beta}(r)$ to \eqref{ode} such that
\begin{align*}
F_\beta(r) = k \pi + \beta r^{-2} + O ( r^{-4})
\end{align*}
as $r \rar -\infty$. 

\begin{proof}
	We seek a solution $F$ to \eqref{ode} with the stated asymptotics \eqref{pre1}.  We first make the change of variables 
	$x = \arcsinh r$ so that \eqref{ode} becomes 
	\begin{align}
	F'' + \tanh x F' -  \sin 2F = 0, \quad x \in \R, \label{pre2}
	\end{align}
	where $F' = \frac{dF}{dx}$. We now rewrite \eqref{pre2} as 
	\begin{align}
	F'' + F' - 2 F = \left [ \sin 2F - 2F \right ] + \left ( 1 - \tanh x \right ) F'.
	\end{align}
	Define $G = e^{x/2}(F - k\pi)$.  Then $G$ satisfies 
	\begin{align}
	G'' + \frac{9}{4} G = N(x, G, G'), \label{pre3}
	\end{align}
	where 
	\begin{align}
	N(x,G,G') = e^{x/2} \left [ \sin (2 e^{-x/2} G) - 2 e^{-x/2} G \right ] + 
	\left ( 1 - \tanh x \right ) \left [ G' - \frac{1}{2} G \right ]. \label{pre4}
	\end{align}
	A fundamental system to the underlying linear equation $G'' - \frac{9}{4} G = 0$ is given by 
	\begin{align*}
	G_1(x) = e^{-3x/2}, \quad 
	G_2(x) = e^{3x/2}.
	\end{align*}
	The Wronskian $W(G_1,G_2) = G_1' G_2 - G_1 G_2'$ is given by $-3$.  By the variation of constants formula, we seek 
	a solution $G = G_\alpha$ to the integral equation 
	\begin{align}
	G = \alpha G_1(x) + \frac{1}{3} \int_x^\infty \left [ G_1(x)G_2(y) - G_1(y)G_2(x) \right ]
	N(y,G, G') dy,  \label{pre5}
	\end{align}
	for $x \geq R$ for some $R$. For $R > 0$, define the 
	Banach space $X_R = \left \{ G \in C^1([R,\infty)) : \| G \|_{X_R} < \infty \right \}$ where 
	\begin{align*}
	\| G \|_{X_R} := \sup_{x \geq R} e^{3x/2} \left [ |G(x)| + |G'(x)| \right ]. 
	\end{align*}
	Denote the right side of \eqref{pre5} by $\Phi(G)$.  From \eqref{pre4}, it is easy to see that 
	\begin{align*}
	|N(y,G,G')| \leq e^{-y} |G|^3 + e^{-2y} \left [ |G| + |G'| \right ].  
	\end{align*}
	Thus, 
	\begin{align*}
	\| \Phi(G) \|_{X_R} \leq |\alpha| + 3|\alpha|/2+ 
	C \left [ e^{-R} \| G \|_{X_R}^3 + e^{-2R} \| G \|_{X_R} \right ].  
	\end{align*}
	For $R$ sufficiently large, a fixed point argument yields the existence of a unique solution $G_\alpha$ to \eqref{pre5}.  Moreover, 
	$G_\alpha$ satisfies
	\begin{align*}
	G_\alpha(x) = \alpha e^{-3x/2} + O ( e^{-7x/2} )
	\end{align*}
	as $x \rar \infty$.  This means $F_\alpha(x) = k\pi + e^{-x/2} G_\alpha(x)$ satisfies \eqref{pre2} and 
	\begin{align*}
	F_\alpha(x) = k \pi + \alpha e^{- 2 x} + O ( e^{-4 x} )
	\end{align*}
	as $x \rar \infty$.  This is the same as \eqref{pre1} under the change of variables $r = \sinh x$. This concludes the proof of existence of $F_\alpha$. 
	Uniqueness follows from the fixed point argument and Lemma \ref{odeprop2}.
\end{proof}

Using Lemma \ref{odeprop2} and monotonicity of the auxiliary function $Q(r)$, we deduce the following monotonicity
result for solutions to \eqref{ode}.

\begin{lem}\label{odeprop3}
	Suppose that $F$ solves \eqref{ode} and 
	\begin{align*}
	F(-\infty) = l\pi, \quad
	F(\infty) = k\pi.
	\end{align*}
	Then $F$ is monotonic on $\R$.  In particular, if $l = k$, then $F$ is the constant solution.
\end{lem}

\begin{proof}
	Recall from the proof of Lemma \ref{odeprop1} that the function 
	\begin{align*}
	Q(r) = (r^2 + 1) \frac{(F')^2}{2} - \sin^2 F,
	\end{align*}
	satisfies $Q'(r) = -r (F')^2$.  In particular the function $Q$
	is nondecreasing on $(-\infty,0)$ and nonincreasing on $(0,\infty)$.    
	
	By Lemma \ref{odeprop2}, there exist $\beta_{\pm} \in \R$ such that as $r \rightarrow -\infty$
	\begin{align*}
	Q(r) = \beta^2 r^{-4} + O(r^{-6}).
	\end{align*}
	Moreover, the case $\beta_{+} = 0$ or $\beta_{-} = 0$ corresponds to the constant solution (which is 
	trivially monotonic). We will assume that $\beta_{\pm} \neq 0$,
	and therefore, $F$ is not the constant solution. Thus, if $|r|$ large, then $Q(r)$ is positive.  
	
	We now conclude that $F$ has no critical points. If not, and there exists $r_0 \in \R$, a critical point for $F$, then $Q(r_0) \leq 0$.  In particular, since $Q(r)$ is nondecreasing 
	on $(-\infty, 0)$ from a positive value near $r = -\infty$, we must have that $r_0 > 0$.  However, since $F$ is nonconstant, 
	$Q(r)$ is strictly decreasing on $[0,\infty)$ since $Q'(r) = -r(F')^2$.  Thus, 
	we have $Q(r) < Q(r_0) \leq 0$ for all $r > r_0$.  This contradicts the fact that $Q(r) > 0$ for 
	large positive $r$. Thus, $F$ has no critical points so that $F$ is monotonic on $\R$.   
\end{proof}

We now prove the existence part of Proposition \ref{harm}.

\begin{lem}\label{odethm1}
	For each $n \in \N$, there exists a solution $Q_{n}$ to \eqref{ode} that satisfies
	\begin{align*}
	Q_{n}(-\infty) = 0, \quad Q_n(\infty) = n \pi, \\
	\forall r, \quad Q_n(r) + Q_n(-r) = n\pi. 
	\end{align*}
\end{lem}

The proof of Lemma \ref{odethm1} now follows from the previous lemmas and a classical shooting argument sketched in \cite{biz2}. For every $\alpha \in (0,\infty)$, define $F(r,\alpha)$ to be the solution
to \eqref{ode} such that 
\begin{align*}
F(0,\alpha) &= \frac{n\pi}{2}, \\
F'(0,\alpha) &= \alpha.
\end{align*}
The variable $\alpha$ is referred to as the shooting variable.  We will show that 
we can choose $\alpha$ so that $F(\infty,\alpha) = n\pi$.  Note that if $F(\infty,\alpha) = n\pi$ for some 
$\alpha$, then the symmetry $F \mapsto n\pi - F$ of the equation yields $F(-r,\alpha) + F(r,\alpha) = n\pi$ so that 
$F(-\infty,\alpha) = 0$. Thus, 
to prove Lemma \ref{odethm1}, it suffices to show there exists $\alpha^* \in (0,\infty)$ such that 
\begin{align*}
F(\infty,\alpha^*) = n\pi. 
\end{align*}
We then set $Q_{n}(r) = F(r,\alpha^*)$. 

Define 
\begin{align*} 
A := \left \{ \alpha \in (0,\infty) : \lim_{r \rightarrow \infty} F(r,\alpha) < n \pi \right \} 
\end{align*}
The proof of Lemma \ref{odethm1} requires a few claims.   

\begin{clm}\label{clm1}
	There exists $\alpha_0 > 0$ so that $(0,\alpha_0) \subset A$. 
\end{clm}

\begin{proof}
	For $\al \in (0,\infty)$, we denote 
	\begin{align*}
	Q(r,\al) = (r^2 + 1) \frac{(F'(r,\al))^2}{2} -  \sin^2 F(r,\al).
	\end{align*}
	The proof is split into two cases depending on whether $n$ is odd or even.
	
	\emph{Case 1.} We first consider the case that $n$ is odd.  Then we may take $\al_0 = \sqrt{2}$.  Indeed, 
	if $\al \in (0,\sqrt{2})$, then 
	\begin{align*}
	Q(0,\al) = \frac{\al^2}{2} -  \sin^2 \left ( \frac{n\pi}{2} \right ) = \frac{\al^2}{2} - 1 < 0.
	\end{align*}
	Since $Q(r,\al)$ is decreasing on $(0,\infty)$, we must have $Q(r,\al) < 0$ for all $r > 0$.  This implies that $F(r_0,\al) \neq 
	n \pi$ for all $r_0 \in (0,\infty)$.  The case that $F(r,\al) \rightarrow n\pi$ as $r \rightarrow \infty$ is also impossible
	since then $Q(r,\al) > 0$ for $r$ sufficiently large (see the proof of Lemma \ref{odeprop3}). Thus, if $n$ is odd, 
	we have $(0,\sqrt{2}) \subset A$. 
	
	\emph{Case 2.} We now consider the case that $n$ is even.  In particular, $\frac{n \pi}{2} = l \pi$ for some $l \in \N$.  We first note that 
	for every $\al \in (0,\infty)$, $F(\cdot,\al)$ is increasing until $F$ leaves the strip $\left ( l \pi, \left ( l + \frac{1}{2} \right ) \pi \right )$. 
	Indeed, if $F$ attains a local maximum for some $r_0$ with $F(r_0,\al) \in 
	\left ( l \pi, \left ( l + \frac{1}{2} \right ) \pi \right )$, then \eqref{ode} implies
	\begin{align*}
	F''(r_0,\al) = \frac{\sin 2F(r_0,\al)}{r_0^2 + 1} > 0.
	\end{align*}
	Thus, $F(\cdot,a)$ is increasing as long as $F \in \left ( l \pi, \left ( l + \frac{1}{2} \right ) \pi \right )$.
	
	Note that since $\frac{n \pi}{2} = l\pi$ for some integer $l$, we have 
	\begin{align*}
	Q(0,\al) = \frac{\al^2}{2}.
	\end{align*}
	We recall that $[(r^2 + 1)Q(r,\al)]' = -2 r \sin^2 F(r,\al)$ so that  
	\begin{align}
	Q(r,\al) \leq \frac{Q(0,\al)}{r^2 + 1} = \frac{\al^2}{2(r^2 + 1)}. \label{clm3e1}
	\end{align}
	Thus, for all $\al$ sufficiently small, we have
	\begin{align}
	F'(r,\al)^2 = \frac{2Q(r,\al)}{(r^2 + 1)^2} + \frac{2}{r^2 + 1} \sin^2 F(r,\al) < \frac{4}{r^2}. \label{clm11}
	\end{align}
	Moreover, by continuity of the initial value problem, for $\al$ sufficiently small, we can also ensure that 
	\begin{align*}
	F(r,\al) &< \left ( l + \frac{1}{6}  \right ) \pi, \quad r \in [0,1].
	\end{align*}
	Fix $\al \in (0,\alpha_0)$ with $\alpha_0$ small to be chosen, and suppose that 
	$F(r,\al)$ leaves the strip $\left ( l \pi, \left ( l + \frac{1}{2} \right ) \pi \right )$ (if not then $\al \in A$ trivially).
	Since $F(\cdot,\al)$ is increasing until it reaches $\left ( l + \frac{1}{2} \right ) \pi$, there exist $1 < r_1 < r_2$ 
	such that 
	\begin{align*}
	F(r_1, \al) = \left ( l + \frac{1}{6} \right ) \pi, \\
	F(r_2, \al) = \left ( l + \frac{1}{4} \right ) \pi. 
	\end{align*}
	Then the fundamental theorem of calculus and \eqref{clm11} imply that 
	\begin{align*}
	\frac{\pi}{4} - \frac{\pi}{6} = \int_{r_1}^{r_2} F'(r,\al) dr < 2 \log (r_2 / r_1),
	\end{align*}
	so that 
	\begin{align*}
	r_2 - r_1 > \left (e^{\pi/24} - 1 \right ) r_1 > e^{\pi/24} - 1 . 
	\end{align*}
	By \eqref{clm3e1}
	\begin{align*}
	(r_2^2 + 1) Q(r_2,\al) &= (r_1^2 + 1) Q(r_1,\al) - 2\int_{r_1}^{r_2} r \sin^2 F(r,\al) dr \\ 
	&\leq Q(0,\al) - \frac{1}{2} \int_{r_1}^{r_2} r dr \\
	&= \frac{\al^2}{2} - \frac{1}{4}(r_2^2 - r_1^2) \\
	&< \frac{\al_0^2}{2} - \frac{e^{\pi/24} - 1}{8}.
	\end{align*}
	Thus, if we choose $\al_0$ so that $\al^2_0 < \frac{e^{\pi/24} - 1}{8}$, we have, for all $\al \in (0,\al_0)$, $Q(r_2,\al) < 0$.  Since 
	$Q(r,\al)$ is decreasing on $(0,\infty)$, it follows that $Q(r,\al) < 0$ for all $r > r_2$.  Thus, we cannot have
	$F(r,\al) = (l+1)\pi$ for any $r \in (0,\infty]$ so that 
	\begin{align*}
	F(\infty,\al) < n \pi. 
	\end{align*}
	Thus, if $\al_0$ is sufficiently small, $\al \in A$ for all $\al \in (0,\al_0)$. 
\end{proof}

\begin{clm}\label{clm3}
	The set $A$ is open. 
\end{clm}
We recall that 
\begin{proof}
	Let $\alpha_0 \in A$.  We consider two cases. 
	
	\emph{Case 1.}  In this case, we assume that there exists $m < n$ such that  
	\begin{align*}
	F(\infty,\alpha_0) = \left ( m + \frac{1}{2} \right) \pi.
	\end{align*}
	We first note that for all $r \geq 0$ 
	\begin{align}\label{clm31}
	F(r,\alpha_0) < (m + 1) \pi.
	\end{align}
	Indeed, if this were not the case, then, since $F(r,\alpha_0)$ is not constant and $F(\infty,\alpha_0) < (m+1) \pi$, there exist
	$r_1 < r_2 < r_3$ such that 
	\begin{align*}
	F(r_1,\alpha_0) = F(r_3&,\alpha_0) = (m+1)\pi, \\
	F'(r_1,\alpha_0) \neq 0,\quad F'(r_2&,\alpha_0) = 0, \quad F'(r_3,\alpha_0) \neq 0. 
	\end{align*}
	In particular, $Q(r_2,\alpha_0) \leq 0$.  But since $Q(r,\alpha)$ is decreasing on $[0,\infty)$, it follows that $Q(r_3, \alpha_0) < 0$ 
	which is a contradiction to our choice of $r_3$.  Thus, for 
	all $r \geq 0$
	\begin{align*}
	F(r,\alpha_0) < (m + 1) \pi.
	\end{align*}
	
	Since $F(\infty,\alpha_0) = \left ( m + \frac{1}{2} \right ) \pi$ and $F'(r,\alpha_0) = o(r^{-1})$ (see Lemma \ref{odeprop1}), there 
	exists
	$R_0 = R_0(\alpha_0)$ large so that
	\begin{align*}
	Q(R_0,\alpha_0) < 0.
	\end{align*}
	By continuous dependence of $F(\cdot,\alpha)$ on $\alpha$, we can ensure that for all $\alpha$ is a small neighborhood of $\alpha_0$ 
	we have 
	\begin{align}
	F(r,\alpha) &< (m + 1)\pi, \quad r \in [0,R_0], \label{clm32} \\
	Q(R_0,\alpha) &< 0. \label{clm33}
	\end{align}
	Since $Q(r,\alpha)$ is decreasing on $[0,\infty)$, \eqref{clm33} implies for all $\alpha$ sufficiently close to $\alpha_0$, $Q(r,\alpha) < 0$ 
	for all $r \geq R_0$.  In particular, $F(r,\alpha) \neq l \pi$ for any $l \in \N$ and all $r \in [R_0,\infty]$.  This along with \eqref{clm32} implies that
	$F(\infty,\alpha) < (m+1)\pi$.  Thus, for all $\alpha$ sufficiently close to $\alpha_0$, we have $\alpha \in A$ as desired. 
	
	\emph{Case 2.} In this case we assume that there exists $m < n$ such that 
	\begin{align*}
	F(\infty,\alpha_0) = m \pi.
	\end{align*}
	
	We first note that in this case, we have $Q(r,\alpha_0)  = O(r^{-4})$ (see the proof of Lemma \ref{odeprop3}), so that, 
	\begin{align}\label{clm34}
	\lim_{r \rightarrow \infty} (r^2 + 1) Q(r,\alpha_0) = 0. 
	\end{align} 
	Let $\epsilon_0 > 0$ to be chosen later.  Then by \eqref{clm34}, there exists $R_0 = R_0(\epsilon_0) > 1$ such that 
	\begin{align*}
	(r^2 + 1) Q(r,\alpha_0) < \epsilon_0, \quad r \geq R_0.
	\end{align*}
	For $\alpha$ in a small (depending on $\epsilon_0$) neighborhood of $\alpha_0$, we have 
	\begin{align}
	Q(0,\alpha) &< 2 Q(0,\alpha_0), \label{clm35} \\
	(R_0^2 + 1) Q(R_0,\alpha) &< 2 \epsilon_0, \label{clm36} \\
	\frac{n\pi}{2} \leq F(r,\alpha) &< m \pi, \quad r \in [0,R_0]. \label{clm37}
	\end{align}
	We now claim that for each such $\alpha$, we have 
	\begin{align}\label{clm38} 
	F(\infty,\alpha) \leq \left ( m + \frac{1}{2} \right ) \pi. 
	\end{align}
	Let $\alpha$ be sufficiently close to $\alpha_0$ so that \eqref{clm35}, \eqref{clm36}, and \eqref{clm37} are satisfied, and assume that 
	\begin{align*}
	F(\infty,\alpha) > m \pi. 
	\end{align*}
	Then by \eqref{clm37}, there exists $r_0 \geq R_0$ such that $F(r_0,\alpha) = m\pi$. Since $F(\cdot,\alpha)$ is increasing as long 
	as $F(\cdot,\alpha)$ is in the strip $\left ( m \pi, \left ( m + \frac{1}{2} \right )\pi \right )$ (see the proof of Claim \ref{clm1}), there 
	exist $r_1,r_2 > R_0$ such that $r_1 < r_2$ and 
	\begin{align*}
	F(r_1,\alpha) = \left ( m + \frac{1}{6} \right ) \pi, \\
	F(r_2,\alpha) = \left ( m + \frac{1}{4} \right ) \pi.
	\end{align*}
	As in the proof of Claim \ref{clm1}, by \eqref{clm35} we have 
	\begin{align}\label{clm39}
	F'(r,\alpha)^2 \leq \frac{2 Q(0,\alpha)}{(r^2 + 1)^2} + \frac{2}{r^2 + 1} \sin^2 F(r,\alpha) \leq
	\frac{C^2(\alpha_0)}{r^2}, 
	\end{align}
	for some positive constant $C(\alpha_0)$.  By our choice of $r_1,r_2$, \eqref{clm39}, and the fundamental theorem of calculus, 
	we deduce that
	\begin{align*}
	\frac{\pi}{4} - \frac{\pi}{6} &= \int_{r_1}^{r_2} F'(r,\alpha) dr \geq C(\alpha_0) \log (r_2 / r_1),
	\end{align*}
	whence for some (possibly small) constant $c(\alpha_0) > 0$
	\begin{align*}
	r_2 - r_1 \geq c(\alpha_0).
	\end{align*}
	By the relation $[(r^2 + 1) Q(r,\alpha)]' = -2 r \sin^2 F(r,\alpha)$ and \eqref{clm36}, we have 
	\begin{align*}
	(r_2^2+ 1) Q(r_2,\alpha) &= (r_1^2 + 1) Q(r_1,\alpha) - 2 \int_{r_1}^{r_2} r \sin^2 F(r,\alpha) dr \\
	&< 2 \epsilon_0 -  \frac{1}{2} \int_{r_1}^{r_2} r dr \\
	&< 2 \epsilon_0 - \frac{1}{2} (r_2^2 - r_1^2 ) \\
	&< 2 \epsilon_0 - \frac{1}{2} c(\alpha_0) .
	\end{align*}
	By initially choosing $\epsilon_0$ sufficiently small (depending only on $\alpha_0$), we see that if $\alpha$ is sufficiently close to $\alpha_0$ so that \eqref{clm35}, 
	\eqref{clm36}, and \eqref{clm37} are satisfied, we have $Q(r_2, \alpha) < 0$.  Thus, $Q(r,\alpha) < 0$ for all $r \geq r_2$.  
	Hence, for any $l > m$, $F(r,\alpha) \neq l \pi$ for all $r \in [R_0,\infty]$.  This along with \eqref{clm37} proves that 
	$F(r,\alpha) \leq \left ( m + \frac{1}{2} \right ) \pi$ for all $r \geq 0$ which establishes \eqref{clm38}.  Thus, all $\alpha$ sufficiently 
	close to $\alpha_0$ are in $A$ which finishes the proof of Claim \ref{clm3}.  
\end{proof}

\begin{clm}\label{clm2}
	There exists $\alpha_1 > 0$ such that $(\al_1,\infty) \subseteq A^c$. 
\end{clm}

\begin{proof}
	We first note that if $\al > 0$ and if $F(r,\al) = n\pi$ for some $r > 0$, then $F(\infty,\al) > n\pi$.  Indeed, suppose $F(r_0,\al) = n\pi$ for some $r_0 > 0$ 
	and $F(\infty,\al) \leq n\pi$.  Since $F(r,\al)$ is not the constant function, there exist $r_0 < r_1 < r_2 \leq \infty$ such that
	$F'(r_1,\al) = 0$ and $F(r_2,\al) = n\pi$.  We then have that $Q(r_1,\al) \leq 0$ and $Q(r_2,\al) > 0$.  This contradicts the 
	fact that $Q(r,\al)$ is decreasing on $[0,\infty)$.  Thus, if $F(r,\al) = n \pi$ for some 
	$r > 0$, then $F(\infty,\al) > n\pi$.  In particular, we have shown that 
	\begin{align*}
	\left \{ a > 0 : F(r_0,a) = n\pi \mbox{ for some } r_0 > 0 \right \} \subset A^c. 
	\end{align*}
	Thus, the proof of Claim \ref{clm2} is reduced to showing that there exists $\al_1 > 0$ such that 
	\begin{align*}
	(\al_1,\infty) \subseteq \left \{ a > 0 : F(r_0,a) = n\pi \mbox{ for some } r_0 > 0 \right \}. 
	\end{align*}
	
	The idea of the proof is now simple.  If the initial velocity $\al$ is large enough, then $F(r,\al) = n\pi$ for some $r > 0$ so that
	$\al \in A^c$.  To make this argument precise, we need the precise asymptotics of $F(r,\al)$ for $r$ near $r = 0$.  First we change variables and set 
	$x = \arcsinh r$.  Then $F(x,\al) := F(r(x),\al)$ satisfies $F(0,\al) = n\pi/2$, $F'(0,\al) = \al$ and 
	\begin{align}
	F'' + \tanh x F' - \sin 2F = 0. \label{clm21} 
	\end{align}
	
	We first claim there exists $x_0 > 0$ small such that for all $\al > 0$
	\begin{align}
	\| F(\cdot,\al) \|_{C^1([0,x_0])} \leq n \pi + 4\al. \label{clm2e1}
	\end{align}
	Indeed, we solve \eqref{clm21} near $x = 0$ by a contraction mapping argument.  Let $X = C^1([0,x_0])$ where $x_0$ is to be 
	chosen later.  Define $\Phi: X \rightarrow X$ by 
	\begin{align*}
	\Phi F(x) = \frac{n \pi}{2} + \al x + \int_0^x (x - y) \left [   \sin 2F(y) - \tanh y F'(y) \right ] dy
	\end{align*}
	If $x_0$ is chosen so small so that $\tanh y \leq 2 y$ for $y \in [0,x_0]$, then it is easy to verify that for all $F,G \in X$
	and for some absolute constant $C > 0$
	\begin{align*}
	\| \Phi F \|_X &\leq \frac{n\pi}{2} + 2\al + C x_0 \| F \|_X, \\
	\| \Phi F - \Phi G \|_X &\leq C x_0 \| F - G \|_X.  
	\end{align*}
	Now fix $x_0$ smaller if necessary so that $x_0 < 1/(8C)$.  Then, we may contract in the ball $B_X(0,n\pi + 4\al)$ and find a unique
	fixed point (namely $F(x,\al)$) of $\Phi$.  This shows that 
	there exists $x_0$ small and independent of $\al$ such that $\| F(\cdot,\al) \|_{C^1([0,x_0])} \leq n \pi + 4\al$ as desired.

	We now conclude that if $\al$ is sufficiently large (depending on $x_0$), then in fact $F(x_0,\al) \geq n\pi$ where $x_0$ was defined previously.  
	We write for $x \in [0,x_0]$
	\begin{align*}
	F(x,\al) = \frac{n\pi}{2} + \al x + \int_0^x (x - y) \left [ \sin 2F(y,\al) - \tanh y F'(y,\al) \right ] dy.
	\end{align*}
	Then by \eqref{clm2e1}, for some constant $C > 0$ and by choosing $x_0$ smaller if necessary, we have  
	\begin{align*}
	|F(x_0,\al)| &\geq \frac{n\pi}{2} + \al x_0 - C x_0^2 \| F(\cdot,\al) \|_{C^1([0,x_0])} \\
	&\geq \frac{n\pi}{2} (1 - Cx_0^2) + \al x_0 ( 1- 4x_0 C) \\
	&\geq \frac{n\pi}{4} + \frac{\al x_0}{2}.  
	\end{align*}
	This shows that for all $\al \geq 2n\pi/x_0$, $F(x_0,\al) \geq n \pi$, i.e. $\al \in \left \{ a > 0 : F(r_0,a) = n\pi \mbox{ for some } r_0 > 0 \right \}
	\subset A^c$. 
\end{proof}

\begin{proof}[Proof of Lemma \ref{odethm1}]
	By Claim 1 and Claim 3, 
	\begin{align*}
	\al^* := \sup A \in (0,\infty).
	\end{align*}
	By Claim 2, $\al^* \notin A$.  Suppose that $\al^* \in \left \{ \al \in (0,\infty) : F(\infty,\alpha) > n\pi \right \}$.  Then by continuous
	dependence of initial data, all $\al$ near $\al^*$ are also in $\left \{ \al \in (0,\infty) : F(\infty,\alpha) > n\pi \right \}$.
	This, however, contradicts the facts that $\al^* = \sup A$ and that $A$ is open (by Claim 2).  Thus, $F(\infty, \al^*) = n\pi$, 
	and we are done. 
\end{proof}

\subsection{Uniqueness of the Harmonic Map}
In this section we show uniqueness of the harmonic map constructed in the previous section which concludes the proof of 
Proposition \ref{harm}.

\begin{lem}\label{odethm2}
	Let $F_1$ and $F_2$ solve \eqref{ode} and assume that for $j = 1,2$
	\begin{align*}
	F_j(-\infty) = 0, \quad
	F_j(\infty) = n\pi.
	\end{align*}
	Then $F_1 = F_2$. 
\end{lem}

\begin{proof}
	Since any $F$ that solves \eqref{ode} and connects $0$ to $n\pi$ must be increasing, we may make a change of variables and consider 
	$F$ as the dependent variable and $p = \frac{dF}{dx}$ as the dependent variable, where $x = \arcsinh r$.  Thus, the equation solved
	by $p$ is 
	\begin{align}
	p \frac{dp}{dF} + (\tanh x) p - \sin 2F = 0. 
	\end{align}
	Suppose towards a contradiction, that we have two different solutions $F_1, F_2$.  These determine two $C^\infty$ diffeomorphisms
	$x_1,x_2: (0,n\pi) \rightarrow (-\infty,\infty)$ by the condition $F_j \circ x_j$ is the identity on $(0,n\pi)$.  Then we have
	\begin{align}
	p_j(F) \frac{dp_j}{dF} + (\tanh x_j(F)) p_j(F) - \sin 2F = 0, \quad j = 1,2.  
	\end{align}
	Set $\phi(F) = p_2(F) - p_1(F)$.  Subtracting the equation satisfied by $p_1$ from the equation satisfied by $p_2$
	and rearranging, we have
	\begin{align*}
	0 &= p_2 \frac{dp_2}{dF} - p_1 \frac{dp_1}{dF} + \tanh x_2 p_2 - \tanh x_1 p_1 \\
	&= p_2 \frac{d \phi}{dF} + \left ( \frac{dp_1}{dF} + \tanh x_2 \right ) \phi - \left ( \tanh x_1 - \tanh x_2 \right ) p_1.  
	\end{align*}
	Define $q = p_2^{-1} \left ( \frac{dp_1}{dF} + \tanh x_2 \right )$, $f = \left ( \tanh x_2 - \tanh x_1 \right ) p_1 p_2^{-1}$. 
	Then $\phi$ satisfies  
	\begin{align*}
	\phi' + q \phi = -f \implies 
	(-\phi e^{-Q} )' = f, 
	\end{align*}
	where $Q(F) = \int_F^{F_0} q(\bar F) d\bar F$ for any choice of $F_0 \in (0,n\pi)$.  Hence, we have that 
	\begin{align}
	\phi(F) = e^{Q(F)} \phi(F_0) + \int_F^{F_0} e^{Q(F)-Q(\bar F)} f(\bar F) d\bar F. \label{odethm21}
	\end{align}  We now make an observation based on \eqref{odethm21}. Note that if $p_2(F_0) > p_1(F_0)$ and $x_2(F_0) > x_2(F)$ imply that $p_2(F) > p_1(F)$ and $x_2(F) > x_1(F)$ for all $F \leq F_0$.
	Indeed, suppose $F_1 < F_0$ and $p_2(F) \geq p_1(F)$ for all $F_1 \leq F \leq F_0$.  Then from the definition of $p_j$, we have for 
	all $F_1 \leq F \leq F_0$ 
	\begin{align*}
	p_2(F) \geq p_1(F) &\implies (x_2(F) - x_1(F))' \leq 0.
	\end{align*}
	This implies upon integrating that 
	\begin{align*}
	0 < x_2(F_0) - x_1(F_0) \leq x_2(F) - x_1(F), \quad F_1 \leq F \leq F_0. 
	\end{align*}
	Since $\tanh x$ is increasing on $(-\infty,\infty)$, 
	\begin{align*}
	x_2(F) > x_1(F), \quad F_1 \leq F \leq F_0 \implies f(F) > 0, \quad F_1 \leq F \leq F_0. 
	\end{align*}
	Hence by \eqref{odethm21}
	\begin{align*}
	p_2(F) > p_1(F), \quad F_1 \leq F \leq F_0.  
	\end{align*}
	Thus, if $p_2(F) \geq p_1(F)$ for all $F_1 \leq F \leq F_0$, we must in fact have the strict inequalities $p_2(F) > p_1(F)$ and $x_2(F) > x_1(F)$
	for $F_1 \leq F \leq F_0$. By continuity, we see that $p_2(F) > p_1(F)$ and $x_2(F) > x_1(F)$ for all $F \leq F_0$.
	
	By Lemma \ref{odeprop2}, a solution $F$ to \eqref{ode} such that $F(-\infty) = 0$, $F(\infty) = n\pi$ 
	satisfies for unique $a, b > 0$, 
	\begin{align*}
	F(x) - n\pi &\sim - ae^{-2x} + a\frac{2}{5} e^{-4x}, \quad \mbox{as } x \rightarrow \infty, \\
	F(x) &\sim  be^{2x} - b\frac{2}{5} e^{4x}, \quad \mbox{as } x \rightarrow -\infty.
	\end{align*}
	It follows that $p$ satisfies 
	\begin{align*}
	p &\sim 2ae^{-2x} - a\frac{8}{5} e^{-4x} \\
	&\sim 2 ( n\pi - F) - a\frac{4}{5} e^{-4x} \\
	&\sim 2 (n\pi - F) - a^{-1} \frac{4}{5} (n\pi - F)^{2},
	\end{align*}
	as $F \rightarrow n\pi^-$.  Similarly, we have
	\begin{align*}
	p &\sim 2 F - b^{-1} \frac{4}{5} F^{2}, 
	\end{align*}
	as $F \rightarrow 0^+$. Suppose $F_2$ has coefficients $a_2,b_2 > 0$ and $F_1$ has coefficients $a_1,b_1 > 0$ where (without loss of generality)
	$a_2 > a_1$.  Then clearly $x_2(F) > x_1(F)$ for all $F$ sufficiently close to $n\pi$ since for $x$ large 
	\begin{align*}
	F_2(x) \sim n\pi - a_2 e^{-2x} < n \pi - a_1 e^{-2x} \sim F_1(x). 
	\end{align*}
	Moreover, we have $p_2(F) > p_1(F)$ for $F$ sufficiently close to $n\pi$ by our previous calculation
	\begin{align*}
	p_2(F) &\sim 2(n\pi - F) - a_2^{-1}\frac{4}{5} (n\pi - F)^{2} \\&> 
	2(n\pi - F) - a_1^{-1}\frac{4}{5} (n\pi - F)^{2} \sim p_1(F). 
	\end{align*}
	Thus, by our observation following \eqref{odethm21},
	we have $p_2(F) > p_1(F)$ and $x_2(F) > x_1(F)$ for all $F \in (0,n\pi)$.  In particular, the constraint $x_2(F) > x_1(F)$ for
	all $F \in (0,n\pi)$ 
	implies that $b_1 > b_2$. But then for $F$ near 0
	\begin{align*}
	p_1(F) &\sim 2 F - b_1^{-1} \frac{4}{5} F^{2} \\& > 2 F - b_2^{-1} \frac{4}{5} F^{2} \sim p_2(F),  
	\end{align*}
	which contradicts $p_2(F) > p_1(F)$ for all $F \in (0,n\pi)$. Thus, no two distinct solutions $F_1,F_2$ exist.
	This completes the proof. 
\end{proof}




\section{Strichartz Estimates for the Free Wave Equation on Wormholes}

In this section we establish Strichartz estimates for radial solutions to the free 
wave equation on the $(d+1)$--dimensional wormhole $\mathcal M^{d+1} = \{ (r,\omega) : r \in \R, \omega \in \s^d \}$ with metric $g$ satisfying
\begin{align*}
ds^2 = dr^2 + (r^2 + 1) d\Omega_{\s^d}^2(\omega).
\end{align*}
Here $d\Omega_{\s^d}^2$ is the line element on $\s^d$ corresponding to the usual round metric.  When we say radial functions we mean 
functions $f : \mathcal M^{d+1} \rightarrow \R$ with $f = f(r)$. These Strichartz estimates will be used in 
Section 4 and Section 5 to establish a small data theory for \eqref{s04}.  However, the results and methods of this section are independent of 
all other sections in this work and may be of interest in their own right.

For the remainder of the section, we fix $d \geq 2$ and 
drop the superscript by writing $\M$ instead of $\M^{d+1}$. We denote $\h(\R; (r^2+1)^{d/2} dr)$ simply by $\h$.
For an interval $I$, we denote the spatial norms on $\M$ and spacetime norms on $I \times \M$ by 
\begin{align*}
\| f \|_{L^p} &:= \left ( \int |f(r)|^p (r^2 + 1)^{d/2} dr \right )^{1/p}, \\
\| u \|_{L_t^p L^q_x(I)} &:= \left ( \int_\R \left ( \int_\R |u(t,r)|^q (r^2 + 1)^{d/2} dr \right )^{p/q} dt \right )^{1/p}.
\end{align*}
Since we only consider radial functions on $\M$, we abuse notation 
slightly and let $\Delta_g$ denote the radial part of the Laplace operator on $\M$, 
\begin{align*}
\Delta_g f(r) = \p_r^2 f + \frac{dr}{r^2 + 1} \p_r f.  
\end{align*}

Let $I$ be an interval with $0 \in I$.  Let $F : I \times \R \rar \R$, and let $u = u(t,r)$ solve the inhomogeneous wave 
equation 
\begin{align} 
\begin{split}\label{inwave}
&\partial_t^2 u - \Delta_{g} u = F, \quad (t,r) \in I \times \R. \\
&\vec u(0) = (u_0,u_1) \in \h. 
\end{split}
\end{align}
We say that a triple $(p,q, \gamma)$ is \emph{admissible} if  
\begin{align*}
p > 2, q \geq 2, \quad 
\frac{1}{p} + \frac{d+1}{q} = \frac{d+1}{2} - \gamma, \quad 
\frac{1}{p} \leq \frac{d}{2} \left ( \frac{1}{2} - \frac{1}{q} \right ).
\end{align*} 

The main result of this section is the following family of Strichartz estimates for \eqref{inwave}. 
\begin{ppn}\label{strm}
	Let $(p,q,\gamma)$ and $(a,b,\rho)$ be admissible triples.  Then any solution $u$ to \eqref{inwave} satisfies 
	\begin{align*}
	\| |\nabla|^{1-\gamma} u \|_{L^p_tL^q_x(I)} + 
	\| |\nabla|^{-\gamma} \p_t u \|_{L^p_tL^q_x(I)}
	\lesssim \| \vec u(0) \|_{\h} + \| |\nabla|^{\rho} F \|_{L^{a'}_t L^{b'}_x(I)}. 
	\end{align*}
\end{ppn}

It is well known (see for example \cite{keeltao} \cite{sogge} \cite{tao}) that by a standard argument using Littlewood--Paley theory (for our wormhole geometry see \cite{zhang}) and $TT^*$ arguments, 
establishing Proposition \ref{strm} can be reduced to proving the following 
frequency localized dispersive estimate:
let $E$ denote the spectral measure for $-\Delta_g$ (restricted to radial functions).  For a standard Littlewood--Paley cutoff 
$\varphi \in C^{\infty}_0(\R)$ with
support in $(1/2,2)$, define (via the functional calculus) 
\begin{align*}
\varphi \left (2^{-j} \sqrt{-\Delta_g} \right ) = \int_0^\infty \varphi(2^{-j}\sqrt{\lambda} ) E(d\lambda). 
\end{align*}
Then for all $f \in C^\infty_0(\R)$, 
\begin{align}\label{disp}
\left \| e^{\pm it\sqrt{-\Delta_{g}}} \varphi \left (2^{-j} \sqrt{-\Delta_{g}} \right ) f \right \|_{L^\infty} \lesssim 
2^{\frac{d+2}{2}} (2^{-j} + |t|)^{-\frac{d}{2}} \| f \|_{L^1}. 
\end{align}

The proof of \eqref{disp} draws heavily from the works \cite{sss1} \cite{sss2}.  In these works, the authors prove 
dispersive estimates for free waves on a manifold with metric of the form 
\begin{align*}
ds^2 = dr^2 + R^2(r) ds_{\Omega}^2(\omega), \quad r \in \R,
\end{align*}
where $ds_{\Omega}^2(\omega)$ is the metric on a compact embedded Riemannian manifold $\Omega \subset \R^N$ with 
dimension $d \geq 1$.  The function $R(r)$ is assumed to be asymptotically conic: 
\begin{align*}
R(r) = |r|\left ( 1 + O(r^{-1}) \right ), \quad \mbox{as } r \rar \pm \infty.
\end{align*}
Note that in the case of the wormhole geometry, $\Omega = \s^d$ and $R(r) = \la r \ra$.  In particular, the authors proved
weighted $L^1 \rar L^\infty$ type estimates for data of the form $f(r)Y_n(\omega)$ where $Y_n$ are eigenfunctions of $-\Delta_{\Omega}$.
For the $n = 0$ case (i.e. a radial solution), they established the dispersive estimate 
\begin{align*}
\left \| e^{\pm it \sqrt{-\Delta_g}} f(r) \right \|_{L^\infty} \lesssim |t|^{-d/2} \left ( \| f \|_{L^1} + \| f' \|_{L^1} \right ).
\end{align*}
In our proof of \ref{disp}, we refine their methods for the case of frequency localized data.   

In what follows, we use the standard Japanese bracket notation $\la r \ra = (r^2 + 1)^{1/2}$. One readily verifies that 
\begin{align*}
-\Delta_{g} f(r) = \left ( \la r \ra^{-d/2} H  \la r \ra^{d/2} \right ) f(r), 
\end{align*}
where $H$ is the Schrodinger operator on $\R$ given by 
\begin{align*}
H = -\frac{d^2}{dr^2} + V, \quad V(r) = \frac{d(d-4)}{4} r^2 \la r \ra^{-4} + \frac{d}{2} \la r \ra^{-2} .
\end{align*}
Note that the potential $V$ satisfies
\begin{align*}
V(r) = \frac{d(d-2)}{2r^2} + O(r^{-3}),
\end{align*}
as $r \rar \pm \infty$ with natural derivative bounds.  We denote the following resolvents $R(z) = (-\Delta_{g} - z)^{-1}$ and $R_H(z) = (H - z)^{-1}$ for $z \notin \sigma(-\Delta_{g})
= \sigma(H) = 
[0,\infty)$.  We note that the decay of $V$ implies that the spectrum of $H$ in $(0,\infty)$ is purely absolutely continuous (in fact, absolute continuity follows from the following explicit formula for the spectral measure). 

Via Stone's theorem, we can write (as an identity of Schwartz kernels) 
\begin{align*}
E(d\lambda^2)(r,\rho) &= \frac{\lambda}{\pi i} \lim_{\epsilon \rightarrow 0^+}
(R(\lambda^2 + i \epsilon) - R(\lambda^2 - i\epsilon))(r,\rho) d \lambda \\
&= \frac{\lambda}{\pi i} \lim_{\epsilon \rightarrow 0^+}
\la r \ra^{-d/2} (R_H(\lambda^2 + i \epsilon) - R_H(\lambda^2 - i\epsilon))(r,\rho) \la \rho \ra^{d/2} d\lambda.
\end{align*}
The final limit may be evaluated \lq explicitly' by using the fact that 
\begin{align*}
\lim_{\epsilon \rightarrow 0^+} \frac{1}{2\pi i}(R_H(\lambda^2 + i \epsilon) - R_H(\lambda^2 - i\epsilon))(r,\rho)
=& \Im \left [ \frac{f_{+}(r,\lambda) f_{-}(\rho,\lambda)}{W(\lambda)} \right ] \chi_{[r > \rho]} \\
&+ \Im \left [ \frac{f_{-}(r,\lambda) f_{+}(\rho,\lambda)}{W(\lambda)} \right ] \chi_{[r < \rho]}, 
\end{align*}
where $f_{\pm}(\cdot, \lambda)$ are the \emph{Jost solutions} which satisfy 
\begin{align*}
H f_{\pm}(r,\lam) = \lambda^2 f_{\pm}(r, \lambda), \\
f_{\pm} (r, \lambda) \sim e^{\pm i r \lambda} \quad \mbox{as } r \rightarrow \pm \infty,
\end{align*}
and 
\begin{align*}
W(\lambda) = W(f_-(\cdot, \lambda), f_+(\cdot, \lambda)) = f_+'(\cdot, \lambda) f_-(\cdot, \lambda) - 
f_+(\cdot, \lambda) f'_-(\cdot, \lambda),
\end{align*}
is their Wronskian. It is easy to see via a standard contraction argument that $f_{\pm}(\cdot,\lambda)$ exist provided 
$V \in L^1(\R)$.  In summary, we see that the spectral measure for $-\Delta_{g}$ satisfies
\begin{align*}
E(d\lambda^2)(r,\rho)=  2 \lambda \la r \ra^{-d/2}
\left \{ \Im \left [ \frac{f_{+}(r,\lambda) f_{-}(\rho,\lambda)}{W_\nu(\lambda)} \right ] \chi_{[r > \rho]} 
+ \Im \left [ \frac{f_{-}(r,\lambda) f_{+}(\rho,\lambda)}{W_\nu(\lambda)} \right ] \chi_{[r < \rho]} \right \}
\la \rho \ra^{d/2}
d\lambda.
\end{align*}
Therefore, the estimate \eqref{disp} (and thus, Proposition \ref{strm}) reduces to proving the following oscillatory integral estimate uniformly in $r > \rho$
(the case $r < \rho$ is analagous) which we state as a proposition.
\begin{ppn}\label{oscint}
	For all $\rho < r$ and $t \in \R$ we have the estimate
	\begin{align*}
	\left | 
	\int_0^{\infty} e^{\pm i t \lambda} \varphi(2^{-j}\lambda) \lambda
	\Im \left [ \frac{f_{+}(r,\lambda) f_{-}(\rho,\lambda)}{W(\lambda)} \right ] d\lambda 
	\right |
	\lesssim (\la r \ra \la \rho \ra)^{d/2} 2^{j(d+2)/2}(2^{-j} + |t|)^{-d/2}.
	\end{align*}
	The implied constant depends only on $\varphi$ and $d$. 
\end{ppn}

Note that we absorbed the volume form $(r^2 + 1)^{d/2} dr$ implicit in the right hand side of \eqref{disp} into the left hand side
in order to conclude that proving the estimate \ref{disp} reduces to proving Proposition \ref{oscint}. To prove Proposition \ref{oscint}, we will need asymptotics for $f_{\pm}(\cdot,\lambda)$
and $W(\lambda)$ for $\lambda$ small. The asymptotics that we require are contained in the following subsection.

\subsection{Scattering Theory for Schrodinger Operators}

In this section, we briefly summarize the scattering theory developed in Section 3 of \cite{sss2}
for the Schr\"odinger operator $H = -\frac{d^2}{dr^2} + V$ on $\R$ where
$V \in C^{\infty}(\R)$ is real--valued and such that 
\begin{align*}
V(r) = \frac{d(d-2)}{4} r^{-2} + U(r), \quad U \in C^{\infty}(\R \backslash \{0\}),
\end{align*}
with 
\begin{align*}
| U^{(k)}(r) | \leq C_k |r|^{-3-k}, \quad |r| \geq 1.
\end{align*}
In particular, we summarize the asymptotics for $f_{\pm}(\cdot,\lambda)$ and $W(\lambda)$ as $\lambda \rightarrow 0$ under a condition
on the point spectrum of $H$.  This condition will be elaborated on below.  In what follows, we assume, as before, that $d \geq 2$.  

First, solutions to the zero energy equation with slow decay at $\pm \infty$ were constructed. 

\begin{lem}[Lemma 3.2 \cite{sss2}]\label{scattlem1}
	For $j = 0,1$, there exist real--valued solutions $u_{j}^{\pm}(\cdot)$ to the zero energy equation
	\begin{align*}
	- u_{j}^{\pm}(r)'' + V(r)u_{j}^{\pm}(r) = 0, \quad r \in \R,
	\end{align*}
	such that $W(u_{0}^{\pm}(\cdot), u_{1}^{\pm}(\cdot)) = $constant, and $u_j^{\pm}$ have the asymptotics
	\begin{align*}
	u_{0}^{\pm}(r) &= |r|^{d/2}( 1 + O ( |r|^{-1}) ), \quad \mbox{as } r \rightarrow \pm \infty,  \\
	u_{1}^{\pm}(r) &= |r|^{-(d-2)/2} ( 1 + O(|r|^{-1}) ), \quad \mbox{as } r \rightarrow \pm \infty. 
	\end{align*}
	The $O(\cdot)$ terms behave like symbols under differentiation in $r$. 
\end{lem}

\begin{defn}\label{scattdef}
	We say that the Schr\"odinger operator $H$ has 0 as a resonance if 
	\begin{align*}
	W(u_1^+(\cdot), u_1^-(\cdot)) = 0,
	\end{align*}
	where $u_1^{\pm}(\cdot)$ are the solutions constructed in Lemma \ref{scattlem1}. This condition
	is equivalent to the existence of a nonzero solution $f$ to $-f'' + Vf = 0$ such that
	$f$ is asymptotic to $|r|^{-(d-2)/2}$ at $\pm \infty$.
\end{defn}

The previously mentioned condition on the point spectrum of $H$ is that 0 is not a resonance. Next,  
perturbing in small $\lambda$, for $j = 0,1$ a basis of real--valued solutions $u_{j}^{\pm}(\cdot,\lambda)$ to 
\begin{align*}
- u_{j}^{\pm}(r,\lambda)'' + V(r)u_{j}^{\pm}(r,\lambda) = \lambda^2 u_{j}^{\pm}(r,\lambda), \quad r \in \R,
\end{align*}
was constructed which are well approximated by $u_j^{\pm}$ when $|r\lam| \ll 1$.

\begin{lem}[Corollary 3.5 \cite{sss2}]\label{scattlem2}
	Let $u_j^+(\cdot)$ be as in Lemma \ref{scattlem1}.  There exist solutions
	$u_j^+(\cdot,\lambda)$ of $H f = \lambda^2 f$ with 
	\begin{align*}
	W(u_1^+(\cdot,\lambda),u_0^+(\cdot,\lambda)) = 1,
	\end{align*}
	such that for $j = 0,1$ and $r_0 \leq r \ll \lambda^{-1}$, we have 
	\begin{align*}
	u_j^+(r,\lambda) = u_j^+(r)(1 + a_j^+(r,\lambda)).
	\end{align*}
	The functions $a_j^{+}(\cdot,\lambda)$ satisfy the bounds
	\begin{align*}
	\left | 
	\partial_r^l \partial_\lambda^k a^+_j(r,\lambda)
	\right | \lesssim_{k,l}
	\begin{cases}
	\lambda^{2-k} \la r \ra^{2 - l}\log|\lambda r| &\mbox{if } d = 2, \\
	\lambda^{2-k} \la r \ra^{2 - l}&\mbox{if } d > 2.
	\end{cases}
	\end{align*}
	A similar statement holds with $u^+_0(\cdot,\lam)$ replaced by $u^-_0(\cdot,\lam)$ for $r \leq 0$. 
\end{lem}

In what follows, $\beta_d = 
\sqrt{\frac{\pi}{2}} e^{id\pi/4}$. The outgoing Jost solution for 
$$H_0 = -\frac{d^2}{dr^2}
+ \frac{d(d-2)}{2r^{2}}$$ is known explicitly.  In particular, we have that the solution to $H_0 f_0(\cdot,\lam) = \lam^2 f_0(\cdot,\lam)$ with 
$f_0(r,\lam) \sim e^{i\lam r}$ as $r \rar \infty$ is given by 
\begin{align*}
f_{0}(r,\lam) = \beta_d \sqrt{r\lam} H^{+}_{(d-1)/2}(r\lam),  
\end{align*}
where $H^+_{(d-1)/2}(z) = J_{(d-1)/2}(z) + i Y_{(d-1)/2}(z)$ is the Hankel function.  Perturbing off of this explicit solution,
we obtain the following asymptotic form for the Jost function $f_+(\cdot,\lam)$.  Similar asymptotics hold for $f_-(\cdot,\lam)$. 

\begin{lem}[Corollary 3.10 \cite{sss2}]\label{jostlem}
	For $\lam \neq 0$, $\lam \ll 1$, and in the range $1 \ll r \ll \lam^{-1}$, we have 
	\begin{align*}
	f_+(r,\lam) &= \beta_d \sqrt{\lam r} \left [ J_{(d-1)/2}(r\lam)(1 + O(\lam))(1 + O(r^{-1})) + 
	Y_{(d-1)/2}(r\lam)O(\lam)(1 + O(r^{-1})) \right ]\\
	&\:+  i \beta_d \sqrt{\lam r} \left [ Y_{(d-1)/2}(r\lam)(1 + O(\lam))(1 + O(r^{-1})) + 
	J_{(d-1)/2}(r\lam)O(\lam)(1 + O(r^{-1})) \right ].
	\end{align*}
	In the range $r\lam \gtrsim 1$, we have 
	\begin{align*}
	f_+(r,\lam) = e^{ir\lam}m_+(r,\lam),
	\end{align*}
	where 
	\begin{align*}
	m_+(r,\lam) = 1 + O_{\C}(r^{-1}\lam^{-1})
	\end{align*}
	The $O(\cdot)$ terms are real--valued, the $O_\C(\cdot)$ term is complex--valued, 
	and all terms obey the natural bounds with respect to differentiation in $\lam$ and $r$.
\end{lem}

Using the previous lemmas, the following expansions were obtained. 

\begin{lem}[Corollary 3.6 and Proposition 3.12 \cite{sss2}]\label{scattlem3}
	We have the expansions
	\begin{align*}
	f_{\pm}(r,\lambda) = a_{\pm}(\lambda) u_{0}^{\pm}(r,\lambda) + 
	b_{\pm}(\lambda) u_{1}^{\pm}(r,\lambda),
	\end{align*}
	where the coefficients satisfy with some small $\epsilon > 0$ depending on $d$ and with some real constants $\alpha_0^{\pm}, \beta_0^{\pm}$,
	\begin{align*}
	a_{\pm}(\lambda) &= \lambda^{d/2} \beta_d \left ( \alpha_{0}^{\pm} + O(\lambda^{\epsilon}) + i O(\lambda^{-(d-2)\epsilon}) \right ), \\
	b_{\pm}(\lambda) &=  i \lambda^{-(d-2)/2} \beta_d \left (\beta_0^{\pm} + O(\lambda^{\epsilon}) + i O(\lambda^{d\epsilon}) \right ).
	\end{align*}
	The $O(\cdot)$ terms are real--valued and satisfy the natural derivative bounds.
\end{lem}

Using the expansions in Lemma \ref{scattlem3}, an asymptotic expansion for $W(\lambda)$ for small $\lambda$ under the nonresonant 
condition was obtained.

\begin{lem}[Corollary 3.13]\label{scattlem4}
	If 0 is not a resonance for $H$, then for all $0 < \e < \e_0(d)$, we have 
	\begin{align}
	W(\lambda)
	&= i e^{i\pi(d-1)/2} \lambda^{-(d-2)} ( W_0 + O_{\C}(\lambda^\epsilon)). \label{scatt2}
	\end{align}
	Here $W_0$ is a nonzero real constant and $O_{\C}(\lambda^{\epsilon})$ is complex valued, and all terms
	satisfy the natural derivative bounds.  We remark 
	that the nonresonant condition is what guarantees that the constant $W_0$ is nonzero.  
\end{lem}

Finally, the following asymptotic expansion for the spectral measure corresponding to $H$ for small $\lambda$ was obtained.    

\begin{lem}[Corollary 5.1 \cite{sss2}] \label{lem1}
	If 0 is not a resonance for $H$, then for $0 < \lambda \ll 1$ and any $r,\rho \in \R$,
	\begin{align*}
	\Im \left [ \frac{f_{+}(r,\lambda) f_{-}(\rho,\lambda)}{W(\lambda)} \right ] = O(\lambda^{d-1}) u^+_0(r, \lambda) u^-_1(\rho,\lambda) \\
	+ O(\lambda^{d-1}) u^+_1(r, \lambda) u^-_0(\rho,\lambda) + O(\lambda^{d-1}) u^+_0(r, \lambda) u^-_0(\rho,\lambda) \\
	+ O(\lambda^{d-1}) u^+_1(r, \lambda) u^-_1(\rho,\lambda),
	\end{align*}
	where the $O(\cdot)$ terms are real--valued and satisfy the natural derivative bounds. 
\end{lem}

We now turn to proving the oscillatory integral estimate 
Proposition \ref{oscint}. 

\subsection{Proof of Proposition \ref{oscint}}

We recall that we wish to prove the oscillatory integral estimate
\begin{align*}
\left | 
\int_0^{\infty} e^{\pm i t \lambda} \varphi(2^{-j}\lambda) \lambda
\Im \left [ \frac{f_{+}(r,\lambda) f_{-}(\rho,\lambda)}{W(\lambda)} \right ] d\lambda 
\right |
\lesssim (\la r \ra \la \rho \ra)^{d/2} 2^{j(d+2)/2}(2^{-j} + |t|)^{-d/2},
\end{align*}
for all $r > \rho$ and $t \in \R$.  Here $H$ is the Schr\"odinger operator on $\R$
\begin{align*}
H = -\frac{d^2}{dr^2} + V, \quad V(r) = \frac{d(d-4)}{4} r^2 \la r \ra^{-4} + \frac{d}{2} \la r \ra^{-2},
\end{align*}
and $f_{\pm}(\cdot,\lambda)$ are the Jost functions associated to $H$.  We distinguish the cases $j \ll 0$ and 
$j \gtrsim 0$.  The case $j \ll 0$ will rely heavily on the scattering theory summarized in the previous subsection.  

We first consider the case $j \ll 0$ so that the integrand in the oscillatory integral is 
localized to small $\lambda$.  We first claim that $H$ is nonresonant so that the results summarized in the previous section apply.  Indeed, if 0 is 
a resonance of $H$, then there exists a nonzero function $f$ such that $Hf = 0$ and $f(r) = O(\la r \ra^{-(d-2)/2})$ as $|r| \rightarrow
\infty$.  This implies by the relation $-\Delta_{g} ( \la r \ra^{-d/2} f ) =  \la r \ra^{-d/2} H f$ that
there exists a nonzero function $u$ such that $\Delta_g u = 0$ and $u(r) = O(\la r \ra^{-(d-1)})$ as $|r| \rightarrow \infty$. 
Since $d \geq 2$, the maximum principle on $\M$ implies that $u \equiv 0$, a contradiction.  Thus, 0 is not
a resonance of the Schrodinger operator $H$.   

The proof of Proposition \ref{oscint} for $j \ll 0$ is split up into several lemmas.  In what follows, we differentiate between the oscillatory regime and the exponential regime for the Jost solutions
$f_{\pm}(\cdot,\lambda)$.  This transition occurs at $|r\lambda| = 1$. Let $\chi \in C^{\infty}_0(\R)$ be even with 
$\chi(r) = 1$ for $|r| \leq 1$ and $\supp \chi \subset \{ |r| < 2 \}$. We denote the smooth cutoff
$\chi(r\lambda)$ by $\chi_{[|r\lambda| < 1]}$ and the smooth cutoff $(1-\chi(r\lambda))$ by $\chi_{[|r\lambda| > 1]}$.

\begin{lem}\label{lem2}
	For all $t \in \R$ and $r,\rho \in \R$,  
	\begin{align}
	\left | 
	\int_0^{\infty} e^{\pm i t \lambda} \lambda \varphi(2^{-j}\lambda) \chi_{[|r\lambda|< 1]} \chi_{[|\rho\lambda|<1]}(\la r \ra \la \rho \ra)^{-d/2} 
	\Im \left [ \frac{f_{+}(r,\lambda) f_{-}(\rho,\lambda)}{W(\lambda)} \right ] d\lambda 
	\right |
	\lesssim 2^{j(d+2)/2}(2^{-j} + |t|)^{-d/2}. \label{lem21} 
	\end{align}
\end{lem}

\begin{proof}
	By Lemma \ref{lem1} we may write
	\begin{align}
	\Im \left [ \frac{f_{+}(r,\lambda) f_{-}(\rho,\lambda)}{W(\lambda)} \right ] = O(\lambda^{d-1})O\left ((\la r \ra \la \rho \ra)^{d/2} \right ), \label{lem22} 
	\end{align}
	where the $O(\cdot)$ terms satisfy natural derivative bounds.  We write \eqref{lem21} as
	\begin{align}
	\left | \int_0^{\infty} e^{\pm it \lambda} a_j(r,\rho,\lambda) d\lambda \right | \lesssim 2^{j(d+2)/2}(2^{-j} + |t|)^{-d/2},
	\end{align}
	where 
	\begin{align*}
	a_j(r,\rho,\lambda) = \lambda \varphi(2^{-j}\lambda)  \chi_{[|r\lambda|< 1]}  \chi_{[|\rho \lambda|< 1]}(\la r \ra \la \rho \ra)^{-d/2} 
	\Im \left [ \frac{f_{+}(r,\lambda) f_{-}(\rho,\lambda)}{W(\lambda)} \right ].
	\end{align*}
	By \eqref{lem22} the function $a_j(r,\rho,\lambda)$ satisfies 
	\begin{align}
	a_j(r,\rho,\lambda) = \varphi(2^{-j}\lambda) O(\lambda^{d}), \label{lem23}
	\end{align}
	with natural derivative bounds.  
	
	First note that if $|t| \leq 2^{-j}$, then by \eqref{lem23} 
	\begin{align*}
	\left | \int_0^{\infty} e^{\pm it \lambda} a_j(r,\rho,\lambda) d\lambda \right | 
	\lesssim \int_{[\lambda \simeq 2^j]} \lambda^d d\lambda 
	\lesssim 2^{j(d+1)} \lesssim 2^{j(d+2)/2} (2^{-j} + |t|)^{-d/2}. 
	\end{align*}
	We now assume that  $|t| \geq 2^{-j}$.  Integration by parts $d$ times and \eqref{lem23} yield 
	\begin{align*}
	\left | \int_0^{\infty} e^{\pm it \lambda} a_j(r,\rho,\lambda) d\lambda \right | &=
	|t|^{-d} \left | \int_0^{\infty} e^{\pm it \lambda} \partial_\lambda^d a_j(r,\rho,\lambda) d\lambda \right | \\
	&\lesssim |t|^{-d} \int_{[\lambda \sim 2^j]} d\lambda \\
	&\lesssim |t|^{-d} 2^j \\
	&\lesssim 2^{j(d+2)/2} (2^{-j} + |t|)^{-d/2}.  
	\end{align*}
	This concludes the proof. 
\end{proof}

We now consider the case when the integrand is supported in $|r\lambda| > 1$ and $|\rho\lambda| > 1$.  With the convention that
$f_{\pm}(\cdot, -\lambda) = \overline{f_{\pm}(\cdot, \lambda)}$, we remove the taking of an imaginary part in the integrand and write 
\begin{align*} 
\int_0^{\infty}& e^{\pm i t \lambda} \lambda \varphi(2^{-j}\lambda) \chi(\lambda r) \chi(\lambda \rho)(\la r \ra \la \rho \ra)^{-d/2} 
\Im \left [ \frac{f_{+}(r,\lambda) f_{-}(\rho,\lambda)}{W(\lambda)} \right ] d\lambda \\
&= \int_{-\infty}^{\infty} e^{\pm i t |\lambda|} \lambda \varphi(2^{-j}\lambda) \chi(\lambda r) \chi(\lambda \rho)(\la r \ra \la \rho \ra)^{-d/2} 
\frac{f_{+}(r,\lambda) f_{-}(\rho,\lambda)}{W(\lambda)} d\lambda.
\end{align*}
We first consider the case $\rho < 0 < r$.  
\begin{lem}\label{lem3}
	For all $t \in \R$ and $\rho < 0 < r$
	\begin{align}
	\left | 
	\int_{-\infty}^{\infty} e^{\pm i t |\lambda|} \lambda \varphi(2^{-j}\lambda) \chi_{[|r\lambda|> 1]} \chi_{[|\rho \lambda|> 1]}(\la r \ra \la \rho \ra)^{-d/2} 
	\frac{f_{+}(r,\lambda) f_{-}(\rho,\lambda)}{W(\lambda)} d\lambda 
	\right |
	\lesssim 2^{j(d+2)/2}(2^{-j} + |t|)^{-d/2}. \label{lem31} 
	\end{align}
\end{lem}

\begin{proof}
	We first note that by Lemma \ref{jostlem} and Lemma \ref{scattlem4},  
	\begin{align*}
	\sup_{|r\lambda| > 1, |\rho\lambda| > 1} |\lambda| \left | \frac{f_{+}(r,\lambda) f_{-}(\rho,\lambda)}{W(\lambda)} \right | \lesssim 1.
	\end{align*}
	This implies that 
	\begin{align*}
	\left | 
	\int_{-\infty}^{\infty} \right. & \left. e^{\pm i t |\lambda|} \lambda \varphi(2^{-j}\lambda)  \chi_{[|r\lambda|> 1]} \chi_{[|\rho \lambda|> 1]}(\la r \ra \la \rho \ra)^{-d/2} 
	\frac{f_{+}(r,\lambda) f_{-}(\rho,\lambda)}{W(\lambda)} d\lambda 
	\right | \\
	&\lesssim  
	\int_{-\infty}^{\infty}  \varphi(2^{-j}\lambda)  \chi_{[|r\lambda|> 1]} \chi_{[|\rho\lambda|> 1]}(\la r \ra \la \rho \ra)^{-d/2} d\lambda \\
	&\lesssim 
	\int_{-\infty}^{\infty}  \varphi(2^{-j}\lambda) \lambda^d d\lambda \\
	&\lesssim 
	2^{j(d+1)}. 
	\end{align*}
	Thus, we only need to consider the case $|t| \geq 2^{-j}$.
	
	Assume that $|t| \geq 2^{-j}$.  By Lemma \ref{jostlem}, we write 
	\begin{align*}
	f_+(r, \lambda) = e^{i \lambda r} m_+(r,\lambda), \quad f_-(\rho,\lambda) = e^{-i\lambda \rho} m_-(\rho,\lambda),
	\end{align*}
	where 
	\begin{align}
	m_+(r, \lambda) = 1+ O(\lambda^{-1} r^{-1}), \quad r|\lambda| > 1, \label{lem32}
	\end{align}
	with natural derivative bounds.  A similar expression holds
	for $m_-(\rho,\lambda)$.  We express \eqref{lem31} as 
	\begin{align*}
	\left | \int_{-\infty}^{\infty} e^{i |\lambda| \left (\pm t  +  \frac{\lambda}{|\lambda|}(r - \rho) \right )} a_j(r, \rho, \lambda) d\lambda \right |
	\lesssim 2^{j(d+2)/2} (2^{-j} + |t|)^{-d/2},
	\end{align*}
	where
	\begin{align*}
	a_j(r,\rho,\lambda) = \lambda \varphi(2^{-j}\lambda)  \chi_{[|r\lambda|> 1]} \chi_{[|\rho\lambda|> 1]}(\la r \ra \la \rho \ra)^{-d/2} 
	\frac{m_{+}(r,\lambda) m_{-}(\rho,\lambda)}{W(\lambda)}.
	\end{align*}
	By Lemma \ref{scattlem4}
	\begin{align*}
	\frac{\lambda}{W(\lambda)} = O(\lambda^{d-1}), 
	\end{align*}
	with natural derivative bounds.
	This fact and \eqref{lem32} imply that 
	\begin{align}
	a_j(r, \rho, \lambda) =
	\varphi(2^{-j}\lambda) \chi_{[|r\lambda|> 1]} \chi_{[|\rho\lambda|> 1]} O(\lambda^{d-1}) (\la r \ra \la \rho \ra)^{-d/2}. \label{lem33}
	\end{align}
	
	Note that if $|\lambda|$ is small, $|r \lambda| > 1$, and $|\rho \lambda | > 1$, then we have
	\begin{align*} 
	(\la r \ra \la \rho \ra)^{-d/2} &\lesssim \la r - \rho \ra^{-d/2}, \\
	(\la r \ra \la \rho \ra)^{-d/2} &\leq \lam^d. 
	\end{align*}
	If $|t| \lesssim |r - \rho|$, then since $j \ll 0$, we have 
	\begin{align*} 
	\left | \int_{-\infty}^{\infty} e^{i |\lambda| \left (\pm t  +  \frac{\lambda}{|\lambda|}(r - \rho) \right )} a_j(r, \rho, \lambda) d\lambda \right |
	&\lesssim \int_{[|\lambda| \sim 2^j]} |\lambda|^{d-1} d\lambda \la r - \rho \ra^{-d/2} \\
	&\lesssim 2^{dj} |t|^{-d/2} \\
	&\lesssim 2^{j(d+2)/2} (2^{-j} + |t|)^{-d/2}.
	\end{align*}
	Now suppose $|t| \gg |r - \rho|$.  By \eqref{lem33} and integration by parts
	\begin{align*}
	\left | \int_{0}^{\infty} e^{i \lambda \left (\pm t  +  (r - \rho) \right )}
	a_j(r, \rho, \lambda) d\lambda \right | 
	&= |\pm t + (r - \rho)|^{-d} \left | \int_{0}^{\infty} e^{i \lambda \left (\pm t  +  (r - \rho) \right )}
	\partial_\lambda^d a_j(r, \rho, \lambda) d\lambda \right | \\
	&\lesssim |t|^{-d}\int_{[\lambda \sim 2^j]}  \lambda^{d-1} d \lambda \\
	&\lesssim |t|^{-d} 2^{dj} \\
	&\lesssim 2^{j(d+2)/2} (2^{-j} + |t|)^{-d/2}.  
	\end{align*}
	A similar argument shows that 
	\begin{align*}
	\left | \int_{-\infty}^{0} e^{-i \lambda \left (\pm t  -  (r - \rho) \right )}
	a_j(r, \rho, \lambda) d\lambda \right | \lesssim 2^{j(d+2)/2}(2^{-j} + |t|)^{-d/2}.
	\end{align*}
	This concludes the proof. 
\end{proof}

We now consider the case when $|r\lambda|>1$ and $|\rho\lambda|< 1$ in the integrand.  The case $|r\lambda| < 1$ and $|\rho \lambda| > 1$
can be handled similarly.  

\begin{lem}\label{lem4}
	For all $t,r \in \R$ and $\rho < r$ 
	\begin{align}
	\left | 
	\int_{-\infty}^{\infty} e^{\pm i t |\lambda|} \lambda \varphi(2^{-j}\lambda) \chi_{[|r\lambda|> 1]} \chi_{[|\rho\lambda|< 1]} (\la r \ra \la \rho \ra)^{-2} 
	\frac{f_{+}(r,\lambda) f_{-}(\rho,\lambda)}{W(\lambda)} d\lambda 
	\right |
	\lesssim 2^{j(d+2)/2}(2^{-j} + |t|)^{-d/2}. \label{lem41} 
	\end{align}
\end{lem}

\begin{proof}
	We write $f_+(r,\lambda) = e^{ir\lambda} m_+(r,\lambda)$ as before, but since $|\rho \lambda| < 1$, we use the representation
	\begin{align}
	f_-(\rho,\lambda) = a_-(\lambda) u_0^-(\rho,\lambda) + b_-(\lambda) u_1^-(\rho,\lambda). 
	\end{align}
	In particular, we have that 
	\begin{align*}
	f_-(\rho,\lambda) = O (\lambda^{-(d-2)/2} ) O ( \la \rho \ra^{d/2} ).   
	\end{align*}
	Now we write \eqref{lem41} as 
	\begin{align*}
	\left | 
	\int_{-\infty}^{\infty} e^{ i |\lambda| (\pm t + \frac{\lambda}{|\lambda|} r)} a_j(r,\rho,\lambda) d\lambda 
	\right |
	\lesssim 2^{j(d+2)/2}(2^{-j} + |t|)^{-d/2}, 
	\end{align*}
	where 
	\begin{align*}
	a_j(r,\rho,\lambda) &= \lambda \varphi(2^{-j}\lambda) \chi_{[|r\lambda|> 1]} \chi_{[|\rho \lambda|< 1]} (\la r \ra \la \rho \ra)^{-d/2} 
	\frac{m_{+}(r,\lambda) f_{-}(\rho,\lambda)}{W(\lambda)}\\
	&= \varphi(2^{-j}) \chi_{[|r\lambda|> 1]} \chi_{[|\rho \lambda|< 1]} O(\lambda^{d/2}) \la r \ra^{-d/2},
	\end{align*}
	with natural derivative bounds.  As before, in the case $|t| \geq 2^{-j}$ we have   
	\begin{align*}
	\left | 
	\int_{0}^{\infty} e^{ i \lambda (\pm t + r)} a_j(r,\rho,\lambda) d\lambda 
	\right | 
	&\lesssim \int_{[\lambda \sim 2^j]} \lambda^d d\lambda \\
	&\lesssim 2^{j(d+1)} \\
	&\lesssim 2^{j(d+2)/2} (2^{-j} + |t|)^{-d/2}. 
	\end{align*}
	Thus, we need only consider the case that $|t| \geq 2^{-j}$.  
	
	Suppose that $|t| \geq 2^{-j}$.  If $|t| \lesssim |r|$ then  
	\begin{align*}
	\left | 
	\int_{0}^{\infty} e^{ i \lambda (\pm t + r)} a_j(r,\rho,\lambda) d\lambda 
	\right | 
	&\lesssim \int_{[\lambda \simeq 2^j]} \lambda^{d/2} |t|^{-d/2} d\lambda \\
	&\lesssim 2^{j(d+2)/2} (2^{-j} + |t|)^{-d/2}. 
	\end{align*}
	If $|t| \gg |r|$, then by integration by parts
	\begin{align*}
	\left | 
	\int_{0}^{\infty} e^{ i \lambda (\pm t + r)} a_j(r,\rho,\lambda) d\lambda 
	\right | 
	&= |\pm t + r|^{-d} \left | \int_{0}^{\infty} e^{i \lambda (\pm t + r)} \partial_\lambda^d a_j(r,\rho,\lambda) d\lambda \right | \\
	&\lesssim |t|^{-d} \int_{[\lambda \sim 2^j]} dr \\
	&\lesssim |t|^{-d} 2^j \\
	&\lesssim 2^{j(d+2)/2} (2^{-j} + |t|)^{-d/2}, 
	\end{align*}
	as desired.  Similarly, 
	\begin{align*}
	\left | 
	\int_{-\infty}^{0} e^{ i \lambda (\pm t - r)} a_j(r,\rho,\lambda) d\lambda 
	\right | \lesssim 2^{j(d+2)/2} (2^{-j} + |t|)^{-d/2}.
	\end{align*}
	This concludes the proof. 
\end{proof}

To finish proving Proposition \ref{oscint} in the case $j \ll 0$, we need only consider the case when 
the integrand is supported in $|\lambda|^{-1} < \rho < r$.  The case $\rho < r < -|\lambda|^{-1}$ can be dealt with in a similar fashion.
We consider \emph{reflection and transmission coefficients} $\alpha_-(\lambda),\beta_-(\lambda)$ defined by the relation
\begin{align*}
f_-(\rho,\lambda) = \alpha_-(\lambda) f_+(\rho,\lambda) + \beta_-(\lambda) \overline{f_+(\rho,\lambda)}.
\end{align*}
Then 
\begin{align*}
W(\lambda) &= W(f_-(\cdot, \lambda), f_+(\cdot, \lambda)) \\
&= -\beta_-(\lambda) W(f_+(\cdot, \lambda), \overline{f_+(\cdot, \lambda)}) \\
&= -\beta_-(\lambda)\lim_{r \rightarrow \infty} W(f_+(r, \lambda), \overline{f_+(r, \lambda)}) \\
&= -\beta_-(\lambda) \lim_{r \rightarrow \infty} W(e^{i \lambda r}, e^{-i\lambda r}) \\ 
&=2i\lambda \beta_-(\lambda). 
\end{align*}
Let $\widetilde{W}(\lambda) = W(f_-(\cdot, \lambda), \overline{f_+(\cdot, \lambda)})$.  Then similar to 
$W(\lambda)$ we have 
\begin{align*}
\widetilde{W}(\lambda) &= \alpha_-(\lambda)W(f_+(\cdot, \lambda), \overline{f_+(\cdot, \lambda)}) \\
&= -2i\lambda \alpha_-(\lambda).  
\end{align*}
We conclude that 
\begin{align*}
\lambda \frac{\beta_-(\lambda)}{W(\lambda)} &= \frac{1}{2i}, \\
\lambda \frac{\alpha_-(\lambda)}{W(\lambda)} &= -\frac{1}{2i} \frac{\widetilde{W}(\lambda)}{W(\lambda)} = \mbox{constant} + 
O(\lambda^{\epsilon}),   
\end{align*}
where the $O(\lambda^{\epsilon})$ term is complex valued and satisfies natural derivative bounds. The second equality in the 
second line above follows from 
Lemma \ref{scattlem3}. 

\begin{lem}\label{lem6}
	For all $t \in \R$ and $0 < \rho < r$
	\begin{align}
	\left | 
	\int_{-\infty}^{\infty} e^{\pm i t |\lambda|} \lambda \varphi(2^{-j}\lambda) \chi_{[|\rho\lambda| > 1]} (\la r \ra \la \rho \ra)^{-2} 
	\frac{f_{+}(r,\lambda) f_{-}(\rho,\lambda)}{W(\lambda)} d\lambda 
	\right |
	\lesssim 2^{j(d+2)/2}(2^{-j} + |t|)^{-d/2}. \label{lem61} 
	\end{align}
\end{lem}

\begin{proof}
	We write $f_+(r,\lambda) = e^{i\lambda r} m_+(r,\lambda)$.  Then 
	\begin{align*}
	\lambda \frac{f_+(r,\lambda)f_-(\rho,\lambda)}{W(\lambda)} &= e^{i(r+\rho)\lambda} \lambda \frac{\alpha_-(\lambda)}{W(\lambda)}
	m_+(r,\lambda)m_+(\rho,\lambda)
	+ e^{i(r-\rho)\lambda} \lambda \frac{\beta_-(\lambda)}{W(\lambda)} m_+(r,\lambda)\overline{m_+(\rho,\lambda)} \\
	&= e^{i(r+\rho)\lambda} O(1) m_+(r,\lambda)m_+(\rho,\lambda) + \frac{1}{2i} e^{i(r-\rho)\lambda} m_+(r,\lambda)\overline{m_+(\rho,\lambda)} 
	\end{align*}
	where the $O(1)$ term is complex valued and satisfies natural derivative bounds. We are thus reduced to 
	proving the following two estimates
	\begin{align}
	\left | 
	\int_{-\infty}^{\infty} e^{\pm i t |\lambda|(\pm t + \frac{\lambda}{|\lambda|} (r + \rho))} 
	\varphi(2^{-j}\lambda)  \chi_{[|\rho\lambda| > 1]}(\la r \ra \la \rho \ra)^{-d/2} O(1) m_+(r,\lambda)m_+(\rho,\lambda)
	d\lambda 
	\right |
	&\lesssim 2^{j(d+2)/2}(2^{-j} + |t|)^{-d/2} \label{lem62}, \\
	\left | 
	\int_{-\infty}^{\infty} e^{\pm i t |\lambda|(\pm t + \frac{\lambda}{|\lambda|} (r - \rho))} 
	\varphi(2^{-j}\lambda)  \chi_{[|\rho\lambda| > 1]}(\la r \ra \la \rho \ra)^{-d/2} m_+(r,\lambda) \overline{m_+(\rho,\lambda)}
	d\lambda 
	\right |
	&\lesssim 2^{j(d+2)/2}(2^{-j} + |t|)^{-d/2}. \label{lem63}
	\end{align}
	We now prove \eqref{lem62}.  
	
	We write \eqref{lem62} as 
	\begin{align*}
	\left | 
	\int_{-\infty}^{\infty} e^{ i |\lambda| (\pm t + \frac{\lambda}{|\lambda|} (r+\rho))} a_j(r,\rho,\lambda) d\lambda 
	\right |
	\lesssim 2^{j(d+2)/2}(2^{-j} + |t|)^{-d/2}. 
	\end{align*}
	where 
	\begin{align*}
	a_j(r,\rho,\lambda) &= \varphi(2^{-j}\lambda)\chi_{[|\rho\lambda| > 1]}) (\la r \ra \la \rho \ra)^{-d/2} O(1)
	m_{+}(r,\lambda) m_{+}(\rho,\lambda) \\
	&= \varphi(2^{-j}\lambda) \chi_{[|\rho\lambda| > 1]} (\la r \ra \la \rho \ra)^{-d/2} O(1), 
	\end{align*}
	with the $O(\cdot)$ term behaving like a symbol under differentiation in $\lambda$. 
	Note that if $|t| \leq 2^{-j}$ then $|r\lambda| > 1$ and $|\rho \lambda| > 1$ imply that  
	\begin{align*}
	\left | 
	\int_{-\infty}^{\infty} e^{ i |\lambda| (\pm t + \frac{\lambda}{|\lambda|} (r+\rho))} a_j(r,\rho,\lambda) d\lambda 
	\right | &\lesssim \int_{[\lambda \sim 2^j]} \lambda^{d} d\lambda \\
	&\lesssim 2^{j(d+1)} \\
	&\lesssim 2^{j(d+2)/2)} (2^{-j} + |t|)^{-d/2}.   
	\end{align*}
	Thus, we need only consider $|t| \geq 2^{-j}$.
	
	Suppose that $|t| \leq 2(r + \rho)$.  Then $0 < \rho < r$ implies that $r \geq |t|/4$ so that 
	\begin{align*}
	\chi_{[|\rho\lambda| > 1]} (\la r \ra \la \rho \ra)^{-d/2}\lesssim \lambda^{d/2}|t|^{-d/2}.  
	\end{align*}
	Thus, 
	\begin{align*}
	\left | 
	\int_{-\infty}^{\infty} e^{ i |\lambda| (\pm t + \frac{\lambda}{|\lambda|} (r+\rho))} a_j(r,\rho,\lambda) d\lambda 
	\right | &\lesssim \int_{[\lambda \simeq 2^j]} \lambda^{d/2} |t|^{-d/2} d\lambda \\
	&\lesssim 2^{j(d+2)/2} |t|^{-d/2}.
	\end{align*}
	as desired. Suppose now that $|t| \geq 2(r+\rho)$.  Integration by parts yields 
	\begin{align*}
	\left | 
	\int_{0}^{\infty} e^{ i \lambda (\pm t + (r+\rho))} a_j(r,\rho,\lambda) d\lambda 
	\right |
	&=  |\pm t + (r + \rho)|^{-d} \left | 
	\int_{0}^{\infty} e^{ i \lambda (\pm t + (r+\rho))} \partial_\lambda^d a_j(r,\rho,\lambda) d\lambda 
	\right | \\
	&\lesssim |t|^{-d} \int_{[\lambda \simeq 2^j]} \lambda^{-d} \lambda^{d} d\lambda\\
	&\lesssim |t|^{-d} 2^j \\
	&\lesssim 2^{j(d+2)/2} ( 2^{-j} + |t| )^{-d/2}. 
	\end{align*}
	In a similar fashion, we obtain 
	\begin{align*}
	\left | 
	\int_{-\infty}^{0} e^{ -i \lambda (\pm t - (r+\rho))} a_j(r,\rho,\lambda) d\lambda 
	\right | 
	&\lesssim 2^{j(d+2)/2} |t|^{-d/2}.
	\end{align*}
	This proves \eqref{lem62}.  The proof of \eqref{lem63} is similar and is omitted. 
\end{proof}

We now prove Proposition \ref{oscint} in the case $j \gtrsim 0$.  This case is considerably simpler than the case $j \ll 0$ since
the Jost functions $f_\pm(\cdot,\lam)$ and their Wronskian $W(\lam)$ are to given by the free case $H = -\frac{d^2}{dr^2}$
to leading order.  Indeed, we write 
\begin{align*}
f_+(r,\lambda) = e^{ir\lambda} m_+(r,\lambda), \quad f_-(\rho,\lambda) = e^{-i\rho\lambda} m_-(\rho,\lambda).
\end{align*}
From \cite{sss1}, we have the estimates 
\begin{align*}
m_+(r, \lambda) &= 1 + O(\lambda^{-1} \la r \ra^{-1}), \\
\left | \partial_\lambda^l \partial_r^k m_+(r,\lambda) \right | &\lesssim_{l,k} \lambda^{-1-l} \la r \ra^{-1-k} 
\end{align*}
for $r \geq 0$ and $l + k > 0$.  Similar estimates hold for $m_-(\rho,\lambda)$ with $\rho \leq 0$. 
It is well known that $|W(\lambda)| \geq |\lambda|$ for all $\lam$.  Using the asymptotics for $m_\pm(\cdot, \lambda)$, we compute the Wronskian 
\begin{align*}
W(\lambda) &= W(f_-(\cdot, \lambda), f_+(\cdot,\lambda) ) \\
&= m_+(0,\lambda)(m_-'(0,\lambda) - i\lambda m_-(0,\lambda)) - m_-(0,\lambda)(m'_+(0,\lambda) + i\lambda m_+(0,\lambda)) \\
&= -2i\lambda + O(\lam^{-1}),
\end{align*}
with natural derivative bounds. We also compute the Wronskian
\begin{align*}
W(f_-(\cdot, \lambda), \overline{f_+(\cdot,\lambda)} )
&= m_-(0,\lambda)(\bar m_+'(0,\lambda) - 2i\lambda \bar m_+(0,\lambda)) - \bar m_+(0,\lambda)(m'_-(0,\lambda) - 2i\lambda m_-(0,\lambda)) \\
&= m_-(0,\lambda)\bar m_+'(0,\lambda) - m_-'(0,\lambda)\bar m_+(0,\lambda) \\
&= O(\lambda^{-1}). 
\end{align*}
with symbol character in $\lam$.  We now prove Proposition \ref{oscint} in the case $j \gtrsim 0$. 

\begin{lem}\label{lem5} For all $\rho < r$ 
	\begin{align}
	\left | 
	\int_{-\infty}^{\infty} e^{\pm i t |\lambda|} \lambda \varphi(2^{-j}\lambda) (\la r \ra \la \rho \ra)^{-d/2} 
	\frac{f_{+}(r,\lambda) f_{-}(\rho,\lambda)}{W(\lambda)} d\lambda 
	\right |
	\lesssim 2^{j(d+2)/2}(2^{-j} + |t|)^{-d/2}. \label{lem51}
	\end{align}
\end{lem}

\begin{proof} 
	We first note that the fact that 
	\begin{align*}
	\sup_{r, \rho} |\lambda| \left | \frac{f_{+}(r,\lambda) f_{-}(\rho,\lambda)}{W(\lambda)} \right | \lesssim 1, 
	\end{align*}
	implies that 
	\begin{align*}
	\left | 
	\int_{-\infty}^{\infty}  e^{\pm i t |\lambda|} \lambda \varphi(2^{-j}\lambda) (\la r \ra \la \rho \ra)^{-d} 
	\frac{f_{+}(r,\lambda) f_{-}(\rho,\lambda)}{W(\lambda)} d\lambda 
	\right | 
	&\lesssim  
	\int_{-\infty}^{\infty}  \varphi(2^{-j}\lambda) (\la r \ra \la \rho \ra)^{-d} d\lambda \\
	&\lesssim 2^j \\
	&\lesssim 
	2^{j(d+2)/2}. 
	\end{align*}
	In the last line we used $j \gtrsim 0$.  Thus, we only need to consider the case $|t| \geq 2^{-j}$.  We split the remainder of the proof into cases: $\rho < 0 < r$, $0 < \rho < r$, and $\rho < r < 0$.  By symmetry we consider only the first two. 
	
	Assume $\rho < 0 < r$.  We write 
	\begin{align*}
	f_+(r, \lambda) = e^{i \lambda r} m_+(r,\lambda), \quad f_-(\rho,\lambda) = e^{-i\lambda \rho} m_-(\rho,\lambda),
	\end{align*}
	where we have for all $r \geq 0$
	\begin{align}
	\left | m_+(r,\lambda) \right | &\lesssim 1, \label{lem53} \\
	\left | \partial_\lambda^l m_+(r, \lambda) \right | &\lesssim \lambda^{-1 - l}, \quad l > 0 \label{lem52}, 
	\end{align} 
	with similar estimates holding for $m_-(\rho,\lambda)$ for $\rho \leq 0$. We express \eqref{lem51} as 
	\begin{align*}
	\left | \int_{-\infty}^{\infty} e^{i |\lambda| \left (\pm t  +  \frac{\lambda}{|\lambda|}(r - \rho) \right )} a_j
	(r, \rho, \lambda) d\lambda \right | \lesssim 2^{j(d+2)/2}(2^{-j} + |t|)^{-d/2},
	\end{align*}
	where
	\begin{align*}
	a_j(r,\rho,\lambda) = \lambda \varphi(2^{-j}\lambda) (\la r \ra \la \rho \ra)^{-d/2} 
	\frac{m_{+}(r,\lambda) m_{-}(\rho,\lambda)}{W(\lambda)}.
	\end{align*}  
	By \eqref{lem53} and \eqref{lem52} we have 
	\begin{align}
	a_j(r,\rho,\lam) = \varphi(2^{-j} \lam) (\la r \ra \la \rho \ra)^{-d/2} O(1), \label{lem54}
	\end{align}
	with natural derivative bounds. 
	
	Suppose that $|t| \leq 2 |r - \rho|$. Then either $|r| \geq |t|/4$ or $|\rho| \geq |t|/4$.  Suppose, 
	without loss of generality,  
	$|r| \geq |t|/4$. Then by \eqref{lem54} we have
	\begin{align*}
	\left | \int_{-\infty}^{\infty} e^{i |\lambda| \left (\pm t  +  \frac{\lambda}{|\lambda|}(r - \rho) \right )} a_j
	(r, \rho, \lambda) d\lambda \right | &\lesssim \int_{[\lambda \sim 2^j]} d\lambda  (\la r \ra \la \rho \ra)^{-d/2} \\
	&\lesssim 2^j |t|^{-d/2} \\
	&\lesssim 2^{j(d+2)/2} (2^{-j} + |t|)^{-d/2}.
	\end{align*} 
	
	Suppose now that $|t| \geq 2 |r - \rho|$.  Then by \eqref{lem54} and integration by parts
	\begin{align*}
	\left | \int_{0}^{\infty} e^{i \lambda \left (\pm t  +  (r - \rho) \right )} a_j
	(r, \rho, \lambda) d\lambda \right |
	&= |\pm t + (r-\rho)|^{-d} \left | 
	\int_{0}^{\infty} e^{i \lambda \left (\pm t  +  (r - \rho) \right )} \partial_\lambda^d a_j
	(r, \rho, \lambda) d\lambda
	\right | \\
	&\lesssim |t|^{-d} \int_{[\lambda \sim 2^j]} d\lambda \\
	&\lesssim |t|^{-d} 2^j \\
	&\lesssim 2^{j(d+2)/2} (2^{-j} + |t|)^{-d/2}, 
	\end{align*}
	as desired.  Similarly, 
	\begin{align*}
	\left | \int_{-\infty}^{0} e^{i |\lambda| \left (\pm t  - (r - \rho) \right )} a_j
	(r, \rho, \lambda) d\lambda \right | \lesssim 2^{j(d+2)/2}(2^{-j} + |t|)^{-d/2}.
	\end{align*}
	This concludes the case $\rho < 0 < r$. 
	
	We now consider the case $0 < \rho < r$.  In this case, we use transmission and reflection coefficients and write
	\begin{align*}
	f_-(\rho,\lambda) = \alpha_-(\lambda) f_+(\rho,\lambda) + \beta_-(\lambda) \overline{f_+(\rho,\lambda)},
	\end{align*}
	where
	\begin{align*}
	\alpha_-(\lambda) &= \frac{W(f_-(\cdot, \lambda) \overline{f_+(\cdot, \lambda)})}{-2i\lambda}, \\
	\beta_-(\lambda) &= \frac{W(\lambda)}{2i\lambda}.
	\end{align*}
	Then using our high energy asymptotics for $W(\lambda)$ and $W(f_-(\cdot,\lambda),\overline{f_+(\cdot,\lambda)})$, we have for
	$\lambda \gtrsim 1$ 
	\begin{align*}
	\lambda \frac{\alpha_-(\lambda)}{W(\lambda)} = O(\lambda^{-2}) = O(1), \quad  
	\lambda \frac{\beta_-(\lambda)}{W(\lambda)} = O(1),
	\end{align*}
	with natural derivative bounds.  Thus, to prove \eqref{lem51} for the case $0 < \rho < r$, we are reduced to proving the bounds
	\begin{align}
	\left | 
	\int_{-\infty}^{\infty} e^{\pm i t |\lambda|(\pm t + \frac{\lambda}{|\lambda|} (r + \rho))} 
	\varphi(2^{-j}\lambda) (\la r \ra \la \rho \ra)^{-d/2} O(1) m_+(r,\lambda)m_+(\rho,\lambda)
	d\lambda 
	\right |
	&\lesssim 2^{j(d+2)/2}(2^{-j} + |t|)^{-d/2} \label{lem55}, \\
	\left | 
	\int_{-\infty}^{\infty} e^{\pm i t |\lambda|(\pm t + \frac{\lambda}{|\lambda|} (r - \rho))} 
	\varphi(2^{-j}\lambda) (\la r \ra \la \rho \ra)^{-d/2} O(1) m_+(r,\lambda)\overline{m_+(\rho,\lambda)}
	d\lambda 
	\right |
	&\lesssim 2^{j(d+2)/2}(2^{-j} + |t|)^{-d/2} \label{lem56}. 
	\end{align}
	We write \eqref{lem55} as 
	\begin{align*}
	\left | 
	\int_{-\infty}^{\infty} e^{ i |\lambda| (\pm t + \frac{\lambda}{|\lambda|} (r+\rho))} a_j(r,\rho,\lambda) d\lambda 
	\right |
	\lesssim 2^{j(d+2)/2}(2^{-j} + |t|)^{-d/2}.
	\end{align*}
	where 
	\begin{align*}
	a_j(r,\rho,\lambda) = \varphi(2^{-j}\lambda) (\la r \ra \la \rho \ra)^{-d/2} O(1)
	m_{+}(r,\lambda) m_{+}(\rho,\lambda).
	\end{align*}
	Then $a_j(r,\rho,\lambda)$ satisfies 
	\begin{align}
	a_j(r,\rho,\lam) = \varphi(2^{-j} \lam) ( \la r \ra \la \rho \ra )^{-d/2} O(1), \label{lem57}
	\end{align}
	with natural derivative bounds.
	But now we are in the same situation as in the case $\rho < 0 < r$ with \eqref{lem57} replacing \eqref{lem54} and we obtain
	\eqref{lem55} in a similar fashion.  The estimate \eqref{lem56} is obtained similarly and we omit the details.  This 
	concludes the proof of Lemma \ref{lem5} and also Proposition \ref{oscint}.  
\end{proof}




\section{Reduction to Higher Dimensions and the Linearized Equation}

In this section, we initiate the study of the evolution \eqref{s04}.  In the first subsection, we linearize degree $n$ 
solutions to \eqref{s04}
around the harmonic map $Q_n$ and make a 
reduction that incorporates the extra dispersion inherent in \eqref{s04}.  Our main result, Theorem \ref{t01}, is then restated in an equivalent 
form which we devote the rest of this work to proving.  The remaining subsections establish Strichartz 
estimates for the linear part of the new equation which will be used in Section 5.  In what follows we use 
the notation from the previous section and denote the $d$--dimensional
wormhole by  
$\M^d$.

\subsection{Reduction to a Wave Equation on a $5d$ Wormhole}

We recall from the introduction that a corotational wave map on a wormhole $U : \R \times \M^3 \rar \s^3$ with topological degree $n$
is a map $U(t,r,\theta,\vphi) = (\psi(t,r), \theta,\vphi)$ such the azimuth angle $\psi = 
\psi(t,r)$ satisfies the Cauchy problem
\begin{align}\label{s41}
\begin{split}
&\p_t^2 \psi - \p_r^2 \psi - \frac{2r}{r^2 + 1} \p_r \psi + \frac{\sin 2 \psi}{r^2 + 1} = 0, \\
&\psi(t,-\infty) = 0, \quad \psi(t,\infty) = n \pi, \quad \forall t, \\
&\vec \psi(0) = (\psi_0,\psi_1). 
\end{split}
\end{align}
The following energy is conserved along the flow
\begin{align*}
\mathcal E(\psi) = \frac{1}{2} \int \left[ |\p_t \psi|^2 + |\p_r \psi|^2 + \frac{2\sin^2 \psi }{r^2 + 1} \right] (r^2 + 1)dr,
\end{align*}
and so, it is natural to take initial data $(\psi_0,\psi_1)$ in the metric space
\begin{align*}
\mathcal E_n = \left \{ 
(\psi_0,\psi_1) : \mathcal E(\psi_0,\psi_1) < \infty, \quad \psi_0(-\infty) = 0, \quad \psi_0(\infty) = n\pi
\right \}.
\end{align*}
For the remainder of this work, we fix the topological degree $n \in \N \cup \{0\}$. We now reduce the study of the large
data solutions to \eqref{s41} to the study of large data solutions to a semilinear wave equation on a $5d$ wormhole.  

By Proposition \ref{harm}, there exists a unique finite energy static solution $Q_n$ to \eqref{s41}, i.e. a solution $Q_n \in 
\mathcal E_n$ such that 
\begin{align}
\p_r^2 Q_n + \frac{2r}{r^2 + 1} \p_r Q_n - \frac{\sin 2 Q_n}{r^2 + 1} = 0. \label{s42}
\end{align}
To simplify notation, we write $Q$ instead of $Q_n$.  For a solution $\psi$ to \eqref{s41}, define $\vphi$ by 
\begin{align*}
\psi(t,r) = Q(r) + \vphi(t,r).
\end{align*}
Then \eqref{s41} and \eqref{s42} imply that $\vphi$ satisfies
\begin{align}\label{s43} 
\begin{split}
&\p_t^2 \vphi - \p_r^2 \vphi - \frac{2r}{r^2 + 1} \p_r \vphi + \frac{2\cos 2 Q}{r^2 + 1} \vphi = Z(r,\vphi), \\
&\vphi(t,-\infty) = \vphi(t,\infty) = 0, \quad \forall t, \\
&\vec \vphi(0) = (\psi_0-Q,\psi_1), 
\end{split}
\end{align}
where 
\begin{align*}
Z(r,\phi) = \frac{1}{r^2 + 1} \left [ 2 \vphi - \sin 2\vphi \right ] \cos 2 Q + ( 1 - \cos 2 \vphi) \sin 2 Q.
\end{align*}
The left--hand side of \eqref{s43} has more dispersion than a free wave on $\M^3$ due to the repulsive potential 
\begin{align*}
\frac{2 \cos 2 Q}{r^2 + 1} = \frac{2}{r^2 + 1} + O(\la r \ra^{-6} )
\end{align*}
as $r \rar \pm \infty$. The $O(\la r \ra^{-6})$ term is due to the asymptotics of $Q$ at $\pm \infty$ (see 
Proposition \ref{harm}).  We now make a standard reduction that incorporates this extra dispersion.  Set $\vphi = \la r \ra u$.  Then $u$ satisfies
the radial semilinear wave equation
\begin{align}\label{s44} 
\begin{split}
&\p_t^2 u - \Delta_g u + V(r) u = N(r,u), \\
&u(t,-\infty) = u(t,\infty) = 0, \quad \forall t, \\
&\vec u(0) = (u_0,u_1),
\end{split}
\end{align}
where $-\Delta_g$ is the (radial) Laplace operator on $\M^5$
\begin{align*}
-\Delta_g u = - \p_r^2 u - \frac{4r}{r^2 + 1} \p_r u,
\end{align*}
the potential is 
\begin{align}\label{s45}
V(r) = \la r \ra^{-4} + 2 \la r \ra^{-2} ( \cos 2 Q - 1 ), 
\end{align}
and $N(r,u) = F(r,u) + G(r,u)$ with 
\begin{align}
\begin{split}\label{s46}
F(r,u) &= 2 \la r \ra^{-3} \sin^2 (\la r \ra u) \sin 2 Q, \\
G(r,u) &= \la r \ra^{-3} \left [ 2 \la r \ra u - \sin (2 \la r \ra u) \right ] \cos 2 Q.
\end{split}
\end{align}
By Proposition \ref{harm}, the potential $V$ is smooth and satisfies
\begin{align}\label{s47}
V(r) = \la r \ra^{-4} + O ( \la r \ra^{-6} ).
\end{align}
Moreover, since $Q(-r) + Q(r) = n\pi$, $V(r)$ is an even function.  The nonlinearities $F$ and $G$ satisfy 
\begin{align}
F(r,u) &= \left ( 2 \sin 2 Q \la r \ra^{-1} \right ) u^2 + F_0(r,u), \label{s48a}\\
|F_0(r,u)| &\lesssim \la r \ra^{-1} u^4, \label{s48b}\\
|G(r,u)| &\lesssim |u|^3, \label{s49}
\end{align}
where the implied constants are absolute.  Based on our definition of $u$ in terms of the original 
azimuth function $\psi$, we consider radial initial data $(u_0,u_1) 
\in \mathcal H(\R; (r^2 + 1)^2 dr)$ in \eqref{s44}.  For the remainder of this section, we denote $
\h_0 := \h(\R; (r^2 + 1) dr)$ and $\h:= \h(\R; (r^2 + 1)^2 dr)$ by $\h$.  We note that 
$\h_0$ is simply the space of radial functions in $\dot H^1 \times L^2(\M^3)$ and $\h$ is the space 
of radial functions in $\dot H^1 \times L^2(\M^5)$. 

In the remainder of the paper, we work only with the \lq $u$--formulation' rather than with the original azimuth angle 
$\psi$.  The reason that a solution $\vec \psi(t) \in C(\R; \cl \h_n)$ to \eqref{s41} with initial data $(\psi_0,\psi_1) \in 
\cl E_n$ yields a solution $\vec u(t) \in C(\R; \h)$ with initial data $(u_0,u_1) = \la r \ra^{-1}( \psi_0 - Q, \psi_1)
\in \h$ and vice versa is as follows.  The only fact that needs to be checked is that 
\begin{align}\label{s410}
\| \vec u \|_{\h} \simeq \| \vec \psi - (Q,0) \|_{\h_0}.
\end{align}
Set $\vphi = \psi - Q = \la r \ra u$.  Then 
\begin{align}
\p_r \vphi = \la r \ra \p_r u + \frac{r}{\la r \ra^2} u. \label{s411}
\end{align}
We note that we have the following Hardy's inequalities
\begin{align*}
\int |\vphi|^2 dr &\lesssim \int |\p_r \vphi|^2 (r^2 + 1) dr, \\
\int |u|^2 (r^2 + 1) dr &\lesssim \int |\p_r u|^2 (r^2 + 1)^2 dr.
\end{align*}
These estimates follow easily from integration by parts and the Strauss estimates
\begin{align}
|\vphi(r)| &\lesssim \la r \ra^{-1/2} \left ( \int |\p_r \vphi|^2 (r^2 + 1) dr \right )^{1/2}, \notag \\
|u(r)| &\lesssim \la r \ra^{-3/2} \left ( \int |\p_r u|^2 (r^2 + 1)^2 dr \right )^{1/2}. \label{s411s}
\end{align}
The Strauss estimates are a simple consequence of the fundamental theorem of calculus.  The two Hardy's inequalities
and \eqref{s411} imply \eqref{s410}.  Hence, the two Cauchy problems \eqref{s41} and \eqref{s44} are equivalent.  

The equivalent $u$--formulation of our main result, Theorem \ref{t01}, is the following. 

\begin{thm}\label{t41}
	For any initial data $(u_0,u_1) \in \h$, there exists a unique global solution $\vec u(t) \in C(\R; \h)$ to \eqref{s44}
	which scatters to free waves on $\M^5$, i.e. there exist solutions $v_L^{\pm}$ to 
	\begin{align*}
	\p_t^2 v - \p_r^2 v - \frac{4r}{r^2 + 1} \p_r v = 0, \quad (t,r) \in \R \times \R, 
	\end{align*}
	such that 
	\begin{align*}
	\lim_{t \rightarrow \pm \infty} \| \vec u(t) - \vec v_L^{\pm}(t) \|_{\h} = 0. 
	\end{align*}
\end{thm}

The remainder of this work is devoted to proving Theorem \ref{t41}.  
In order to study the nonlinear evolution \eqref{s44}, we will need Strichartz estimates for the linear operator $\p_t^2 - 
\Delta_g + V$ where $V$ is as in \eqref{s45}.  

\subsection{Strichartz Estimates for the Linearized Operator}

The goal of this subsection is to prove Strichartz estimates for radial solutions to the free wave equation on $\M^5$ perturbed
by a radial potential $V = V(r)$
\begin{align}
\begin{split}\label{s71}
&\p_t^2 u - \Delta_g u + V u = F, \quad (t,r) \in I \times \R, \\
&\vec u(0) = (u_0,u_1).
\end{split}
\end{align}
The particular case we are interested in is the case that the potential $V$ is given by
\begin{align*}
V(r) = \la r \ra^{-4} + 2 \la r \ra^{-2} ( \cos 2 Q - 1 ), 
\end{align*}
where $Q$ is the unique harmonic map of degree $n$. The Strichartz estimates we establish will be used in the next section
to study the nonlinear evolution \eqref{s44}.  We recall from Section 3 that we say that a triple $(a,b,\gamma)$ is admissible for $M^5$
if 
\begin{align*}
p > 2, q \geq 2, \quad 
\frac{1}{p} + \frac{5}{q} = \frac{5}{2} - \gamma, \quad 
\frac{1}{p} \leq 1 - \frac{2}{q}.
\end{align*} 
The main result of this subsection is the following. 

\begin{ppn}\label{p71}
	Let $V \in C^\infty(\R)$ be even such that 
	\begin{align}\label{s72}
	|V^{(j)}(r)| \lesssim_j \la r \ra^{-4 - j}
	\end{align}
	for all $r \in \R$.  Assume that $-\Delta_g + V$ has no point spectrum (when restricted to radial functions) and that 0
	is not a resonance of the Schr\"odinger operator on the line given by $-\frac{d^2}{dr^2} + 2 \la r \ra^{-2} + V(r)$.  Let
	$(p,q,\gamma)$ and $(a,b,\rho)$ be two admissible triples for $\M^5$.  Then any radial solution $u$ to \eqref{s71}
	satisfies
	\begin{align}\label{s72a}
	\| |\nabla|^{1-\gamma} u \|_{L^p_tL^q_x(I)} + 
	\| |\nabla|^{-\gamma} \p_t u \|_{L^p_tL^q_x(I)} \lesssim \| \vec u(0) \|_{\h} + \| |\nabla |^{\rho} F \|_{L^{a'}_t
		L^{b'}_x(I)}.
	\end{align}
\end{ppn}

\begin{proof}
	The proof is based on arguments in Section 5 of \cite{ls}.  By standard $TT^*$ arguments and Minkowski's inequality (c.f. \cite{sogge} or \cite{tao}), we only need to consider
	the case $F = 0$.  As we will see, the proof of Proposition \ref{p71} reduces to proving certain local energy estimates.  
	Indeed, define
	\begin{align*}
	A = \sqrt{-\Delta_g}.
	\end{align*}
	Note that 
	\begin{align}\label{s73}
	\| A f \|_{L^2}^2 = (A^2 f , f )_{L^2} = (-\Delta_g f , f)_{L^2} = \| \nabla f \|^2_{L^2}.  
	\end{align}
	For a solution $u$ to \eqref{s71}, define 
	\begin{align}\label{s74}
	w(t) = A u(t) + i \p_t u(t). 
	\end{align}
	Then by \eqref{s73}, 
	\begin{align}\label{s75}
	\| w(t) \|_{L^2} = \| \vec u(t) \|_{\h},
	\end{align} 
	and $w$ satisfies
	\begin{align}
	\begin{split}\label{s76}
	&i \p_t w = A w + V u, \quad (t,r) \in I \times \R, \\
	&w(0) = A u_0 + i u_1. 
	\end{split}
	\end{align}
	By Duhamel's principle, \eqref{s76} implies that 
	\begin{align*}
	w(t) = e^{-itA} w(0) - i \int_0^t e^{-i(t-s)A} Vu(s) ds.
	\end{align*}
	The Strichartz estimates \eqref{s72a} can be restated as 
	\begin{align}\label{s77}
	\| P w \|_{X} \leq \| w(0) \|_{L^2},
	\end{align}
	where $P := A^{-1} \Re$ and $\| \cdot \|_X := \| |\nabla|^{-\gamma} \nabla_{t,x} \cdot \|_{L^p_t L^q_x(I)}$. By Proposition \ref{strm}, 
	\begin{align}\label{s77b}
	\| P e^{-itA} w(0) \|_X \lesssim \| w(0) \|_{L^2}.
	\end{align}
	Thus, 
	\begin{align*}
	\| P w \|_X \lesssim \| w(0) \|_{L^2} + \left \| P \int_0^t e^{-i(t-s)A} Vu(s) ds \right \|_X. 
	\end{align*}
	By the Christ--Kiselev lemma, to bound the second term above, it suffices to show that 
	\begin{align}\label{s78}
	\left \| P \int_{-\infty}^\infty e^{-i(t-s)A} Vu(s) ds \right \|_X \lesssim \| w(0) \|_{L^2}. 
	\end{align}
	To prove \eqref{s78}, we write $V = V_1V_2$ were each factor $V_j$ is even and satisfies $|V_j(r)|\lesssim \la r \ra^{-2}$.
	Then 
	\begin{align}\label{s79}
	\left \| P \int_{-\infty}^\infty e^{-i(t-s)A} Vu(s) ds \right \|_X \lesssim \| K \|_{L^2_{t,x} \rar X} 
	\| V_2 u \|_X
	\end{align}
	where
	\begin{align*}
	K F(t) := P \int_{-\infty}^\infty e^{-i(t-s)A} V_1 F(s) ds.
	\end{align*}
	If $F \in L^2_{t,x}$, then by \eqref{s77b} 
	\begin{align*}
	\| K F \|_X &\leq \| P e^{-itA} \|_{L^2_{x} \rar X} \left \| 
	\int_{-\infty}^\infty e^{isA} V_1 F(s) ds \right \|_{L^2_x} \\
	&\lesssim \left \| 
	\int_{-\infty}^\infty e^{isA} V_1 F(s) ds \right \|_{L^2_x}.
	\end{align*}
	We now wish to show that 
	\begin{align*}
	\left \| 
	\int_{-\infty}^\infty e^{isA} V_1 F(s) ds \right \|_{L^2_x}  \lesssim \| F \|_{L^2_{t,x}}.
	\end{align*}
	By duality, this estimate is equivalent to the local energy estimate
	\begin{align*}
	\| V_1 e^{-itA} \vphi \|_{L^2_{t,x}} \lesssim \| \vphi \|_{L^2_x}.
	\end{align*}
	Thus, by \eqref{s79}, the proof of Proposition \ref{p71} is reduced to proving the local energy estimates
	\begin{align}
	\| V_1 e^{-itA} \vphi \|_{L^2_{t,x}} &\lesssim \| \vphi \|_{L^2_x}, \label{s711}\\
	\| V_2 u \|_{L^2_{t,x}} &\lesssim \| \vec u(0) \|_{\h}. \label{s712}
	\end{align}
	
	To prove \eqref{s711} and \eqref{s712}, we first eliminate the weight $\la r \ra^4$ inherent in them.  Consider
	the isomorphism $\phi \mapsto f := \la r \ra^2 \vphi$ from $L^2(\M^5)$ (restricted to radial functions) to $L^2(\R)$. Define
	the following Schr\"odinger operators on $\R$ by
	\begin{align}
	\begin{split}\label{s713}
	H_0 &:= -\frac{d^2}{dr^2} + \frac{2}{r^2 + 1}, \\
	H := H_0 + V &:= -\frac{d^2}{dr^2} + \frac{2}{r^2 + 1} + V(r).
	\end{split}
	\end{align}
	Then 
	\begin{align}
	\begin{split}\label{s714}
	H_0 &= \la r \ra^{2} (-\Delta_g ) \la r \ra^{-2}, \\
	H &= \la r \ra^2 (-\Delta_g + V) \la r \ra^{-2}.
	\end{split}
	\end{align}
	Thus, from \eqref{s714}, we see that \eqref{s711} is equivalent to the estimate
	\begin{align}\label{s715}
	\| V_1 e^{-it\sqrt{H_0}} f \|_{L^2_{t,r}(\R \times \R)} \lesssim \| f \|_{L^2(\R)}.
	\end{align}
	We claim that there exist a distorted Fourier basis $\{ \tht_0(r,\lam^2), \phi_0(r,\lam^2) \}$ that satisfies
	\begin{align}
	\begin{split}\label{s716}
	H_0 \tht_0(r,\lam^2) = \lam^2 \tht_0(r,\lam^2),& \quad H_0 \phi_0(r,\lam^2) = \lam^2 \phi_0(r,\lam^2), \\
	\tht_0(0,\lam^2) = 1,& \quad \phi_0(0,\lam^2) = 0, \\
	\tht_0'(0,\lam^2) = 0,& \quad \phi_0'(0,\lam^2) = 1,
	\end{split}
	\end{align}
	and positive measures $\rho_{0,1}(d\lam) = \om_{0,1}(\lam)d\lam$
	and $\rho_{0,2}(d\lam) = \om_{0,2}(\lam)d\lam$ such that 
	if we define 
	\begin{align*}
	\hat f_{0,1}(\lam) := \int \tht_0(r,\lam^2) f(r) dr, \quad \hat f_{0,2}(\lam) := \int \phi_0(r,\lam^2) f(r) dr, 
	\quad f \in L^2(\R),
	\end{align*}
	then 
	\begin{align}
	f(r) &= \int_0^\infty \tht_0(r,\lam^2) \hat f_{0,1}(\lam) \rho_{0,1}(d\lam) +
	\int_0^\infty \phi_0(r,\lam^2) \hat f_{0,2}(\lam) \rho_{0,2}(d\lam), \label{s718} \\
	\| f \|^2_{L^2(\R)} &=
	\int_0^\infty |\hat f_{0,1}(\lam)|^2 \rho_{0,1}(d\lam) + 
	\int_0^\infty |\hat f_{0,2}(\lam)|^2 \rho_{0,2}(d\lam), \label{s719} \\
	\sup_{r \in \R, \lam > 0}&
	\left ( \frac{1 + \lam^2 \la r \ra^2}{\lam^2 \la r \ra^2} \right )
	\left [ 
	|\tht_0(r,\lam^2)|^2\om_{0,1}(\lam) + 
	|\phi_0(r,\lam^2)|^2\om_{0,2}(\lam)
	\right ] < \infty. \label{s720}
	\end{align}
	The proof of this claim is postponed until the next subsection.  Assuming the claim, we can easily establish 
	\eqref{s715}.  Indeed, since $H_0 \mapsto \lam^2$ on the Fourier side, \eqref{s715} can be rewritten as 
	\begin{align}\label{s721}
	\int \left \| 
	V_1(r) \left [
	\int_0^\infty e^{-it\lam} \tht_0(r,\lam^2) \hat f_{0,1}(\lam) \rho_{0,1}(d\lam) +
	\int_0^\infty e^{-it\lam} \phi_0(r,\lam^2) \hat f_{0,2}(\lam) \rho_{0,2}(d\lam)
	\right ] \right  \|_{L^2(\R)}^2 dt \lesssim \| f \|_{L^2(\R)}^2.
	\end{align}
	Expanding and carrying out the $t$--integration, the left hand side of \eqref{s721} becomes 
	\begin{align}
	\int V_1^2(r) & \Bigl [ 
	\int_0^\infty \int_0^\infty \de(\lam-\mu) \tht_0(r,\lam^2) \tht_0(r,\mu^2) \hat f_{0,1}(\lam) \overline
	{\hat f_{0,1}(\mu)} \rho_{0,1}(d\lam) \rho_{0,1}(d\mu) \notag \\
	&\: +\int_0^\infty \int_0^\infty \de(\lam-\mu) \phi_0(r,\lam^2) \phi_0(r,\mu^2) \hat f_{0,2}(\lam) \overline
	{\hat f_{0,2}(\mu)} \rho_{0,2}(d\lam) \rho_{0,2}(d\mu) \Bigr ] dr \notag \\
	&= \int V_1^2(r) \Bigl [ 
	\int_0^\infty |\hat f_{0,1}(\lam)|^2 |\tht_0(r,\lam^2)|^2 \om^2_{0,1}(\lam) \:d\lam + 
	\int_0^\infty |\hat f_{0,2}(\lam)|^2 |\phi_0(r,\lam^2)|^2 \om^2_{0,2}(\lam) \:d\lam
	\Bigr ] dr. \label{s722}
	\end{align}
	We remark here that no cross terms involving $\tht_0(r,\lam^2)\phi_0(r,\mu^2)$ appeared when expanding since \newline 
	$V_1^2(r) \tht_0(r,\lam^2)\phi_0(r,\mu^2)$ is an odd function of $r$ by \eqref{s716} and our assumption that $V(r)$ is even.
	By \eqref{s720} and \eqref{s719}, we conclude that 
	\begin{align*}
	\eqref{s722} &\lesssim  \int V_1^2(r) \Bigl [ 
	\int_0^\infty |\hat f_{0,1}(\lam)|^2  \om_{0,1}(\lam) \:d\lam + 
	\int_0^\infty |\hat f_{0,2}(\lam)|^2 \om_{0,2}(\lam) \:d\lam
	\Bigr ] dr \\
	&= \| f \|_{L^2(\R)}^2 \int V^2_1(r) dr \\
	&\lesssim \| f \|_{L^2(\R)}^2. 
	\end{align*}
	This proves \eqref{s721} which proves \eqref{s715} as desired. 
	
	The proof of \eqref{s712} is very similar and we sketch the details.  As in the case for $H_0$, we claim that there 
	exist a distorted Fourier basis $\{ \tht(r,\lam^2), \phi(r,\lam^2) \}$ that satisfies
	\begin{align*}
	\begin{split}
	H \tht(r,\lam^2) = \lam^2 \tht(r,\lam^2),& \quad H \phi(r,\lam^2) = \lam^2 \phi(r,\lam^2), \\
	\tht(0,\lam^2) = 1,& \quad \phi(0,\lam^2) = 0, \\
	\tht'(0,\lam^2) = 0,& \quad \phi'(0,\lam^2) = 1,
	\end{split}
	\end{align*}
	and positive measures $\rho_{1}(d\lam) = \om_{1}(\lam)d\lam$
	and $\rho_{2}(d\lam) = \om_{2}(\lam)d\lam$ such that 
	if we define 
	\begin{align*}
	\hat f_{1}(\lam) := \int \tht(r,\lam^2) f(r) dr, \quad \hat f_{2}(\lam) := \int \phi(r,\lam^2) f(r) dr, 
	\quad f \in L^2(\R),
	\end{align*}
	then 
	\begin{align}
	f(r) &= \int_0^\infty \tht(r,\lam^2) \hat f_{1}(\lam) \rho_{1}(d\lam) +
	\int_0^\infty \phi(r,\lam^2) \hat f_{2}(\lam) \rho_{2}(d\lam), \label{s723} \\
	\| f \|^2_{L^2(\R)} &=
	\int_0^\infty |\hat f_{1}(\lam)|^2 \rho_{1}(d\lam) + 
	\int_0^\infty |\hat f_{2}(\lam)|^2 \rho_{2}(d\lam), \label{s724} \\
	\sup_{r \in \R, \lam > 0}&
	\left ( \frac{1 + \lam^2 \la r \ra^2}{\lam^2 \la r \ra^2} \right )
	\left [ 
	|\tht(r,\lam^2)|^2\om_{1}(\lam) + 
	|\phi(r,\lam^2)|^2\om_{2}(\lam)
	\right ] < \infty. \label{s725}
	\end{align}
	Again, the proof of this claim is postponed until the next subsection.  We remark that it is in proving \eqref{s723}, \eqref{s724}, and especially \eqref{s725} that 
	the spectral assumptions are crucial.  By \eqref{s714} and \eqref{s723}, we see that \eqref{s712} follows from showing 
	\begin{align}
	\begin{split}\label{s726}
	\int \Bigl \| 
	V_2(r) \Bigl [
	&\int_0^\infty \bigl (\cos(t\lam) \hat f_1(\lam) + \lam^{-1} \sin(t\lam) \hat g_1(\lam) \bigr) \tht(r,\lam^2) 
	\rho_{0,1}(d\lam) \\
	&+ \int_0^\infty \bigl (\cos(t\lam) \hat f_2(\lam) + \lam^{-1} \sin(t\lam) \hat g_2(\lam) \bigr) \phi(r,\lam^2) 
	\rho_{2}(d\lam) 
	\Bigr ] \Bigr  \|_{L^2(\R)}^2 dt \lesssim \| (\sqrt{H}f,g) \|_{L^2(\R)}^2.
	\end{split}
	\end{align}
	Assume that $g = 0$.  Then, as in the case for $H_0$, the left side of \eqref{s726} becomes after expanding and integrating in $t$
	\begin{align*}
	\int &V_2^2(r) \Bigl [ 
	\int_0^\infty |\hat f_{1}(\lam)|^2  \cos^2(t\lam) |\tht(r,\lam^2)|^2 \om^2_{1}(\lam) \:d\lam + 
	\int_0^\infty |\hat f_{2}(\lam)|^2 \cos^2(t\lam) |\phi(r,\lam^2)|^2 \om^2_{2}(\lam) \:d\lam
	\Bigr ] dr \\
	&\lesssim \int V_2^2(r) \la r \ra^2 \Bigl [ 
	\int_0^\infty \lam^2 |\hat f_{1}(\lam)|^2  \om_{1}(\lam) \:d\lam + 
	\int_0^\infty \lam^2 |\hat f_{2}(\lam)|^2 \om_{2}(\lam) \:d\lam
	\Bigr ] dr \\
	&= \| \sqrt{H} f \|_{L^2(\R)}^2 \int V^2_2(r) \la r \ra^2 dr \\
	&\lesssim \| \sqrt(H) f \|_{L^2(\R)}^2. 
	\end{align*}
	The case $g = 0$ is handled similarly.  This establishes \eqref{s726} which proves \eqref{s712}.  This completes the proof
	of Proposition \ref{p71} modulo the proofs of the claims about the distorted Fourier bases.  We address this in the next subsection. 
\end{proof}

\subsection{The Distorted Fourier Transform} 

In this subsection, we prove the technical statements about the distorted Fourier bases for $H_0$ and $H$ used in the 
previous section.

\begin{ppn}\label{p72}
	Let $H = -\frac{d^2}{dr^2} + V(r)$ be a Schr\"odinger operator on the line where $V \in C^\infty(\R)$ is even and 
	\begin{align}\label{s728}
	V(r) = \frac{2}{r^2} + O(r^{-3}),
	\end{align}
	as $r \rightarrow \pm \infty$ with natural derivative bounds.  Assume that $H$ has no point spectrum and that $0$ is not a resonance of $H$.  Then 
	there exist a distorted Fourier basis $\{ \tht(r,\lam^2), \phi(r,\lam^2) \}$ that satisfies
	\begin{align}
	\begin{split}\label{s728b}
	H \tht(r,\lam^2) = \lam^2 \tht(r,\lam^2),& \quad H \phi(r,\lam^2) = \lam^2 \phi(r,\lam^2), \\
	\tht(0,\lam^2) = 1,& \quad \phi(0,\lam^2) = 0, \\
	\tht'(0,\lam^2) = 0,& \quad \phi'(0,\lam^2) = 1,
	\end{split}
	\end{align}
	and positive measures $\rho_{1}(d\lam) = \om_{1}(\lam)d\lam$
	and $\rho_{2}(d\lam) = \om_{2}(\lam)d\lam$ such that 
	if we define 
	\begin{align*}
	\hat f_{1}(\lam) := \int \tht(r,\lam^2) f(r) dr, \quad \hat f_{2}(\lam) := \int \phi(r,\lam^2) f(r) dr, 
	\quad f \in L^2(\R),
	\end{align*}
	then 
	\begin{align}
	f(r) &= \int_0^\infty \tht(r,\lam^2) \hat f_{1}(\lam) \rho_{1}(d\lam) +
	\int_0^\infty \phi(r,\lam^2) \hat f_{2}(\lam) \rho_{2}(d\lam), \label{s729} \\
	\| f \|^2_{L^2(\R)} &=
	\int_0^\infty |\hat f_{1}(\lam)|^2 \rho_{1}(d\lam) + 
	\int_0^\infty |\hat f_{2}(\lam)|^2 \rho_{2}(d\lam), \label{s730} \\
	\sup_{r \in \R, \lam > 0}&
	\left ( \frac{1 + \lam^2 \la r \ra^2}{\lam^2 \la r \ra^2} \right )
	\left [ 
	|\tht(r,\lam^2)|^2\om_{1}(\lam) + 
	|\phi(r,\lam^2)|^2\om_{2}(\lam)
	\right ] < \infty. \label{s731}
	\end{align}
\end{ppn}

Many of the statements made in Proposition \ref{p72} follow from basic Weyl--Titchmarsch theory for \newline Schr\"odinger operators
on the line.  We recall these basic facts now (see Section 2 of \cite{ges} for a thorough discussion).  Let 
$H = -\frac{d}{dr^2} + V$ with $V \in L^\infty(\R)$ (much less is needed) such that $H$ is in the limit point 
case at $\pm \infty$.  We define 
$\tht(r,z), \phi(r,z)$ to be the fundamental system of solutions to 
\begin{align*}
H f(r) = z f(r), \quad z \in \C,
\end{align*}
such that 
\begin{align}
\begin{split}\label{s731b}
\tht(0,z) = 1,& \quad \phi(0,z) = 0, \\
\tht'(0,z) = 0,& \quad \phi'(0,z) = 1.
\end{split}
\end{align}
By \eqref{s731b}, the Wronskian is computed
\begin{align*}
W(\tht(\cdot,z), \phi(\cdot, z) ) = 1. 
\end{align*}
The condition that $H$ is in the limit point case at $\pm \infty$ implies that for $z \in \C \backslash \R$ there exist unique solutions 
$\psi_{\pm}(r,z)$ to $H f = z f$ that satisfy 
\begin{align*}
\psi_{\pm}(\cdot,z) &\in L^2([0,\pm \infty)), \\
\psi_{\pm}(0,z) &= 1.  
\end{align*}
The condition at $r = 0$ implies that 
\begin{align}\label{s732}
\psi_{\pm}(r,z) = \tht(r,z) + m_{\pm}(z) \phi(r,z)
\end{align}
where $m_{\pm}(z) = W(\tht(\cdot,z), \psi_{\pm}(\cdot,z))$ and  
\begin{align*}
W(\psi_+(\cdot,z), \psi_-(\cdot,z)) = m_-(z) - m_+(z).  
\end{align*}
The functions $m_{\pm}(z)$ can be shown to be Herglotz functions ($\Im z > 0 \implies \Im m_{\pm}(z) > 0$) and are referred
to as the \emph{Weyl--Titchmarsch functions}.  The associated \emph{Weyl--Titchmarsch matrix}
\begin{align}
M(z) :=
\begin{bmatrix}
\frac{1}{m_-(z) - m_+(z)} & \frac{1}{2} \frac{m_-(z) + m_+(z)}{m_-(z) - m_+(z)} \\
\frac{1}{2} \frac{m_-(z) + m_+(z)}{m_-(z) - m_+(z)} & \frac{m_-(z)m_+(z)}{m_-(z) - m_+(z)}
\end{bmatrix} \label{s733}
\end{align}
is a Herglotz matrix.  Thus, there exists a nonnegative $2 \times 2$ matrix--valued measure $\Omega(d\lam)$ such that 
\begin{align*}
M(z) = C + \int_{\R} \left [ \frac{1}{\lam - z} - \frac{\lam}{1 + \lam^2} \right ] \Omega(d\lam),
\end{align*}
where 
\begin{align*}
C^* = C, \quad \int \frac{ \| \Omega(d\lam) \| }{1 + \lam^2} < \infty. 
\end{align*}
The measure $\Omega(d\lam)$ is computed via 
\begin{align*}
\Omega((\lam_1,\lam_2]) = \frac{1}{\pi} \lim_{\de \rightarrow 0^+} \lim_{\e \rar 0^+}
\int_{\lam_1 + \de}^{\lam_2 + \de} \Im M(\lam + i \e) \: d\lam. 
\end{align*}
A consequence of Weyl--Titchmarsch theory is that we have the following distorted Fourier representation for $H$.  

\begin{ppn}\label{p73}
	Let $f,g \in C^\infty_0(\R), F \in C(\R) \cap L^\infty(\R)$.  Let $E(d\lam)$ denote the spectral measure for $H$.  
	Define 
	\begin{align*}
	\hat f_{1}(\lam) := \int \tht(r,\lam) f(r) dr, \quad \hat f_{2}(\lam) := \int \phi(r,\lam) f(r) dr, 
	\end{align*}
	and 
	\begin{align*}
	\hat f(\lam) = (\hat f_1(\lam), \hat f_2(\lam) )^T. 
	\end{align*}
	Then 
	\begin{align*}
	\left (f, F(H) E((\lam_1,\lam_2]) g \right )_{L^2(\R)} = \int_{(\lam_1,\lam_2]} {\hat f}(\lam)^T \Omega(d\lam) \overline{\hat g(\lam)} F(\lam).
	\end{align*}
\end{ppn}

For the free case $V = 0$, we have the following explicit expressions: 
\begin{align}
\begin{split}\label{s734}
\tht(r,z) = \cos (r z^{1/2}),& \quad \phi(r,z) = \frac{\sin (r z^{1/2})}{z^{1/2}}, \\
\psi_{\pm}(r,z) = e^{\pm i r z^{1/2}},& \quad m_{\pm}(z) = \pm i z^{1/2} \\
\Omega(d\lam) = \frac{1}{2\pi} \chi_{(0,\infty)}(\lam)& 
\begin{bmatrix}
\lam^{-1/2} & 0 \\
0 & \lam^{1/2}
\end{bmatrix}.
\end{split}
\end{align}
This leads to the usual Fourier transform on the line.  

\begin{proof}[Proof of Proposition \ref{p72}]
	The decay of $V$ at $\pm \infty$ implies that $H = -\frac{d^2}{dr^2} + V$ is in the limit point case at $\pm \infty$ (see \cite{ges}).  The decay
	of $V$ 
	and the assumption that $H$ has no point spectrum imply that $\sigma(H) = [0,\infty)$ and that the spectrum 
	is purely absolutely continuous.  By Proposition \ref{p73}, this implies that the matrix valued measure $\Omega(d\lam)$ is supported in 
	$[0,\infty)$.   
	Since $V$ is even, we have by \eqref{s728b} 
	\begin{align*}
	\tht(-r,\lam) = \tht(r,\lam), \quad \phi(-r,\lam) = -\phi(r,\lam), \quad \psi_-(r,\lam) = \psi_+(-r,\lam),
	\end{align*}
	so that $m_-(\lam) = -m_+(\lam)$.  We recall from the previous section that the Jost solutions $f_{\pm}(r,\lam)$ are the 
	unique solutions to $H f = \lam^2 f$ such that $f_{\pm}(r,\lam) \sim e^{\pm i r \lam}$ as $r \rar \pm \infty$, and that 
	for $\lam \neq 0$, $W(f_+(\cdot,\lam),f_-(\cdot,\lam)) \neq 0$.  Then 
	\begin{align*}
	f_{-}(r,\lam) &= f_+(-r,\lam), \\
	W(f_+(\cdot,\lam),f_-(\cdot,\lam)) &= -2f_+(0,\lam)f_+'(0,\lam), \\
	\psi_+(r,\lam^2) &= \frac{f_+(r,\lam)}{f_+(0,\lam)}, \\
	m_+(\lam^2) &= \frac{f_+'(0,\lam)}{f_+(0,\lam)}.
	\end{align*}
	The matrix \eqref{s733} satisfies
	\begin{align*}
	M(\lam^2) &= 
	\begin{bmatrix}
	-\frac{1}{2} \frac{f_+(0,\lam)}{f_+'(0,\lam)} & 0 \\
	0 & -\frac{1}{2} \frac{f_+'(0,\lam)}{f_+(0,\lam)}
	\end{bmatrix} \\
	&= 
	\begin{bmatrix}
	\frac{1}{2} \frac{W(f_+(\cdot,\lam), \phi(\cdot,\lam^2))}{W(f_+(\cdot,\lam), \tht(\cdot,\lam^2))} & 0 \\
	0 & -\frac{1}{2} \frac{W(f_+(\cdot,\lam), \tht(\cdot,\lam^2))}{W(f_+(\cdot,\lam), \phi(\cdot,\lam^2))}.
	\end{bmatrix} 
	\end{align*}
	Thus, 
	\begin{align}
	\Omega(d\lam^2) &= 
	\begin{bmatrix}
	\rho_1(d\lam) & 0 \\
	0 & \rho_2(d\lam)
	\end{bmatrix} \label{s735}
	\end{align}
	where 
	\begin{align}
	\begin{split}\label{s736}
	\rho_1(d\lam) &:= 
	\frac{1}{\pi} \lam \Im \left [ \frac{W(f_+(\cdot,\lam), \phi(\cdot,\lam^2))}
	{W(f_+(\cdot,\lam), \tht(\cdot,\lam^2))} \right ] d\lam, \\
	\rho_2(d\lam) &:= 
	-\frac{1}{\pi} \lam \Im \left [ 
	\frac{W(f_+(\cdot,\lam), \tht(\cdot,\lam^2))}{W(f_+(\cdot,\lam), \phi(\cdot,\lam^2))}
	\right ]d \lam.
	\end{split}
	\end{align}
	By Proposition \ref{p73}, \eqref{s735} and \eqref{s736} imply \eqref{s729} and \eqref{s730}.  It remains to prove
	\eqref{s731}.  As in Section 3, the main difficulty is encountered when considering $0 < \lam \ll 1$.  Indeed, 
	it is not hard to show that if $\lam$ is bounded away from 0, $\lam \gtrsim 1$, then the distorted Fourier basis 
	$\theta(r,\lam^2),\phi(r,\lam^2)$ and measure $\Omega(d\lam^2)$ in \eqref{s735} are approximated to leading order by 
	the free case \eqref{s734}.  For the free case, \eqref{s731} holds (for $\lam \gtrsim 1$). Thus, \eqref{s731} holds in the 
	perturbed case for $\lam \gtrsim 1$.  We omit the details, and instead focus on establishing 
	\eqref{s731} in the case $0 < \lam \ll 1$.  To establish \eqref{s731} in the small $\lam$ regime, we use the scattering theory summarized in Section 3 to derive 
	asymptotic expansions for $\theta(r,\lam^2), \phi(r,\lam^2), \rho_1(d\lam)$, and $\rho_2(d\lam)$.  The upcoming calculations
	will freely use the notation from Section 3.
	
	We first consider the zero energy equation.  Let $\tht_0(r), \phi_0(r)$ be the fundamental system for $H f = 0$ such that 
	\begin{align}
	\begin{split}\label{s737}
	\tht_0(0) = 1,& \quad \phi_0(0) = 0, \\
	\tht_0'(0) = 0,& \quad \phi'_0(0) = 1. 
	\end{split}
	\end{align}
	Then 
	\begin{align}
	\begin{split}\label{s738}
	\phi_0(r) &= a_0 u^+_0(r) + a_1 u^+_1(r), \\
	\tht_0(r) &= b_0 u^+_0(r) + b_1 u^+_1(r), \\
	\end{split}
	\end{align}
	where, we recall that, the solutions $u^+_j(r)$ satisfy $H u^+_j(r) = 0$ and 
	\begin{align}
	\begin{split}\label{s738b}
	u^+_0(r) &= \frac{1}{3} r^2 + O(r), \\
	u^+_1(r) &= r^{-1} + O(r^{-2}),
	\end{split}
	\end{align}
	as $r \rar \infty$ (see Lemma \ref{scattlem1}). Since $W(\tht_0,\phi_0) = 1 = W(u_1^+,u_0^+)$, we conclude that 
	\begin{align}\label{s739}
	a_0 b_1 - a_1 b_0 = 1.
	\end{align}
	Since $\phi_0$ and $\tht_0$ are odd and even respectively, the assumption that 0 is not a resonance implies the crucial
	condition that 
	\begin{align}\label{s740}
	a_0 \neq 0 \quad \mbox{and} \quad b_0 \neq 0.
	\end{align}
	We now perturb in small $\lam$.  We claim that the smooth fundamental system $\tht(r,\lam^2), \phi(r,\lam^2)$ that satisfies
	\eqref{s728b} also satisfies
	\begin{align}
	\begin{split}\label{s741}
	\phi(r,\lam^2) = \phi_0(r) + O(\lam^2 \la r \ra^2 r^2),\\
	\tht(r,\lam^2) = \tht_0(r) + O(\lam^2 \la r \ra^2 r^2),    
	\end{split}
	\end{align}
	for $0 \leq r \leq \lam^{-1}$. The $O(\cdot)$ terms are real--valued and satisfy natural derivative bounds.  Indeed, by variation of constants, we can write $\phi(r,\lam^2)$ as a solution to 
	\begin{align}\label{s742}
	\phi(r,\lam^2) = \phi_0(r) + \lam^2 \int_0^r \left [ u^+_0(r) u^+_1(\rho) -  u^+_0(\rho) u^+_1(r)\right ] \phi(\rho,\lam^2) d\rho.
	\end{align}
	If we define $\tilde \phi(r,\lam^2) = \la r \ra^{-2} \phi(r,\lam^2)$ and $K(r,\rho,\lam) = \lam^2 \la \rho \ra^2 \la r \ra
	^{-2} \left [ u^+_0(r) u^+_1(\rho) -  u^+_0(\rho) u^+_1(r)\right ]$, then \eqref{s742} takes the form of the 
	Volterra equation 
	\begin{align}\label{s743}
	\tilde \phi(r,\lam^2) = \la r \ra^{-2} \phi_0(r) + \int_0^r K(r,\rho,\lam) \tilde \phi(\rho,\lam^2) d\rho.
	\end{align}
	By \eqref{s738b}, if $0 < \rho < r$, then the kernel satisfies
	\begin{align*}
	|K(r,\rho,\lam)| \lesssim \lam^2 \la \rho \ra.
	\end{align*}
	Thus, 
	\begin{align*}
	\int_0^{\lam^{-1}} \sup_{r > \rho} |K(r,\rho,\lam)| d\rho \lesssim 1,
	\end{align*}
	which implies that the Volterra iterates for \eqref{s743} converge on $[0,\lam^{-1}]$ to a unique solution 
	$\tilde \phi(r,\lam^2)$ satisfying 
	\begin{align*}
	\tilde \phi(r,\lam^2) = \la r \ra^{-2} \phi_0(r) + O(\lam^2 r^2). 
	\end{align*}
	This proves \eqref{s741} for $\phi(r,\lam^2)$.  An identical argument proves \eqref{s741} for $\tht(r,\lam^2)$ as well. 
	By Lemma \ref{scattlem2} there exists a fundamental system $u^+_0(r,\lam), u^+_1(r,\lam)$ for $Hf = \lam^2 f$ such that 
	$W(u^+_1(\cdot,\lam), u^+_0(\cdot,\lam) ) = 1$ and for $j = 0,1$
	\begin{align}\label{s744}
	u^+_j(r,\lam) = u^+_j(r)(1 + O(\la r \ra^2 \lam^2)), \quad r \in [r_0,\e_0 \lam^{-1}],
	\end{align}
	for some fixed $r_0,\e_0 > 0$.  Similar to \eqref{s738}, we can write 
	\begin{align}
	\begin{split}\label{s745}
	\phi(r,\lam^2) &= a_0(\lam) u^+_0(r,\lam) + a_1(\lam) u^+_1(r,\lam), \\
	\tht(r,\lam^2) &= b_0(\lam) u^+_0(r,\lam) + b_1(\lam) u^+_1(r,\lam),
	\end{split}
	\end{align}
	with $a_0(\lam)b_1(\lam) - a_1(\lam)b_0(\lam) = 1$.  We claim that 
	\begin{align}
	\begin{split}\label{s746}
	a_0(\lam) = a_0 + O(\lam^2),& \quad a_1(\lam) = a_1 + O(\lam^2), \\ 
	b_0(\lam) = b_0 + O(\lam^2),& \quad b_1(\lam) = b_1 + O(\lam^2), 
	\end{split}
	\end{align}
	as $\lam \rar 0$ where $a_0,a_1,b_0,$ and $b_1$ are as in \eqref{s738}.  Indeed, using \eqref{s741} and \eqref{s744},
	we evaluate the Wronskian at $r = r_0$ and deduce that 
	\begin{align*}
	a_0(\lam) &= W(u^+_1(r,\lam), \phi(r,\lam^2) \\
	&= W\left (u^+_1(r)(1 + O(\lam^2 \la r \ra^2), a_0 u^+_0(r) + a_1 u^1_+(r) + O(\lam^2 \la r \ra^2 r^2) \right ) \\
	&= a_0 + O(\lam^2). 
	\end{align*}
	The computation for $a_1,b_0,$ and $b_1$ are similar so that \eqref{s746} follows. We are now in a position to 
	derive asymptotics for 
	\begin{align*}
	\om_1(\lam) &:= 
	\frac{1}{\pi} \lam \Im \left [ \frac{W(f_+(\cdot,\lam), \phi(\cdot,\lam^2))}
	{W(f_+(\cdot,\lam), \tht(\cdot,\lam^2))} \right ], \\
	\om_2(\lam) &:= 
	-\frac{1}{\pi} \lam \Im \left [ 
	\frac{W(f_+(\cdot,\lam), \tht(\cdot,\lam^2))}{W(f_+(\cdot,\lam), \phi(\cdot,\lam^2))}
	\right ].
	\end{align*}
	By Lemma \ref{scattlem3}, we have 
	\begin{align}
	\begin{split}\label{s747}
	W(f_+(\cdot,\lam),u^+_1(\cdot,\lam) ) &= \al_0^+(\lam) + i \al_1^+(\lam), \\ 
	W(f_+(\cdot,\lam),u^+_0(\cdot,\lam) ) &= \beta_0^+(\lam) + i \beta_1^+(\lam), \\ 
	\end{split}
	\end{align}
	where 
	\begin{align}
	\begin{split}\label{s748}
	\al_0^+(\lam) = \lam^2( \al_0 + O(\lam^\e) ),& \quad \al_1^+(\lam) = O(\lam^{2 - 2\e}), \\  
	\beta_0^+(\lam) = O(\lam^{-1+4\e}),& \quad \beta_1^+(\lam) = \lam^{-1}(\beta_0 + O(\lam^\e)),  
	\end{split}
	\end{align}
	for all $0 < \e < \e_0$.  The constants $\al_0$ and $\beta_0$ in \eqref{s748} are positive.  From \eqref{s745} and 
	\eqref{s747}, we conclude that 
	\begin{align*}
	W(f_+(\cdot,\lam), \phi(\cdot,\lam^2)) &= A_0(\lam) + i A_1(\lam), \\
	W(f_+(\cdot,\lam), \tht(\cdot,\lam^2)) &= B_0(\lam) + i B_1(\lam), 
	\end{align*}
	where 
	\begin{align}
	\begin{split}\label{s749}
	A_0(\lam) &= a_0(\lam) \beta^+_0(\lam) + a_1(\lam) \al^+_0(\lam), \\ 
	A_1(\lam) &= a_0(\lam) \beta^+_1(\lam) + a_1(\lam) \al^+_1(\lam), \\ 
	B_0(\lam) &= b_0(\lam) \beta^+_0(\lam) + b_1(\lam) \al^+_0(\lam), \\ 
	B_1(\lam) &= b_0(\lam) \beta^+_1(\lam) + b_1(\lam) \al^+_1(\lam). \\ 
	\end{split}
	\end{align}
	Then 
	\begin{align}
	\Im \left [ \frac{W(f_+(\cdot,\lam), \phi(\cdot,\lam^2))}
	{W(f_+(\cdot,\lam), \tht(\cdot,\lam^2))} \right ] = \frac{A_1 B_0 - A_0 B_1}{B_0^2 + B_1^2}. \label{s750}
	\end{align}
	By \eqref{s748} and the condition that $a_0(\lam)b_1(\lam) - a_1(\lam)b_0(\lam) = 1$, we conclude that  
	\begin{align}
	A_1 B_0 - A_0 B_1 &= \beta^+_1(\lam)\al_0^+(\lam) - \al_1^+(\lam)\beta_0^+(\lam) \notag \\
	&= \lam \left (\al_0 \beta_0 + O(\lam^\e) \right ). \label{s751}
	\end{align}
	By \eqref{s748} and \eqref{s746}
	\begin{align}
	B_0^2 + B_1^2 = \lam^{-2} \left ( b_0^2 \beta_0^2 + O(\lam^\e) \right ). \label{s752} 
	\end{align}
	Thus, \eqref{s750}, \eqref{s751}, and \eqref{s752} yield
	\begin{align}\label{s753}
	\Im \left [ \frac{W(f_+(\cdot,\lam), \phi(\cdot,\lam^2))}
	{W(f_+(\cdot,\lam), \tht(\cdot,\lam^2))} \right ] = \lam^3 \frac{\al_0 \beta_0 + O(\lam^\e)}{b_0^2 \beta_0^2 + O(\lam^\e)}, 
	\end{align}
	as $\lam \rar 0^+$.  Similarly, 
	\begin{align}\label{s754}
	- \Im \left [ \frac{W(f_+(\cdot,\lam), \tht(\cdot,\lam^2))}
	{W(f_+(\cdot,\lam), \phi(\cdot,\lam^2))} \right ] = \lam^3 \frac{\al_0 \beta_0 + O(\lam^\e)}{a_0^2 \beta_0^2 + O(\lam^\e)}. 
	\end{align}
	The crucial nonresonant condition \eqref{s740} implies that \eqref{s753} and \eqref{s754} are both $O(\lam^3)$.  In summary, 
	we have shown that the measures $\rho_1(d\lam) = \om_1(\lam)d\lam$ and $\rho_2(d\lam) = \om_2(\lam)d\lam$
	in \eqref{s736} have weights
	that satisfy 
	\begin{align}\label{s755}
	\om_1(\lam) = O(\lam^4), \quad \om_2(\lam) = O(\lam^4). 
	\end{align}

	We now prove \eqref{s731} using the asymptotics from the previous paragraph.  The expressions \eqref{s741}, \eqref{s738}, \eqref{s738b}, and \eqref{s755} imply that 
	\begin{align}\label{s755b}
	\left ( \frac{1 + \lam^2 \la r \ra^2}{\lam^2 \la r \ra^2} \right )
	\left [ 
	|\tht(r,\lam^2)|^2\om_{1}(\lam) + 
	|\phi(r,\lam^2)|^2\om_{2}(\lam)
	\right ] \lesssim 1, \quad r \in [0,\lam^{-1}].  
	\end{align}
	We now consider the case $r \geq \lam^{-1}$.  We first recall that 
	\begin{align*}
	W(f_+(\cdot,\lam), \overline{f_+(\cdot,\lam)}) = -2i\lam \neq 0, 
	\end{align*}
	for $\lam > 0$.  Thus, we can write 
	\begin{align}\label{s756}
	\phi(r,\lam^2) = c(\lam)f_+(r,\lam) + d(\lam) \overline{f_+(r,\lam)}.
	\end{align}
	Since $\phi(r,\lam^2)$ is real--valued, $d(\lam) = \overline{c(\lam)}$.  Note that 
	\begin{align*}
	W(\phi(\cdot,\lam^2),\overline{f_+(\cdot,\lam)}) = -2i\lam c(\lam), 
	\end{align*}
	so that by \eqref{s747}--\eqref{s749}, we conclude that 
	\begin{align}\label{s757}
	c(\lam) = \frac{1}{2i\lam} W(f_+(\cdot,\lam), \phi(\cdot,\lam^2)) = O_{\C}(\lam^{-2}).
	\end{align}
	By Lemma \ref{jostlem} we have 
	\begin{align}\label{s758}
	f_+(r,\lam) = e^{ir\lam}(1 + O(\lam^{-1} \la r \ra^{-1})), \quad r \geq \lam^{-1}.
	\end{align}
	From \eqref{s755}--\eqref{s758}, we conclude that 
	\begin{align}\label{s759}
	\left ( \frac{1 + \lam^2 \la r \ra^2}{\lam^2 \la r \ra^2} \right ) 
	|\phi(r,\lam^2)|^2\om_{2}(\lam) \lesssim 1, \quad r \geq \lam^{-1}.
	\end{align}
	By the exact same arguments, 
	\begin{align}\label{s760}
	\left ( \frac{1 + \lam^2 \la r \ra^2}{\lam^2 \la r \ra^2} \right ) 
	|\tht(r,\lam^2)|^2\om_{1}(\lam) \lesssim 1, \quad r \geq \lam^{-1}.
	\end{align} 
	In summary, we have shown that for $0 < \lam \ll 1$,
	\begin{align*}
	\left ( \frac{1 + \lam^2 \la r \ra^2}{\lam^2 \la r \ra^2} \right )
	\left [ 
	|\tht(r,\lam^2)|^2\om_{1}(\lam) + 
	|\phi(r,\lam^2)|^2\om_{2}(\lam)
	\right ] \lesssim 1, \quad r \in \R.
	\end{align*}
	This proves \eqref{s731} and concludes the proof of Proposition \ref{p72}.
\end{proof}




\section{Small Data Theory and Concentration--Compactness}

In this section we use the tools developed in the previous sections to initiate the study of the nonlinear evolution introduced in the previous section:
\begin{align}
\begin{split}\label{s81}
&\p_t^2 u - \Delta_g u + V(r) u = N(r,u), \quad (t,r) \in \R \times \R, \\
&\vec u(0) = (u_0,u_1) \in \h,
\end{split}
\end{align}
where $\h := \h(\R; (r^2 + 1)^2 dr)$, $-\Delta_g$ is the (radial) Laplace operator on the $5d$ wormhole $\M^5$, and $V(r)$ and 
$N(r,u)$ are given in \eqref{s45} and \eqref{s46}.  In particular, we begin our proof of 
Theorem \ref{t41}, i.e. every solution to \eqref{s81} is global and 
scatters to free waves on $\M^5$. 

\subsection{Small Data Theory}

As summarized in the introduction, the proof of Theorem \ref{t41} (which we have shown in Section 4 is equivalent to Theorem \eqref{t01}) uses the powerful concentration--compactness/rigidity 
methodology introduced by Kenig and Merle in their study of energy--critical dispersive equations \cite{km06} \cite{km08}.  The methodology
is split up into three main steps and proceeds by contradiction.  In the first step, we establish small data global well--posedness 
and scattering for \eqref{s44}.  In particular, we establish Theorem \ref{t41} for small data $(u_0,u_1)$.  In the 
second step, the first step and a concentration--compactness argument shows that the \emph{failure} of Theorem \ref{t41} implies
that that there exists a nonzero \lq critical element' $u_*$;  a minimal non--scattering global solution to \eqref{s44}. 
The minimality of $u_*$ imposes the following compactness property on $u_*$: the trajectory 
\begin{align*}
K = \left \{ \vec u_*(t) : t \in \R \right \} 
\end{align*}
is precompact in $\h$. In the third and final step, we establish  the following rigidity theorem: every solution $u$ with $\{ \vec u(t) : t \in \R \}$
precompact in $\h$ must identically 0.  This contradicts the second step which implies that Theorem \ref{t41} holds.  In this section 
we complete the first two steps in the program: small data theory and concentration--compactness.  These steps follow from, by now, standard 
arguments using the Strichartz estimates for $\p_t^2 - \Delta_g + V$ established in Section 4. 

We first establish a global well--posedness and small data theory for \eqref{s44}.  This follows from a contraction mapping argument
using Strichartz estimates established in Proposition \ref{p71} for the inhomogeneous wave equation with potential
\begin{align}
\begin{split}\label{s412}
&\p_t^2 u - \Delta_g u + V(r) u = h(t,r), \quad (t,r) \in \R \times \R, \\
&\vec u(0) = (u_0,u_1) \in \h.
\end{split}
\end{align}
Here, as in the previous section, the potential $V$ is given by 
\begin{align*}
V(r) = \la r \ra^{-4} + 2 \la r \ra^{-2} \left ( \cos 2Q - 1 \right ),
\end{align*}
where $Q$ is the unique harmonic map of degree $n$.  To see that $V$ satisfies the hypotheses in Proposition \ref{p71}, we note that
by Proposition \ref{harm}, we only need to verify the spectral assumptions are satisfied.  
This was shown in \cite{biz2}, and we recall the argument.  We have the relation 
\begin{align}\label{conj relation}
\la r \ra^2 (-\Delta_g + V) \la r \ra^{-2} = H, 
\end{align}
where $H$ is the Schr\"odinger operator on $L^2(\R)$ given by  
\begin{align*}
H = -\frac{d^2}{dr^2} + \frac{2}{r^2 + 1} + V(r).
\end{align*}
We need to check that $H$ has no point spectrum and that 0 is not a resonance for $H$. 
First, we note that the decay of the potential $\frac{2}{r^2 + 1} + V(r)$ implies that $\sigma_{ac}(H) =  [0,\infty)$ and there 
are no embedded eigenvalues.  If $Q \equiv 0$ (the $n = 0$ case), the fact that $H$ has no eigenvalues in $(-\infty,0]$ follows from 
the fact that the potential term $2 \la r \ra^{-2} + V(r)$ is nonnegative.  For the case $n \in \N$, 
multiply the equation
\begin{align*}
\p_r^2 Q + \frac{2r}{r^2 + 1} \p_r Q - \frac{\sin 2Q}{r^2 + 1} = 0 
\end{align*}
by $r^2 + 1$ and differentiate to conclude that 
\begin{align*}
\tilde H ( \la r \ra^2 Q'(r) ) = 0,  
\end{align*}
where $\tilde H = H - \la r \ra^4$.  By Proposition \ref{harm} the harmonic map $Q$ is strictly increasing on $\R$ so that $\la r \ra^2 Q'(r) > 0$
for all $r \in \R$.  By Sturm oscillation theory we conclude that $\tilde H$ has no negative eigenvalues and that 
$\sigma(\tilde H) = [0,\infty)$.  In particular, we have for all $h \in C^\infty_0(\R)$ 
\begin{align}\label{coercivity}
( H h, h )_{L^2(\R)} = (\tilde H h , h )_{L^2(\R)} + \int |h|^2 \la r \ra^{-4} dr \geq \int |h|^2 \la r \ra^{-4} dr.
\end{align}
By a variational principle, the previous implies that $H$ has no eigenvalues in $(-\infty,0]$, and thus, $H$ has no 
point spectrum.  We now check that 0 is not a resonance of $H$.  
The asymptotics of the potential $2 \la r \ra^{-2} + V(r)$ imply that $0$ is a resonance if and only 
if $0$ is an eigenvalue (see Lemma \ref{scattlem1} and Definition \ref{scattdef}).  Thus, $0$ is not a resonance of $H$.  We conclude that 
$V$ satisfies the hypotheses of Proposition \ref{p71}. 

For $I \subseteq \R$, we denote the following spacetime norms
\begin{align*}
\| u \|_{S(I)} := \| u \|_{L^3_t L^6_x(I)}, \quad \| u \|_{W(I)} := \| u \|_{L^3_t \dot W^{1/2,L^3}_x(I)}, \quad
\| h \|_{N(I)} := \| F \|_{L^1_tL^2_x(I) + L^{3/2}_t \dot W^{1/2,3/2}_x(I)}.
\end{align*}
By the previous discussion and Proposition \ref{p71}, a solution $u$ to \eqref{s412} satisfies the estimate
\begin{align}\label{s413}
\| u \|_{W(I)} \lesssim \| \vec u(0) \|_{\h} + 
\| h \|_{N(I)}.
\end{align}
We claim that if $f \in C^\infty_0(\M^5)$ is radial, then 
\begin{align*}
\| f \|_{L_x^6} \lesssim \| |\nabla|^{1/2} f \|_{L_x^3}.  
\end{align*}
Indeed, by the fundamental theorem of calculus, we have 
\begin{align*}
|f(r)| \lesssim \la r \ra^{-2/3} \left ( \int |f'(r)|^3 (r^2 + 1) dr \right )^{1/3} = \la r \ra^{-2/3} \| \nabla f \|_{L_x^3}. 
\end{align*}
Thus, $\| f \|_{L_x^\infty} \lesssim \| \nabla f \|_{L_x^3}$.  Interpolating this estimate with the trivial embedding 
$L_x^3 \hookrightarrow L_x^3$ yields the desired bound $\| f \|_{L_x^6} \lesssim \| |\nabla|^{1/2} f \|_{L_x^3}.$  Thus, we have 
that the `scattering norm' $\| \cdot \|_{S(I)}$ is weaker than the norm $\| \cdot \|_{W(I)}$.  This 
fact and \eqref{s413} imply that a solution to \eqref{s412} satisfies the Strichartz estimate 
\begin{align}\label{s414}
\| u \|_{S(\R)} + \| u \|_{W(\R)} \lesssim \| \vec u(0) \|_{\h} + \| h \|_{N(\R)}.
\end{align}
We now use \eqref{s414} and standard contraction mapping arguments to establish the following global well--posedness
and small data theory. We remark here that it will be important in later applications to use the weaker norm $\| \cdot \|_{S(I)}$ 
along with the norm $\| \cdot \|_{W(\R)}$ when establishing the small data scattering.

\begin{ppn}\label{p42}
	For every $(u_0,u_1) \in \h$, there exists a unique global solution $u$ to \eqref{s81}
	such that $\vec u(t) \in C(\R; \h) \cap L^\infty(\R; \h)$.  A solution $u$ scatters to a free wave on $\M^5$ as $t \rar \infty$
	if and only if 
	\begin{align*}
	\| u \|_{S(\R)} < \infty. 
	\end{align*}
	Here, scattering to a free wave on $\M^5$ as $t \rar \infty$ means that there exists a solution $v_L$ to \eqref{s412} with $V 
	\equiv h \equiv 0$ such that 
	\begin{align*}
	\lim_{t \rar \infty} \| \vec u(t) - \vec v_L(t) \|_{\h} = 0.
	\end{align*}
	A similar characterization of $u$ scattering to a free wave on $\M^5$ as $t \rar -\infty$ also holds.  Moreover, there 
	exists $\delta > 0$ such that if $\| \vec u(0) \|_{\h} < \delta$, then 
	\begin{align}\label{s415a}
	\| \vec u \|_{L^\infty_t \h} + \| u \|_{S(\R)} + \| u \|_{W(\R)} \lesssim \| \vec u(0) \|_{\h} < \delta. 
	\end{align}
\end{ppn}

\begin{proof}
	We first show that for every $(u_0,u_1) \in \h$, there exists a unique global solution $\vec u(t) \in C(\R; \h) \cap 
	L^\infty(\R; \h)$ to \eqref{s81} with $\vec u(0) = (u_0,u_1)$.  Denote the propagator for the free wave equation on $\M^5$
	by $S(t)$, i.e. $S(t)(u_0,u_1)$ solves \eqref{s412} with $V \equiv h \equiv 0$. Denote the propagator for the free wave equation on 
	$\M^5$ with potential $V$ by $S_V(t)$, i.e. $S_V(t)(u_0,u_1)$ solves \eqref{s412} with $h \equiv 0$.  Let 
	\begin{align}
	\mathcal E_V(f,g) := \frac{1}{2} \int \left ( |g|^2 + |\p_r f|^2 + V |f|^2 dr \right ) (r^2 + 1)^2 dr
	\end{align}
	denote the conserved energy associated to $S_V$.  Using the coercivity bound
	\eqref{coercivity} it is not hard to conclude that
	\begin{align}\label{equ norms}
	\| \partial_r f \|_{L^2(\R; (r^2 + 1)^2)} \simeq \left \| \sqrt{-\Delta_g + V} f \right \|_{L^2(\R; (r^2 + 1)^2)}
	\end{align} 
	for all radial $f$ so that 
	\begin{align}\label{s415} 
	\| (f,g) \|_{\h}^2 \simeq \mathcal E_V(f,g)
	\end{align}
	for all radial $f,g$.  Indeed, by the decay of $V$ and the Strauss estimate \eqref{s411s} we have 
	\begin{align*}
	\left \| \sqrt{-\Delta_g + V} f \right \|_{L^2(\R; (r^2 + 1)^2)}^2 &= \int \bigr ((-\Delta_g f) f + V |f|^2 \bigr ) (r^2 + 1)^2 dr \\
	&= \| \p_r f \|_{L^2(\R; (r^2 +1)^2 dr)}^2 + \int V |f|^2 (r^2 +1)^2 dr \\
	&\lesssim \|\p_r f \|_{L^2(\R; (r^2 + 1)^2 dr)}^2.
	\end{align*}
	We now note that by the second equality above and the decay of $V$ we have
	\begin{align*}
	\| \p_r f \|_{L^2(\R; (r^2 +1)^2 dr)}^2 \lesssim 
	\left \| \sqrt{-\Delta_g + V} f \right \|_{L^2(\R; (r^2 + 1)^2)}^2 
	+ \int |f|^2 dr. 
	\end{align*}
	By \eqref{conj relation} and \eqref{coercivity} (applied to $h = (r^2 + 1)f$) we see that 
	\begin{align*}
	\int |f|^2 dr \lesssim \left \| \sqrt{-\Delta_g + V} f \right \|_{L^2(\R; (r^2 + 1)^2)}^2
	\end{align*}
	whence 
	\begin{align*}
	\| \p_r f \|_{L^2(\R; (r^2 +1)^2 dr)}^2 \lesssim 
	\left \| \sqrt{-\Delta_g + V} f \right \|_{L^2(\R; (r^2 + 1)^2)}^2.
	\end{align*}
	This proves \eqref{equ norms}.
	
	We write the nonlinear equation \eqref{s81} in Duhamel form as 
	\begin{align*}
	u(t) = S_V(t)(u_0,u_1) + \int_0^t S_V(t-s)\left ( 0, F(\cdot, u(s)) + G(\cdot, u(s)) \right ) ds. 
	\end{align*}
	Using a simple energy estimate, \eqref{s48a}, \eqref{s48b}, \eqref{s49}, \eqref{s411s}, and \eqref{s415}, we obtain the following
	a--priori estimate for a solution $\vec u(t) \in C([0,T]; \h)$ to \eqref{s81}: for $t \in [0,T]$
	\begin{align}
	\begin{split}\label{s416}
	\| \vec u(t) \|_{\h} &\lesssim \| \vec u(0) \|_{\h} + \int_0^T \| F(\cdot, u(s)) + G(\cdot, u(s) \|_{L^2} ds \\
	&\lesssim \| \vec u(0) \|_{\h} + T \left ( \| \vec u \|_{L^\infty_t([0,T]; \h)}^2 +
	\| \vec u \|_{L^\infty_t([0,T]; \h)}^3 \right ).
	\end{split}
	\end{align}
	By a contraction mapping argument based on \eqref{s416} and the conservation of energy for \eqref{s41}, we conclude that 
	there exists a unique global solution $\vec u(t) \in C(\R; \h) \cap L^\infty(\R; \h)$ to \eqref{s81}. 
	
	We now prove the scattering criterion and small data scattering.  Note that every solution $\vec u(t) 
	\in C(\R; \h)$ to \eqref{s81} satisfies $\| u \|_{S(I)} + \| u \|_{W(I)} < \infty$ for all $I \Subset \R$.
	Indeed, by \eqref{s411s} we have $\| u \|_{L^6_x} \lesssim \| \nabla u \|_{L^2_x}$ whence by interpolation we have 
	$\| u \|_{\dot W_x^{1/2,3}} \lesssim \| \nabla u \|_{L^2_x}$.  We first prove the 
	small data scattering estimate \eqref{s415a} as this will also illustrate the validity of the scattering criterion.  Let 
	$u$ be a solution to \eqref{s81} and let $I \subset \R$.  We first note that by the Leibniz rule for 
	Sobolev spaces (see \cite{coul} for asymptotically conic manifolds), we have 
	\begin{align*}
	\| (\la r \ra^{-1} \sin 2Q) u^2 \|_{\dot W^{1/2,3/2}_x} 
	\lesssim \| \la r \ra^{-1} \sin 2Q \|_{\dot W^{1/2,3}_x} \| u^2 \|_{L^3_x} + \| \la r \ra^{-1} \sin 2Q \|_{L^6_x}
	\| u^2 \|_{\cdot W^{1/2,2}_x} \lesssim \| u \|_{L^6_x}^2 + \| u \|_{L^6_x} \| u \|_{\cdot W^{1/2,3}_x},
	\end{align*}
	whence by H\"older's inequality in time we have 
	\begin{align}\label{u2est}
	\| (\la r \ra^{-1} \sin 2Q) u^2 \|_{L^{3/2}_t \dot W^{1/2,3/2}_x(I)} \lesssim \| u \|_{S(I)}^2 + 
	\| u \|_{S(I)} \| u \|_{W(I)}.
	\end{align}
	Then by \eqref{s48a}, \eqref{s48b}, \eqref{s49}, the Strichartz estimate 
	\eqref{s414}, and \eqref{u2est} we have  
	\begin{align*}
	\| u \|_{S(I)} + \| u \|_{W(I)} &\lesssim \| \vec u(0) \|_{\h} + \| N(\cdot, u) \|_{N(I)} \\
	&\lesssim \| \vec u(0) \|_{\h} + \| F(\cdot, u ) \|_{N(I)} + \| G(\cdot, u) \|_{N(I)} \\
	&\lesssim \| \vec u(0) \|_{\h} + + \| (\la r \ra^{-1} \sin 2Q ) u^2 \|_{L^{3/2}_t \dot W^{1/2,3/2}_x(I)} + 
	\| \la r \ra^{-1} u^4 \|_{L^1_t L^2_x(I)} + 
	\| |u|^3 \|_{L^1_t L^2_x(I)}  \\
	&\lesssim \| \vec u(0) \|_{\h} + \| u \|_{S(I)} \| u \|_{W(I)} + \| u \|^2_{S(I)} + 
	\| \vec u \|_{L^\infty_t \h} \| u \|^3_{S(I)} + \| u \|_{S(I)}^3.
	\end{align*}
	By a standard continuity argument, there exists $\delta > 0$ such that if $\| \vec u(0) \|_{\h} < \de$ then 
	$\| \vec u \|_{L^\infty_t \h} + \| u \|_{S(\R)} + \| u \|_{W(\R)} \lesssim \| \vec u(0) \|_{\h}$ as desired.  A simple variant of the above argument also shows that 
	if $\| u \|_{S(0,\infty)} < \infty$, then 
	\begin{align*}
	w_L(0) = \vec u(0) + \int_0^\infty S_V(-s)(0,N(\cdot, u(s))) ds 
	\end{align*}
	converges in $\h$.  Thus, by Duhamel we conclude that 
	\begin{align}
	\vec u(t) = \vec S_V(t)w_L(0) + o_{\h}(1), \label{s417}
	\end{align}
	as $t \rar \infty$.  To extract a free wave $v_L(t) = S(t)\vec v_L(0)$ from the perturbed wave $w_L(t) = S_V(t)\vec w_L(0)$,
	we write, via Duhamel, 
	\begin{align*}
	w_L(t) &= S(t) \vec w_L(0) + \int_0^t S(t-s)(0,Vw_L(s)) ds \\
	&= S(t)\left [ \vec w_L(0) + \int_0^t S(-s)(0, Vw_L(s)) ds \right ]. 
	\end{align*}
	We then take 
	\begin{align*}
	\vec v_L(0) = \vec w_L(0) + \int_0^\infty S(-s)(0, Vw_L(s)) ds
	\end{align*}
	which converges in $\h$ by \eqref{s78} with $X = L^\infty_t \h$.  Then $\vec w_L(t) = \vec v_L(t) + o_{\h}(1)$ as $t \rar \infty$.  This along with \eqref{s417}
	allow us to conclude that if $\| u \|_{S(0,\infty)} < \infty$, then $u$ scatters to a free wave on $\M^5$ as 
	$t \rar \infty$.  The fact that the finiteness of $\| u \|_{S(0,\infty)}$ is necessary if $u$ scatters
	as $t \rar \infty$ follows from similar arguments using the fact that $\| v_L \|_{S(0,\infty)} < \infty$ holds
	for any free wave $v_L$ on $\M^5$. This concludes the proof. 
\end{proof}

A tool that will be essential in establishing the second step of the concentration--compactness/rigidity theorem method is 
the following long--time perturbation theory for \eqref{s81}. 

\begin{ppn}[Long--time perturbation theory]\label{p91}
	Let $A > 0$.  Then there exists $\e_0 = \e_0(A) > 0$ and $C = C(A) > 0$ such that the following holds.  Let $0 < \e
	< \e_0$, $(u_0,u_1) \in \h$, and $I \subseteq \R$ with $0 \in I$.  Assume that $\vec U(t) \in C(I; \h)$ satisfies on $I$
	\begin{align}\label{e91}
	\p_t^2 U - \Delta_g U + V U = N(\cdot, U) + e, 
	\end{align}
	such that 
	\begin{align}
	\sup_{t \in I} \| \vec U(t) \|_{\h} + \| U \|_{S(I)} &\leq A, \label{e92} \\
	\| \vec U(0) - (u_0,u_1) \|_{\h} + \| e \|_{N(I)} &\leq \e.  \label{e93}
	\end{align}
	Then the unique global solution $u$ to \eqref{s81} with initial data $\vec u(0) = (u_0,u_1)$ satisfies
	\begin{align}\label{e93b}
	\sup_{t \in I} \| \vec u(t) - \vec U(t) \|_{\h} + \| u - U \|_{S(I)} \leq C(A) \e. 
	\end{align}
\end{ppn}

\begin{proof}
	We establish the estimate \eqref{e93b} with $I_+ := I \cap [0,\infty)$ in place of $I$.  Establishing \eqref{e93b} with 
	$I_- := I \cap (-\infty,0]$ in place of $I$ is similar, and these two estimates yield \eqref{e93b}.  
	We first make some preliminary observations.  The bounds \eqref{e92} and \eqref{e93} along with conservation of energy imply that 
	\begin{align}\label{e94}
	\| \vec u(t) \|_{\h} \leq C_0(A).
	\end{align}
	Also, by interpolation and \eqref{e92}, $\| U \|_{W(J)} < \infty$ for all $J \Subset I$.  We claim that 
	\begin{align}\label{e95}
	\| U \|_{W(I)} \leq C_1(A). 
	\end{align}
	To see this, let $\eta > 0$ to be chosen later, and partition $I_+$ into subintervals $I_+ = \cup_{j = 1}^{J_0(A)} I_j$
	such that $\forall j$, $\| U \|_{S(I_j)} < \eta$.  Then via \eqref{e91} and Duhamel, we have on $I_j := [t_j,t_{j+1}]$
	\begin{align*}
	U(t) = S_V(t-t_j)\vec U(t_j) + \int_{t_j}^t S_V(t-s)\left ( 0, N(\cdot, U(s)) + e \right ) ds. 
	\end{align*}
	By arguing as in the proof of Proposition \ref{p42} and Strichartz estimates, we have 
	\begin{align*}
	\| U \|_{W(I_j)} &\leq C \| \vec U(t_j) \|_{\h} + C \| N(\cdot, U) \|_{N(I_j)} + C \| e \|_{N(I_j)} \\
	&\leq C A + C \| U \|_{W(I_j)} \| U \|_{S(I_j)} + C \| U \|_{S(I_j)}^2 + C \| \vec U \|_{L^\infty_t \h (I_j)} \| U \|_{S(I_j)}^3 
	+ C \| U \|_{S(I_j)}^3 + C \e \\
	&\leq C \eta \| U \|_{W(I_j)} + C \e + C\cdot (A + 1)^4.  
	\end{align*}
	If we choose $\eta = (2C)^{-1}$, then we obtain \eqref{e95}.
	
	We now establish \eqref{e93b}.  Define $w = u - U$.  Then $w$ solves on $I$ 
	\begin{align}
	\begin{split}\label{e97}
	&\p_t^2 w - \Delta_g w + V w = N(\cdot, U + w) - N(\cdot, U) - e, \\
	&\vec w(0) = (u_0,u_1) - \vec U(0).
	\end{split}
	\end{align}
	By \eqref{e92} and \eqref{e94}, $w$ satisfies 
	\begin{align}\label{e98}
	\sup_{t \in I} \| \vec w(t) \|_{\h} \leq A + C_0(A). 
	\end{align}
	Let $\eta > 0$ to be chosen later.  Partition $I_+$ into subintervals $I = \cup_{j = 1}^{J_1(A)} I_j$ such that 
	\begin{align}\label{e96}
	\forall j, \quad \| U \|_{S(I_j)} + \| U \|_{W(I_j)} \leq \eta. 
	\end{align}
	On $I_j := [t_j,t_{j+1}]$, we have via \eqref{e97} and Duhamel
	\begin{align}\label{e99}
	w(t) = S_V(t - t_j) \vec w(t_j) + \int_{t_j}^t S_V(t-s)\left ( 0, N(\cdot, U(s) + w(s)) - N(\cdot, U(s)) - e \right ) ds. 
	\end{align}
	By arguing as in the proof of Proposition \ref{p42} and Strichartz estimates, we have
	\begin{align*}
	\| w \|_{S(I_j)} + \| w \|_{W(I_j)} &\leq \| S_V(t - t_j) \vec w(t_j) \|_{S(\R)} + \| S_V(t - t_j) \vec w(t_j) \|_{W(\R)}
	+ C \| e \|_{N(I_j)} \\ 
	&\:+ C \| N(\cdot, U + w) - N(\cdot, U) \|_{N(I_j)}  \\
	&\leq \| S_V(t - t_j) \vec w(t_j) \|_{S(\R)} + \| S_V(t - t_j) \vec w(t_j) \|_{W(\R)}
	+ C \| e \|_{N(I_j)} \\ 
	&\:+ C \Bigl [
	\| w \|_{W(I_j)} \| U \|_{S(I_j)} + \| w \|_{S(I_j)} \| U \|_{W(I_j)} + \| w \|_{S(I_j)} \| w \|_{W(I_j)} + 
	\| w \|_{S(I_j)}^2   \\
	&\:+ \| w \|_{S(I_j)} \| U \|_{S(I_j)} + \| w \|_{S(I_j)} \| U \|_{S(I_j)}^2 ( \| U \|_{L^\infty_t\h(I_j)} + 1 ) 
	+ \| w \|_{S(I_j)}^2 \| U \|_{S(I_j)} ( \| U \|_{L^\infty_t\h(I_j)} + 1 ) \\
	&\:+ \| w \|_{S(I_j)}^3 ( \| U \|_{L^\infty_t \h} + \| w \|_{L^\infty_t\h(I_j)} + 1 ) 
	\Bigr ]  \\ 
	&\leq \| S_V(t - t_j) \vec w(t_j) \|_{S(\R)} + \| S_V(t - t_j) \vec w(t_j) \|_{W(\R)}
	+ C \e \\
	&\:+ (\eta + \eta^2) (A + 1) C \Bigl [ \| w \|_{S(I_j)} + \| w \|_{W(I_j)} \Bigr ] \\
	&\:+ C_2(A) \Bigl [ ( \| w \|_{S(I_j)} + \| w \|_{W(I_j)}  )^2 + ( \| w \|_{S(I_j)} + \| w \|_{W(I_j)} )^3  
	\Bigr ]. 
	\end{align*}
	Here $C_2 = C_2(A)$ is a constant which depends only $A$. Define 
	\begin{align*}
	\gamma_j := \| S_V(t - t_j) \vec w(t_j) \|_{S(\R)} + \| S_V(t - t_j) \vec w(t_j) \|_{W(\R)}
	+ C \e.
	\end{align*}
	If we fix $\eta$ so small so that $\eta + \eta^2 < (2(A + 1)C)^{-1}$, then we obtain 
	\begin{align}\label{e910}
	\| w \|_{S(I_j)} + \| w \|_{W(I_j)} \leq 2 \gamma_j + 2 C_1(A) 
	\Bigl [ ( \| w \|_{S(I_j)} + \| w \|_{W(I_j)}  )^2 + ( \| w \|_{S(I_j)} + \| w \|_{W(I_j)} )^3 \Bigr ].
	\end{align}
	In particular, by a standard continuity argument there exists $\delta_0 = \delta_0(C_1(A))$ such that if $\gamma_j 
	< \delta_0$, then 
	\begin{align}
	\| w \|_{S(I_j)} + \| w \|_{W(I_j)} &\leq 4 \gamma_j, \label{e911} \\
	2 C_2(A) \Bigl [ ( \| w \|_{S(I_j)} + \| w \|_{W(I_j)}  )^2 + ( \| w \|_{S(I_j)} + \| w \|_{W(I_j)} )^3 \Bigr ] &\leq
	4 \gamma_j. \label{e911b}
	\end{align}
	We now iterate, and insert $t_{j+1}$ into \eqref{e99}.  Applying $S_V(t-t_{j+1})$ to both sides, we obtain 
	\begin{align*}
	S(t-t_{j+1})\vec w(t_{j+1}) = S(t - t_j) \vec w(t_j) + \int_{t_j}^{t_{j+1}} 
	S(t - s) \left ( 0 , N(\cdot, U(s) + w(s)) - N(\cdot, U(s)) - e \right ) ds.
	\end{align*}
	By \eqref{e911} and \eqref{e911b} and the previous arguments, we deduce that 
	\begin{align*}
	\gamma_{j+1} \leq 10 \gamma_j,
	\end{align*}
	provided that $\gamma_j < \delta_0$.  By Strichartz estimates and \eqref{e93}, we have for some absolute constant $C_3$ 
	\begin{align*}
	\gamma_1 := \| S_V(t) \vec w(0) \|_{S(\R)} + \| S_V(t) \vec w(0) \|_{S(\R)}
	+ C \e \leq C_3 \e < C_3 \e_0.
	\end{align*}
	Iterating, we have that $\gamma_{j+1} \leq 10^{j} C_3 \e$ as long as $\gamma_j < \delta_0$.  If we choose 
	$\e_0 = \e_0(A)$ so small so that $10^J C_3 \e_0 < \delta_0$, then the condition $\gamma_j < \delta_0$ is always 
	satisfied.  This along with \eqref{e911} imply that 
	\begin{align*}
	\| w \|_{S(I_+)} + \| w \|_{W(I_+)} \leq C(A) \e 
	\end{align*}
	as desired. The estimate for $\| w \|_{L^\infty_t \h(I_+)}$ follows a posteriori from
	\eqref{e97}, \eqref{e92}, \eqref{e93}, the estimate for $\| w \|_{S(I_+)} + \| w \|_{W(I_+)}$, and 
	Strichartz estimates.  This completes the proof.  
\end{proof}


\subsection{Concentration--compactness}

In this, the second step of the concentration--compactness methodology, we show that if our main result Theorem \ref{t41} (or 
equivalently Theorem \ref{t01}) fails,  then there exists a nonzero `critical element.'  More precisely, we prove the following.

\begin{ppn}\label{p43}
	Suppose that Theorem \ref{t41} fails.  Then there exists a nonzero global solution $u_*$ to \eqref{s81} such that the 
	set 
	\begin{align*}
	K = \left \{ \vec u_*(t) : t \in \R \right \}
	\end{align*}
	is precompact in $\h$.  
\end{ppn} 

Essential tools for proving Proposition \ref{p43} are the following linear and nonlinear \emph{profile decompositions}. 

\begin{lem}[Linear Profile Decomposition]\label{l45}
	Let $\{ (u_{0,n}, u_{1,n}) \}_n$ be a bounded sequence in $\h$.  Then 
	after extraction of subsequences and relabeling, there exist a sequence of solutions $\left \{ U_L^j \right \}_{j \geq 1}$ to 
	\eqref{s412} with $h \equiv 0$ which are bounded in $\h$ and a sequence of times  $\{ t_{j,n} \}_n$ for $j \geq 1$ that 
	satisfy the orthogonality condition 
	\begin{align*}
	\forall j \neq k, \quad \lim_{n \rar \infty} |t_{j,n} - t_{k,n}| = \infty, 
	\end{align*}
	such that for all $J \geq 1$, 
	\begin{align*}
	(u_{0,n},u_{1,n}) = \sum_{j = 1}^J \vec U^j_L(-t_{j,n}) + (w^J_{0,n},w^J_{1,n}), 
	\end{align*}
	where the error $w_n^J(t) := S_V(t)(w^J_{0,n},w^J_{1,n})$ satisfies
	\begin{align}
	\lim_{J \rar \infty} \lims_{n \rar \infty} \| w^J_n \|_{L^\infty_t L^p_x(\R) \cap S(\R)} = 0, 
	\quad \forall \: \frac{10}{3} < p \leq \infty. \label{e913}
	\end{align}
	
	Moreover, we have the following Pythagorean expansion of the energy 
	\begin{align}\label{e914e}
	\cl E_V( \vec u_n) = \sum_{j = 1}^J \cl E_V( \vec U^j_L ) + \cl E_V ( \vec w^J_n ) + o(1),
	\end{align}
	as $n \rar \infty$. 
\end{lem}

The proof of of Lemma \ref{l45} is identical to the proof of Lemma 3.2 in \cite{ls} and we omit it.  The sequence $\{(u_{0,n}, u_{1,n})\}_n$ in Lemma \ref{l45} is said to have a \emph{profile decomposition} with profiles 
$\{ U_L^j \}_j$ and parameters $\{ t_{j,n} \}_{j,n}$. We note that after passing to a further subsequence if necessary, we may 
assume that for all $j \geq 1$, either $t_{j,n} = 0$ $\forall n$ or $\lim_n t_{j,n} = \pm \infty$. 

In order to apply Lemma \ref{l45} in the context of the nonlinear problem \eqref{s81}, we will need the notion of \emph{nonlinear profiles}. 
For each profile $U^j_L$ with time parameters $\{ t_{j,n} \}_n$, we define its associated nonlinear profile $U^j$ to be the unique
global solution to \eqref{s81} such that 
\begin{align*}
\lim_{n \rar \infty} \| \vec U^j(-t_{j,n}) - \vec U^j_L(-t_{j,n}) \|_{\h} = 0. 
\end{align*}
It is easy to see that a nonlinear profile always exists.  Indeed, if $t_{j,n} = 0$ for all $n$, then we set $U^j$ to be the 
solution to \eqref{s81} with initial data $\vec U^j(0) = \vec U^j_L(0)$.  If $\lim_n -t_{j,n} = \infty$, say, then we set 
$U^j$ to be the unique globally defined solution to the integral equation 
\begin{align}
U^j(t) = \vec U^j_L(-t_{j,n}) - \int_t^\infty S_V(t-s)(0,N(\cdot, U^j(s)) ) ds. \label{e912}
\end{align}
A unique global solution to \eqref{e912} can be shown to exist using contraction mapping arguments in the spirit of those used in Proposition 
\ref{p42} and Proposition \ref{p91}.  

For each nonlinear profile $U^j$, we denote 
\begin{align*}
U^j_n(t) := U^j(t - t_{j,n}).
\end{align*}
Using Proposition \ref{p91}, we obtain the following nonlinear profile decomposition from the linear profile decomposition
in Lemma \ref{l45}.

\begin{lem}[Nonlinear Profile Decomposition]\label{l92}
	Let $\{(u_{0,n},u_{1,n})\}_n$ be a bounded sequence in $\h$ admitting a profile decomposition with profiles 
	$\{U^j_L\}_j$ and parameters $\{t_{j,n}\}_{j,n}$.
	Let $T_n \in [0,+\infty)$.  Assume
	\begin{align}
	\forall j \geq 1, \quad 
	\limsup_{n \rightarrow \infty} \| U^j \|_{S( -t_{j,n}, T_n-t_{j,n} )}
	< \infty. \label{e92a} 
	\end{align}
	Let $u_n$ be the unique global solution to \eqref{s81} with initial data $\vec u(0) = (u_{0,n},u_{1,n})$.  Then 
	\begin{align*}
	\limsup_{n \rightarrow \infty} \|u_n\|_{S(0,T_n) } < \infty,
	\end{align*}
	and for all $t \in [0,T_n]$
	\begin{align*}
	\vec{u}_n(t) = \sum_{j = 1}^J  \vec{U}^j_n(t) +  \vec{w}^J_n(t) +  \vec{r}^J_n(t),
	\end{align*}
	with
	\begin{align*}
	\lim_{J \rightarrow \infty} \lims_{n \rightarrow \infty} \left [ \| r^J_n \|_{S \left ( 0,T_n \right )}
	+ \sup_{t \in [0,T_n]} \left \| \vec r^J_n(t) \right \|_{\h} \right ]  = 0.
	\end{align*}
	An analogous statement holds if $T_n < 0$.
\end{lem} 

\begin{proof}
	For $J \geq 1$, $n \geq 1$, define
	\begin{align*}
	U^J_n(t) := \sum_{j = 1}^J U^j_n(t) + w^J_n(t).
	\end{align*}
	We will apply Proposition \ref{p91} with $U = U^J_n$ and $u = u_n$ for $n$ and $J$ large. We first show that 
	\begin{align}\label{e914s}
	\lims_{J} \lims_{n} \| U^J_n \|_{S(0,T_n)} < \infty. 
	\end{align}
	By assumption, there exists $M > 0$ such that $\forall n$, $\|(u_{0,n}, u_{1,n}) \|_{\h}^2 \simeq \cl E_V(u_{0,n}, u_{1,n})  \leq M$.  The Pythagorean expansion 
	of the energy \eqref{e914e} implies that 
	\begin{align}\label{e914}
	\lims_J \lims_n \cl E_V (w^J_n) + \sum_{j = 1}^\infty \cl E_V (U^j_L) \leq M. 
	\end{align}
	Hence, there exists $J_0 \geq 1$ such that 
	\begin{align*}
	\sum_{j > J_0} \cl E_V(\vec U^j_L) \ll \delta^2, 
	\end{align*}
	where $\delta$ is from Proposition \ref{p42}.  In particular, this implies by Proposition \ref{p42} that the nonlinear profiles
	satisfy for all $j > J_0$ 
	\begin{align*}
	\| \vec U^j \|_{L^\infty_t \h} + \| U^j \|_{S(\R)} + \| U^j \|_{W(\R)} \lesssim \cl E_V( \vec U^j_L )^{1/2}. 
	\end{align*}
	Let $J \geq 1$.  Then 
	\begin{align*}
	\| U^J_n \|_{S(0,T_n)} \leq \left \| \sum_{j = 1}^J U^j_n \right \|_{S(0,T_n)} + \| w^J_n \|_{S(\R)}.   
	\end{align*}
	Now 
	\begin{align}
	\begin{split}\label{e115}
	\left \| \sum_{j = 1}^J U^j_n \right \|_{S(0,T_n)}^3 &\leq 
	\left \| \sum_{j = 1}^J \| U^j_n \|_{L^6_x} \right \|_{L^3_t(0,T_n)}^3 \\
	&= \sum_{j = 1}^J \| U^j_n \|_{S(0,T_n)}^3 + \e^J_n, 
	\end{split}
	\end{align}
	where the error $\e^J_n$ is a sum of terms of the form 
	\begin{align*}
	\int_0^T \| U^j_n(t) \|_{L^6_x} \| U^k_n(t) \|_{L^6_x} \| U^l_n(t) \|_{L^6_x} dt,
	\end{align*}
	with $1 \leq j, k, l \leq J$ and $j \neq k$.  We claim that  
	\begin{align}\label{e114}
	\lim_{n \rar \infty} \int_0^{T_n} \| U^j_n(t) \|_{L^6_x} \| U^k_n(t) \|_{L^6_x} \| U^l_n(t) \|_{L^6_x} dt = 0. 
	\end{align}
	Indeed, by the assumption \eqref{e92a} and an approximation argument, we may assume that the functions $U^j, U^k$ are compactly supported in $t$.  Now 
	\begin{align*}
	\int_0^{T_n} \| U^j_n(t) \|_{L^6_x} \| U^k_n(t) \|_{L^6_x} \| U^l_n(t) \|_{L^6_x} dt &\lesssim 
	\Bigl ( \int_0^{T_n} \| U^j_n(t) \|_{L^6_x}^{3/2} \| U^k_n(t) \|_{L^6_x}^{3/2}dt\Bigr )^{2/3} \| U^l_n \|_{S(0,T_n)} \\
	&\lesssim \Bigl ( \int_0^{T_n} \| U^j_n(t) \|_{L^6_x}^{3/2} \| U^k_n(t) \|_{L^6_x}^{3/2}dt \Bigr )^{2/3}.
	\end{align*}
	Extending the integration over all of $\R$ and changing variables implies that 
	\begin{align*}
	\int_0^{T_n} \| U^j_n(t) \|_{L^6_x}^{3/2} \| U^k_n(t) \|_{L^6_x}^{3/2}dt \leq \int \| U^j(t) \|_{L^6_x}^{3/2} \| U^k(t + t_{j,n} - t_{k,n}) \|_{L^6_x}^{3/2}dt
	\end{align*}
	The orthogonality of the parameters implies that $|t_{j,n} - t_{k,n}| \rar_n \infty$.  Thus, the support of $U^j(\cdot)$ and
	$U^k(\cdot + t_{j,n} - t_{k,n})$ are eventually disjoint whence 
	\begin{align*}
	\lim_{n \rar \infty} \int \| U^j(t) \|_{L^6_x}^{3/2} \| U^k(t + t_{j,n} - t_{k,n}) \|_{L^6_x}^{3/2}dt = 0. 
	\end{align*}
	This proves \eqref{e114}.  Returning to \eqref{e115} and recalling our choice of $J_0$, we see that 
	\begin{align*}
	\lims_n \left \| \sum_{j = 1}^J U^j_n \right \|_{S(0,T_n)}^3 &\leq \lims_n 
	\sum_{j = 1}^J \| U^j_n \|_{S(0,T_n)}^3 \\
	&\lesssim \sum_{j = 1}^{J_0} \lims_n \| U^j_n \|_{S(0,T_n)}^3 + \sum_{j > J_0}  \| U^j \|_{S(\R)}^3 \\
	&\lesssim 1 + \sum_{j > J_0} \cl E_V (\vec U^j_L )^{3/2} \\
	&\lesssim 1 + M,
	\end{align*}
	where the implied constant is independent of $J$.  Thus, 
	\begin{align*}
	\lims_J \lims_n \| U^J_n \|_{S(0,T_n)} \leq \lims_J \lims_n \left \| \sum_{j = 1}^J U^j_n \right \|_{S(0,T_n)}
	+ \lims_J \lims_n \| w^J_n \|_{S(\R)} < \infty. 
	\end{align*}
	Using similar arguments, we also conclude that 
	\begin{align*}
	\lims_J \lims_n \| U^J_n \|_{L^\infty_t \h(0,T_n)} < \infty. 
	\end{align*}
	
	We now verify that the following error 
	\begin{align*}
	e^J_n &:= \p_t^2 U^J_n - \Delta_g U^J_n + V U^J_n - N(\cdot, U^J_n) \\
	&= \sum_{j = 1}^J N(\cdot, U^j_n) - N \left ( \cdot, \sum_{j = 1}^J U^j_n + w^J_n \right ),  
	\end{align*}
	satisfies 
	\begin{align}\label{e116}
	\lims_J \lims_n \| e^J_n \|_{N(0,T_n)} = 0. 
	\end{align}
	We focus only on the quadratic part of $N(\cdot,u)$ since the other parts can be handled similarly.  More precisely, we 
	show that 
	\begin{align}\label{e915}
	\lims_J \lims_n 
	\left \| (2 \la r \ra^{-1} \sin 2Q ) \left ( \sum_{j = 1}^J (U^j_n)^2 - \left ( \sum_{j = 1} U^j_n + w^J_n \right ) 
	^2 \right ) \right \|_{L^{3/2}_t \dot W^{1/2,3/2}_x (0,T_n)} = 0. 
	\end{align}
	To lessen the notation, for $I \subseteq \R$, we denote $W'(I) := L^{3/2}_t \dot W^{1/2,3/2}_x (I)$.  We observe that 
	\begin{align*}
	&\left \| (2 \la r \ra^{-1} \sin 2Q ) \left ( \sum_{j = 1}^J (U^j_n)^2 - \left ( \sum_{j = 1} U^j_n + w^J_n \right ) 
	^2 \right ) \right \|_{W'(0,T_n)} \\&\:\: \lesssim 
	\left \| (2 \la r \ra^{-1} \sin 2Q ) w^J_n \sum_{j = 1}^J U^j_n \right \|_{W'(0,T_n)} + 
	\sum_{j \neq k} \left \| (2 \la r \ra^{-1} \sin 2Q ) U^j_n \sum_{j = 1}^J U^k_n \right \|_{W'(0,T_n)} \\&\:\:\:+ 
	\| (2 \la r \ra^{-1} \sin 2Q ) (w^J_n)^2 \|_{W'(0,T_n)} \\
	&\:\: =: A^J_n + B^J_n + C^J_n. 
	\end{align*}
	Using the orthogonality of the parameters and arguments as in the previous paragraph, it is straightforward to show that 
	\begin{align*}
	\lim_n B^J_n = 0. 
	\end{align*}
	To estimate $C^J_n$, we recall that $\lims_j \lims_n \| w^J_n \|_{S(\R)} = 0$ and $\vec w^J_n(0)$ is bounded in $\h$.  
	Thus, by the product rule (see the proof of Proposition \ref{p42}) and Strichartz estimates, we have 
	\begin{align*}
	C^J_n \lesssim \| w^J_n \|_{S(\R)} \| w^J_n \|_{W(\R)} + \| w^J_n \|_{S(\R)}^2 \lesssim
	\| w^J_n \|_{S(\R)} + \| w^J_n \|_{S(\R)}^2 ,
	\end{align*}
	whence $\lims_J \lims_n C^J_n = 0$.  We now show that $\lims_J \lims_n A^J_n = 0$. Let $\e > 0$.  By the arguments used to show that 
	$\lims_J \lims_n \| U^J_n \|_{S(0,T_n)} < \infty$,  there exists $J_1 = J_1(\e) > J_0$ such that for all $J > J_1$
	\begin{align}\label{e117}
	\lims_n \left ( \left \| \sum_{j = J_1 + 1}^J U^j_n \right \|_{S(0,T_n)} + \left \| \sum_{j = J_1 + 1}^J U^j_n 
	\right \|_{S(0,T_n)} \right )
	< \e.
	\end{align}
	Thus, by the product rule, we obtain 
	\begin{align}
	&\lims_n \left \| (2 \la r \ra^{-1} \sin 2Q ) w^J_n \sum_{j = J_1 + 1}^J U^j_n \right \|_{W'(0,T_n)} \notag\\
	&\:\lesssim \lims_n \| w^J_n \|_{W(\R)} \left \| \sum_{j = J_1 + 1}^J U^j_n \right \|_{S(0,T_n)} 
	+ \lims_n \| w^J_n \|_{S(\R)} \left \| \sum_{j = J_1 + 1}^J U^j_n \right \|_{S(0,T_n)} \\
	&\:\:+ \lims_n \| w^J_n \|_{S(\R)} \left \| \sum_{j = J_1 + 1}^J U^j_n \right \|_{W(0,T_n)} \notag \\
	&\: \lesssim \e, \notag
	\end{align}
	where the implied constant is independent of $J$.  Thus, 
	\begin{align}\label{e118}
	\lims_J \lims_n A^J_n \lesssim \e + \lims_J \lims_n \sum_{j = 1}^{J_1} 
	\left \| (2 \la r \ra^{-1} \sin 2Q ) w^J_n U^j_n \right \|_{W'(0,T_n)}. 
	\end{align}
	Fix $j \in \{ 1, \ldots, J_1 \}$.  We wish to show that 
	\begin{align}\label{e122}
	\lims_J \lims_n \sum_{j = 1}^{J_1} 
	\left \| (2 \la r \ra^{-1} \sin 2Q ) w^J_n U^j_n \right \|_{W'(0,T_n)} = 0. 
	\end{align}
	By the product rule, 
	\begin{align}\label{e120a}
	\left \| (2 \la r \ra^{-1} \sin 2Q ) w^J_n U^j_n \right \|_{L^{3/2}_x}
	\lesssim \| w^J_n \|_{L^6_x} \| U^j_n \|_{L^6_x}
	+ \| w^J_n \|_{L^6_x} \| U^j_n \|_{\dot W^{1/2,3}_x} + \| w^J_n \|_{\dot W^{1/2,3}_x} \| U^j_n \|_{L^6_x}
	\end{align}
	Arguing as in the proof of Proposition \ref{p91}, the assumption \eqref{e92a} also implies that for all $j \geq 1$,
	\begin{align}\label{e119}
	\lims_n \| U^j \|_{W(0,T_n)} < \infty.
	\end{align}
	This fact, \eqref{e120a}, H\"older's inequality, and the fact that $\lims_J \lims_n \| w^J_n \|_{S(0,T_n)} = 0$ imply that  
	\begin{align}\label{e120}
	\lims_J \lims_n \left \| (2 \la r \ra^{-1} \sin 2Q ) w^J_n U^j_n \right \|_{W'(0,T_n)} \lesssim 
	\lims_J \lims_n \left \| \| w^J_n(t) \|_{\dot W^{1/2,3}_x} \| U^j_n(t) \|_{L^6_x} \right \|_{L^{3/2}_t(0,T_n)}. 
	\end{align}
	We now show that 
	\begin{align}\label{e121}
	\lims_J \lims_n \int_0^{T_n} \| w^J_n(t) \|_{\dot W^{1/2,3}_x}^{3/2} \| U^j_n(t) \|_{L^6_x}^{3/2} dt = 0.
	\end{align}
	By the assumption \eqref{e92a} and an approximation argument, we can assume that $U^j$ is compactly supported in $t$. By 
	interpolation, we have the estimate
	\begin{align}
	\forall t, \quad \| w^J(t) \|_{\dot W^{1/2,3}_x } \lesssim \| \nabla w^J_n(t) \|_{L^2_x}^{1/2} \| w^J_n(t) \|_{L^6_x}^{1/2}.  
	\end{align}
	Thus, by H\"older's inequality 
	\begin{align*}
	\int_0^{T_n} \| w^J_n(t) \|_{\dot W^{1/2,3}_x}^{3/2} \| U^j_n(t) \|_{L^6_x}^{3/2} dt
	&\lesssim \int \| w^J_n(t+t_{j,n}) \|_{\dot W^{1/2,3}_x}^{3/2} \| U^j(t) \|_{L^6_x}^{3/2} dt \\
	&\lesssim \int \| \vec w^J_n \|_{L^\infty_t \h}^{3/4} \| w^J_n(t + t_{j,n}) \|_{L^6_x}^{3/4} \| U^j(t) \|_{L^6_x}^{3/2} dt \\
	&\lesssim \| w^J_n \|_{S(\R)}^{3/4},
	\end{align*}
	where the implied constant depends on $U^j$.  Thus,
	\begin{align*}
	\lims_J \lims_n \int_0^{T_n} \| w^J_n(t) \|_{\dot W^{1/2,3}_x}^{3/2} \| U^j_n(t) \|_{L^6_x}^{3/2} dt
	\lesssim \lims_J \lims_n \| w^J_n \|_{S(\R)}^{3/4} = 0. 
	\end{align*} 
	This proves \eqref{e121}.  By \eqref{e120}, this also proves \eqref{e122}.  By \eqref{e118}, this proves 
	\begin{align*}
	\lims_J \lims_n A^J_n \lesssim \e,
	\end{align*}
	which proves \eqref{e915}.  
	
	We have now demonstrated that the function $U^J_n$ satisfies the hypotheses stated in Proposition \ref{p91} uniformly in 
	$J,n$ large and 
	\begin{align*}
	\lims_J \lims_n \| e^J_n \|_{N(0,T_n)} = 0. 
	\end{align*}
	Since 
	$\vec U^J_n(0) = u_n(0) + o_{\h}(1)$ as $n \rar \infty$, we have by Proposition \ref{p91}, for $t \in [0,T_n]$,
	\begin{align*}
	\vec u_n(t) = \vec U^J_n(t) + \vec r^J_n(t),
	\end{align*}
	with 
	\begin{align*}
	\lim_{J \rightarrow \infty} \lims_{n \rightarrow \infty} \left [ \| r^J_n \|_{S \left ( 0,T_n \right )}
	+ \sup_{t \in [0,T_n]} \left \| \vec r^J_n(t) \right \|_{\h} \right ]  = 0.
	\end{align*}
	This completes the proof. 
\end{proof}

We now prove Proposition \ref{p43}. 

\begin{proof}[Proof of Proposition \ref{p43}]
	For $A > 0$, define 
	\begin{align*}
	\cl B(A) := 
	\left \{ (u_0,u_1) \in \h : \mbox{if $u$ solves \eqref{s81} with $\vec u(0) = (u_0,u_1)$ then } \sup_{t \in [0,\infty)} 
	\cl E_V (\vec u_n(t) )^{1/2} \leq A \right \}.
	\end{align*}
	We say that the property $\mathcal{SC}(A)$ holds if for all $(u_0,u_1) \in \cl B(A)$, the solution $u$ to \eqref{s81} satisfies 
	$ \| u \|_{S(0,\infty)} < \infty$. Note that by Proposition \ref{p42} and \eqref{s415}, every solution $u$ to \eqref{s81} is in $\cl B(A)$ 
	for some $A$ and if $0 < A < \delta$, where $\delta$ is as in Proposition \ref{p42}, then $\cl{SC}(A)$ holds.  Define 
	\begin{align*}
	A_C := \sup \left \{ A > 0 : \cl{SC}(A) \mbox{ holds.} \right \} > 0.
	\end{align*}
	By the temporal symmetry of \eqref{s81} and Proposition \ref{p42}, we see that Theorem \ref{t41} is equivalent to the statement 
	\begin{align*}
	A_C = \infty. 
	\end{align*}
	
	Suppose not, i.e. $0 < A_C < \infty$.  Then there exists a sequence of real numbers $A_n \downarrow A$ and a sequence 
	$\{ (u_{0,n}, u_{1,n}) \}_n$ in $\h$ such that the corresponding solutions $u_n$ to \eqref{s81} with initial data 
	$\vec u_n(0) = (u_{0,n}, u_{1,n})$ satisfy
	\begin{align}
	\begin{split}\label{e111}
	\exists T_n < 0, T_n \rar -\infty,\quad \sup_{t \in (T_n,\infty)} \cl E_V (\vec u_n(t) )^{1/2} \leq A_n, \\
	\| u_n \|_{S(0,\infty)} = \infty, \\
	\lim_{n \rar \infty} \| u_n \|_{S(-T_n,0)} = \infty.
	\end{split}
	\end{align}
	Note that \eqref{e111} and \eqref{s415} imply that the sequence $\{ \vec u_n(0) = (u_{0,n}, u_{1,n}) \}_n$ is bounded in  $\h$.  After
	passing to a subsequence if necessary,
	$\vec u_n(0)$ admits a profile decomposition 
	\begin{align}\label{e112}
	\vec u_n(0) = \sum_{j = 1}^J \vec U^j_L(-t_{j,n}) + \vec w^J_n(0)
	\end{align}
	with profiles $\{ U^j_L \}_j$ and time parameters $\{ t_{j,n} \}_{j,n}$ by Lemma \ref{l45}.  As before we assume, 
	without loss of generality, that for all $j$ either $t_{j,n} = 0$ $\forall n$ or $\lim_n t_{j,n} = \pm \infty$. Let $\{U^j\}_j$
	be the sequence of associated nonlinear profiles.  By the Pythagorean expansion of the 
	energy, there exists $J_0 > 1$ such that 
	\begin{align*}
	\sum_{j > J_0} \cl E_V ( \vec U^j_L ) \ll \delta^2, 
	\end{align*}
	where $\delta$ is as in the small data theory, Proposition \ref{p42}.  Thus, the associated nonlinear profiles satisfy 
	\begin{align*}
	\| U^j \|_{S(\R)} \lesssim \cl E_V( \vec U^j_L )^{1/2}.
	\end{align*}
	Define 
	\begin{align*}
	\cl J = \left \{ j \in \{ 1, \ldots, J_0 \} : \| U^j \|_{S(0,\infty)} = \infty \right \}. 
	\end{align*}
	First, we note that $\cl J \neq \varnothing$.  Otherwise, by the definition of nonlinear profiles and our choice of $J_0$, we have
	\begin{align*}
	\forall j \geq 1, \quad \| U^j \|_{S(0,\infty)} < \infty.
	\end{align*}
	By Lemma \ref{l92}, this would imply that $\| u_n \|_{S(0,\infty)} < \infty$ for large $n$, a contradiction to \eqref{e111}. Thus, 
	$\cl J \neq \varnothing$.  Note that if $j \in \cl J$ and $-t_{j,n} \rar_n \infty$, then $U^j$ scatters forward in time, 
	i.e. $\| U^j \|_{S(0,\infty)} < \infty$, a contradiction to our definition of $\cl J$.  Thus, for all $j \in \cl J$, 
	we have that $-t_{j,n} \rar_n -\infty$. By the orthogonality of the parameters and after rearranging the first $J_0$ profiles if necessary, we 
	may assume that if
	$j > 1$, then  
	\begin{align*}
	\lim_{n \rar \infty} t_{1,n} - t_{j,n} = -\infty
	\end{align*}
	
	We now claim that $\cl J = \{ 1 \}$ and that for all $j \geq 2$, $\vec U^j_L = 0$.  Suppose not and, say, $\vec U^2_L \neq 0$.  Then for $T \geq 0$ and 
	for all $j \geq 1$, 
	\begin{align*}
	\lims_n \| U^j \|_{S(-t_{j,n}, T + t_{1,n} - t_{j,n})} < \infty. 
	\end{align*}
	By Proposition \ref{l92}, the Pythagorean expansion of the energy, and conservation of the energy $\cl E_V(\cdot)$ we conclude that 
	\begin{align*}
	\cl E_V ( \vec u_n(T + t_{1,n}) ) &= \cl E_V (\vec U^1(T)) + \sum_{j = 2}^J \cl E_V( \vec U^j(T + t_{1,n} - t_{j,n}) ) + 
	\cl E_V ( \vec w^J_n ) + \cl E_V (\vec r^J_n(T+t_{1,n})) + o_n(1) \\
	&= \cl E_V (\vec U^1(T)) + \sum_{j = 2}^J \cl E_V(\vec U^j_L) + 
	\cl E_V ( \vec w^J_n ) + \cl E_V (\vec r^J_n(T+t_{1,n})) + o_n(1) \\
	&\geq \cl E_V (\vec U^1(T) ) + \cl E_V( \vec U^2_L ) + o_n(1)
	\end{align*}
	as $n \rar \infty$.  In particular, 
	\begin{align*}
	\cl E_V (\vec U^1(T) ) \leq A_0^2 < A_C^2.
	\end{align*}
	Since $T \geq 0$ was arbitrary, we conclude that $\sup_{t \in [0,\infty)} \cl E_V (\vec U^1(t))^{1/2} \leq A_0 < A_C$.  By the minimality of $A_C$, 
	it follows that $\| U^1 \|_{S(0,\infty)} < \infty$, a contradiction to the fact that $1 \in \cl J$.  Thus, $\vec U^j_L = 0$ for all $j 
	\geq 2$.  By a similar argument, we also deduce that 
	\begin{align*}
	\lim_{n \rar \infty} \cl E_V( \vec w^1_n ) = 0,
	\end{align*}
	or equivalently $\lim_n \| \vec w^J_n(0) \|_{\h} = 0$. 
	
	We have now shown that 
	\begin{align*}
	\vec u_n(0) = \vec U^1(-t_{1,n}) + o_{\h}(1), 
	\end{align*}
	as $n \rar \infty$.  We claim that $t_{1,n} = 0$ for all $n$.  If not, then by our initial assumptions on the 
	parameters we have $-t_{1,n} \rar_n -\infty$.  This implies that $\| U^1 \|_{S(-\infty,0)} < \infty$.  By 
	Proposition \ref{p91}, we deduce that $\lims_n \| u_n \|_{S(-\infty,0)} < \infty$, a contradiction to
	\eqref{e111}.  Thus, 
	\begin{align*}
	\vec u_n(0) = \vec U^1(0) + o_{\h}(1), 
	\end{align*}
	as $n \rar \infty$.  Define $u_* = U^1$.  Then by Proposition \ref{p91} and \eqref{e111}, $u_*$ satisfies 
	\begin{align}
	\label{e113} 
	\begin{split}
	\sup_{t \in (-\infty,\infty)} \cl E_V( \vec u_*(t) ) \leq A_C, \\
	\| u_* \|_{S(-\infty, 0)} = \| u_* \|_{S(0,\infty)} = \infty.  
	\end{split}
	\end{align}
	
	Finally, we show that $\{ \vec u_*(t) : t \in \R \}$ is precompact in $\h$.  By continuity of the flow, it suffices to show
	that if $\{ t_n \}_n$ is a sequence in $\R$, with $\lim_{n \rar \infty} t_n = \pm \infty$, then there exists a subsequence 
	(still denoted by $t_n$) such that $\vec u_*(t_n)$ converges in $\h$.  Suppose that $t_n \rar_n \infty$.  Define 
	$(u_{0,n}, u_{1,n}) := \vec u_*(t_n)$.  Then the solution $u_n$ to \eqref{s81} with initial data $\vec u_n(0) = (u_{0,n},
	u_{1.n})$ is given by $u_n(t) = u_*(t + t_n)$ whence by \eqref{e113}, the solutions $u_n$ satisfy the conditions 
	given in \eqref{e111}.  Thus, we may repeat the previous argument to conclude that there exists
	a subsequence (still indexed by $n$) and $\vec U^1(0) \in \h$ such that 
	\begin{align*}
	u_*(t_n) = \vec u_n(0) = \vec U^1(0) + o_{\h}(1),
	\end{align*}
	as $n \rar \infty$.  If $t_n \rar -\infty$, then we apply the previous argument to 
	$u_*(-t)$ to conclude.  Thus, the set 
	\begin{align*}
	K := \{ \vec u_*(t) : t \in \R \}, 
	\end{align*}
	is precompact in $\h$.  This completes the proof.

\end{proof}




\section{Rigidity Theorem}

In this section, we show that the critical element from Proposition \ref{p43} does not exist and conclude the proof of our main 
result Theorem \ref{t41} (equivalently Theorem \ref{t01}).  In particular, we prove the following.

\begin{ppn}\label{p51} 
	Let $u$ be a global solution of \eqref{s81} such that the trajectory
	\begin{align*}
	K = \{ \vec u(t) : t \in \R \},
	\end{align*}
	is precompact in $\mathcal H := \h((-\infty,\infty); (r^2 + 1)^2 dr)$.  Then $\vec u = (0,0)$.  
\end{ppn}

We first note that for a solution $u$ as in Proposition \ref{p51}, we have the following uniform control of the energy 
in exterior regions. 

\begin{lem}\label{l52} 
	Let $u$ be as in Proposition \ref{p51}.  Then we have 
	\begin{align}
	\begin{split}\label{s51}
	\forall R \geq 0, \quad \lim_{|t| \rar \infty} \| \vec u(t) \|_{\h(|r| \geq R + |t|; (r^2 +1)^2 dr )} &=  0, \\
	\lim_{R \rar \infty} \left [ \sup_{t \in \R} \| \vec u(t) \|_{\h(|r| \geq R + |t|; (r^2 + 1)^2 dr)} \right ] &= 0.
	\end{split}
	\end{align}
\end{lem}

To prove that $\vec u = (0,0)$, we will show that $u$ is a finite energy static solution to \eqref{s81}.  

\begin{ppn}\label{p53}
	Let $u$ be as in Proposition \ref{p51}.  Then there exists 
	a static solution $U$ to \eqref{s81} such that $\vec u = (U,0)$. 
\end{ppn}

We will first show that $\vec u$ is equal to static solutions $(U_{\pm},0)$ to \eqref{s81} on $\pm r > 0$.  Since the proof for $r < 0$ is nearly identical, we only consider the case $r > 0$. 

\begin{ppn}\label{positive lemma}
	Let $u$ be as in Lemma \ref{p51}.  Then there exists a static solution $U_+$ to \eqref{s81} such that $\vec u(t,r) = (U_+(r),0)$ for all $t \in \R$
	and all $r > 0$.
\end{ppn}

\subsection{Proof of Proposition \ref{positive lemma}} 

Let $\eta > 0$ and let $u$ be as in Proposition \ref{p51}.  We will 
first show that $\vec u(0,r) = (U_+(r),0)$ on $r \geq \eta$ for some static solution $U_+$ to \eqref{s81}.   

We now introduce a function that will be integral in the proof.  Define 
\begin{align*}
u_e(t,r) := \frac{r^2 + 1}{r^2} u(t,r), \quad (t,r) \in \R \times (0,\infty).
\end{align*}
If $u$ solves \eqref{s81}, then $u_e$ solves 
\begin{align}\label{s52e} 
\p_t^2 u_e - \p^2_r u_e - \frac{4}{r} \p_r u_e + V_e(r) u_e = N_e (r,u_e), \quad t \in \R, r > 0,
\end{align}
where 
\begin{align}
V_e(r) = V(r) - \frac{2}{r^2 (r^2 + 1)}, \label{s52}
\end{align}
and $N_e(r,u_e) = F_e(r,u_e) + G_e(r,u_e)$ where 
\begin{align}
F_e(r,u_e) &= \frac{r^2 + 1}{r^2} F \left (r, \frac{r^2}{r^2 + 1} u_e \right ), \label{s53}\\
G_e(r,u_e) &= \frac{r^2 + 1}{r^2} G \left (r, \frac{r^2}{r^2 + 1} u_e \right ). \label{s54}
\end{align}
Note that for all $R > 0$, we have 
\begin{align}\label{s55}
\| \vec u_e(t) \|_{\h( r \geq R; r^4 dr)} \leq C(R) \| \vec u \|_{\h( r \geq R; (r^2 + 1)^2 dr)},
\end{align}
so that by Lemma \ref{l52}, $u_e$ inherits the compactness properties
\begin{align}
\begin{split}\label{s56}
\forall R > 0, \quad \lim_{|t| \rar \infty} \| \vec u_e(t) \|_{\h( r \geq R + |t|; r^4 dr)} = 0, \\
\lim_{R \rar \infty} \left [ \sup_{t \in \R} \| \vec u_e(t) \|_{\h( r \geq R + |t|; r^4 dr)} \right ] = 0.
\end{split}
\end{align}
We also note that due to \eqref{s47}--\eqref{s49} and the definition of $V_e, F_e,$ and $G_e$, we have for all $r > 0$,
\begin{align} 
| V_e(r) | &\lesssim r^{-4}, \label{s57} \\
|F_e(u_e,r)| &\lesssim r^{-3} |u_e|^{2}, \label{s58} \\
|G_e(u_e,r)| &\lesssim |u_e|^3, \label{s59}
\end{align} 
where the implied constants depend on the harmonic map $Q$. 

The proof that $\vec u = (U_+,0)$ on $r \geq \eta$ for some 
$U_+$ is split into three main steps.  In the first two steps, we determine the precise asymptotics of the associated ``Euclidean" solution $u_{e,0}(r)
:= u_e(0,r), u_{e,1}(r) := \p_t u_e(0,r)$, as $r \rar \infty$.  In particular, we show that there exists $\al \in \R$ such that 
\begin{align}
r^3 u_{e,0}(r) &= \al + O(r^{-1}), \label{s510} \\
r \int_r^\infty u_{e,1}(\rho) \rho d\rho &= O(r^{-1}), \label{s511}
\end{align}  
as $r \rar \infty$.  In the final step, we use this information to conclude 
the argument.   
For the remainder of this subsection we denote $\h (r \geq R; r^4 dr)$ simply by $\h(r \geq R)$ and the exterior region $\R^5 \backslash B(0,\eta)$ by $\R^5_*$.  

The key tool for 
establishing \eqref{s510} and \eqref{s511} is the following exterior energy estimate for radial free waves on Minkowski space $\R^{1+5}$.  

\begin{ppn}[Proposition 4.1 \cite{kls1}] \label{p54}
	Let $v$ be a radial solution to the free wave equation in $\R^{1 + 5}$ 
	\begin{align*}
	&\p_t ^2 v - \Delta v = 0, \quad (t,x) \in \R^{1+5}, \\
	&\vec v(0) = (f,g) \in \dot H^1 \times L^2 ( \R^5).
	\end{align*}
	Then for any $R > 0$,
	\begin{align}\label{s512}
	\max_{\pm} \inf_{\pm t \geq 0} \int_{r \geq R + |t|} |\nabla_{t,x} v(t,r)|^2 r^4 dr \geq \frac{1}{2} \| \pi^{\perp}_R (f,g) \|_{\h(r \geq R; r^4 dr)},
	\end{align}
	where $\pi_R = I - \pi_R^{\perp}$ is the orthogonal projection onto the plane 
	\begin{align*}
	P(R) = \mbox{span} \{ (r^{-3},0), (0,r^{-3}) \} 
	\end{align*}
	in $\h( r \geq R)$. The left--hand side of \eqref{s512} is identically 0 for data satisfying $(f,g) = (\al r^{-3}, \beta r^{-3})$ for 
	on $r \geq R$.
\end{ppn}  

We remark here that Proposition \ref{p54} states, quantitatively, that generic solutions to the free wave equation
on $\R^{1+5}$ emit 
a fixed amount of energy into regions exterior to light cones.  However, this property is very sensitive to dimension and 
in fact fails in the case $R = 0$ for general data $(f,g)$ in even dimensions (see \cite{cks}).  Proposition \ref{p54} has been generalized to all 
odd dimensions $d \geq 3$ in the work \cite{klls1}.  We note that the orthogonal projections $\pi_R$, $\pi_R^{\perp}$ are given by 
\begin{align}
\begin{split}\label{proj formulas}
&\pi_R(f, 0) = R^3r^{-3} f(R), \quad 
\pi_R(0, g) = R r^{-3} \int_R^{\infty} g( \rho)  \rho \, d \rho,\\
&\pi_R^{\perp}(f, 0) = f(r) - R^3r^{-3} f(R), \quad \pi_R^{\perp}(0, g) = g(r) - R r^{-3} \int_R^{\infty} g( \rho)  \rho \, d \rho,
\end{split}
\end{align}
and thus we have 
\begin{align}
\|\pi_R(f,g)\|_{\dot{H}^1 \times L^2(r>R)}^2 &= 3R^3 f^2(R) + R\left(\int_R^{\infty} rg(r) \, dr\right)^2, \label{proj norm 1} \\
\begin{split}\label{proj norm 2}
\|\pi_R^{\perp}(f, g)\|_{\dot{H}^1 \times L^2(r>R)}^2&= \int^{\infty}_R f_r^2(r) \, r^4 \, dr - 3R^3f^2(R)\\
& \quad + \int_R^{\infty} g^2(r) \, r^4 \, dr - a\left(\int_R^{\infty} rg(r) \, dr\right)^2.
\end{split}
\end{align}

We now proceed to the first step in proving Proposition \ref{p53}. 

\subsubsection{Step 1: Estimate for $\pi^{\perp}_R \vec u_e$ in $\h(r \geq R)$}
The first step in proving Proposition \ref{p53} is the following decay estimate for $\pi_R^{\perp} \vec u_e(t)$.  

\begin{lem}\label{l55}
	There exists $R_0 > 1$ such that for all $R \geq R_0$ and for all $t \in \R$ we have 
	\begin{align}\label{s513a}
	\begin{split}
	\| \pi_R^{\perp} \vec u_e(t) \|_{\h(r \geq R)}^2 \lesssim R^{-10/3} \| \pi_R \vec u_e(t) \|^2_{\h(r \geq R)} + 
	R^{-22/6} \| \pi_R \vec u_e(t) \|_{\h(r \geq R)}^4 + \| \pi_R \vec u_e(t) \|_{\h(r \geq R)}^6. 
	\end{split}
	\end{align}
\end{lem}


Since we are only interested in the behavior of $u$ in exterior regions $\{ r \geq R + |t| \}$, we first consider a modified Cauchy problem.  
In particular, we can, by finite speed of propagation, alter $V_e$, $F_e$, and $G_e$ in \eqref{s52e} on the interior region 
$\{ 1 \leq r \leq R + |t| \}$ without affecting the behavior of $\vec u_e$ on the exterior region $\{r \geq R + |t|\}$.

\begin{defn}\label{d56}
	For a function $f = f(r,u) : \R^5_* \times \R \rar \R$, we define for $R \geq \eta$,
	\begin{align*}
	f_R(t,r,u) := 
	\begin{cases}
	f(R + |t|, u) \quad &\mbox{if } \eta \leq r \leq R + |t|, \\
	f(r,u) \quad &\mbox{if } r \geq R + |t|.
	\end{cases}.
	\end{align*}
\end{defn}

We now consider solutions to a modified version of \eqref{s52e}: 
\begin{align}\label{s513}
\begin{split}
&\p_t^2 h - \p_r^2 h - \frac{4}{r} \p_r h = N_R (t,r,h), \quad (t, r) \in \R \times \R^5_*, \\
&\vec h(0) = (h_0,h_1) \in \h_0(r \geq \eta), 
\end{split}
\end{align}
where $\h_0(r \geq \eta) = \{ (u_0,u_1) \in \h(r \geq \eta) : u_0(\eta) = 0 \}$ and 
\begin{align*}
N_R(t,r,h) = - V_{e,R}(t,r) h + F_{e,R}(t,r,h) + G_e(t,r,h). 
\end{align*}
We note that from Definition \ref{d56} and \eqref{s57}, \eqref{s58}, and \eqref{s59}, we have that 
\begin{align}
| V_{e,R}(t,r) | &\lesssim 
\begin{cases}\label{s514}
(R + |t|)^{-4} \quad &\mbox{if } \eta \leq r \leq R + |t|, \\
r^{-4} \quad &\mbox{if } r \geq R + |t|, 
\end{cases}\\
|F_{e,R}(t,r,h)| &\lesssim 
\begin{cases}\label{s515} 
(R + |t|)^{-3}|h|^2 \quad &\mbox{if } \eta \leq r \leq R + |t|, \\
r^{-3}|h|^3 \quad &\mbox{if } r \geq R + |t|, 
\end{cases}\\
|G_e(r,h)| &\lesssim |h|^3, \label{s516} 
\end{align}

\begin{lem}\label{l57}
	There exist $R_0 > 0$ large and $\de_0 > 0$ small such that for all $R \geq R_0$ and all $(h_0,h_1) \in \h_0(r \geq \eta)$ with 
	\begin{align*}
	\| (h_0,h_1) \|_{\h_0(r \geq \eta)} \leq \de_0, 
	\end{align*}
	there exists a unique globally defined solution $h$ to \eqref{s513} such that 
	\begin{align}\label{s517}
	\| h \|_{L^3_tL^6_x(\R \times \R_*^5)} \lesssim \| \vec h(0) \|_{\h(r \geq \eta)}.
	\end{align}
	Moreover, if we define $h_L$ to be the solution to the free equation $\p_t^2 h_L - \Delta h_L = 0$, $(t,x) \in \R\times \R_*^5$, $\vec h_L(0) = (h_0,h_1)$, then 
	\begin{align}\label{s518}
	\begin{split}
	\sup_{t \in \R} \| \vec h(t) - \vec h_L(t) \|_{\h(r \geq \eta)} \lesssim R^{-5/3} \| \vec h(0) \|_{\h(r \geq \eta)} + 
	R^{-11/6} \| \vec h(0) \|_{\h(r \geq \eta)}^2 + \| \vec h(0) \|_{\h(r \geq \eta)}^3. 
	\end{split}
	\end{align}
\end{lem}

\begin{proof}
	The small data global well--posedness and spacetime estimate \eqref{s517} follow from standard contraction mapping and continuity 
	arguments using Strichartz estimates for free waves on $\R \times \R^5_*$ with Dirichlet boundary condition (see
	\cite{kls1}).  We now 
	prove \eqref{s518}.  By the Duhamel formula and Strichartz estimates we have
	\begin{align*}
	\sup_{t \in \R} \| \vec h(t) - \vec h_L(t) \|_{\h(r \geq \eta)} &\lesssim \| N_R(\cdot, \cdot, h) \|_{L^1_t L^2_x(\R \times \R^5_*)} \\
	&\lesssim \| V_{e,R} h \|_{L^1_t L^2_x(\R \times \R^5_*)} + \| F_{e,R}(\cdot, \cdot, h) \|_{L^1_t L^2_x(\R \times \R^5_*)} \\ &\: + 
	\| G_e(\cdot, h) \|_{L^1_t L^2_x(\R \times \R^5_*)}.
	\end{align*}
	The third term is readily estimated by \eqref{s517} and \eqref{s518}
	\begin{align*}
	\| G_e(\cdot, h) \|_{L^1_t L^2_x(\R \times \R^5_*)} \lesssim \| h^3 \|_{L^1_t L^2_x(\R \times \R^5_*)} 
	\lesssim \| \vec h(0) \|_{\h(r \geq \eta)}^3.  
	\end{align*}
	For the first term we have 
	\begin{align*}
	\| V_{e,R} h \|_{L^1_t L^2_x(\R \times \R^5_*)} \leq \| V_{e,R} \|_{L^{3/2}_t L^3_x(\R \times \R^5_*)}
	\| h \|_{L^3_t L^6_x(\R \times \R^5_*)} \lesssim \| V_{e,R} \|_{L^{3/2}_t L^3_x(\R \times \R^5_*)} \| \vec h(0) \|_{\h(r \geq \eta)}.
	\end{align*}
	By \eqref{s514} 
	\begin{align*}
	\| V_{e,R} \|^3_{L^3_x(\R^5_*)} \lesssim \int_0^{R + |t|} (R + |t|)^{-12} r^4 dr +  \int_{R + |t|}^\infty r^{-12} r^4 dr
	\lesssim (R + |t|)^{-7}.
	\end{align*}
	Hence, 
	\begin{align*}
	\| V_{e,R} \|_{L^{3/2}_t L^3_x(\R \times \R^5_*)}^{3/2} \lesssim \int (R + |t|)^{-7/2} dt \lesssim R^{-5/2}.   
	\end{align*}
	Thus, 
	\begin{align*}
	\| V_{e,R} h \|_{L^1_t L^2_x(\R \times \R^5_*)} \lesssim R^{-5/3} \| \vec h(0) \|_{\h(r \geq 1)}. 
	\end{align*}
	Similarly, using \eqref{s515} and \eqref{s517} we conclude that $\| F_{e,R}(\cdot, \cdot,h) \|_{L^1_t L^2_x(\R \times \R^5_*)} \lesssim R^{-11/6} \| h(0) \|^2_{\h(r \geq 1)}$ which proves \eqref{s518}. 
\end{proof}

\begin{proof}[Proof of Lemma \ref{l55}]
	We first prove Lemma \ref{l55} for $t = 0$.  For $R > \eta$, define the truncated initial data $\vec u_R(0) = (u_{0,R}, 
	u_{1,R}) \in \h_0(r \geq \eta)$ via 
	\begin{align}
	u_{0,R} &= 
	\begin{cases}\label{s519}
	u_e(0,r) \quad &\mbox{if } r \geq R, \\
	\frac{r - \eta}{R-\eta} u_e(0,R) \quad &\mbox{if } r < R,
	\end{cases} \\
	u_{1,R} &= 
	\begin{cases}\label{s520}
	\p_t u_e(0,r) \quad &\mbox{if } r \geq R, \\
	0 \quad &\mbox{if } r < R.
	\end{cases}
	\end{align}
	Note that for $R$ large, 
	\begin{align}\label{s521}
	\| \vec u_R (0) \|_{\h(r \geq \eta)} \lesssim \| \vec u_e(0) \|_{\h(r \geq R)}.
	\end{align}
	In particular, by Lemma \ref{l52}, there exists $R_0 > \eta$ such that for all $R \geq R_0$, $\| \vec u_R(0) \| \leq \de_0$
	where $\de_0$ is from Lemma \ref{l57}.  Let $u_R(t)$ be the solution to \eqref{s513} with initial data $(u_{0,R}, u_{1,R})$, and let
	$\vec u_{R,L}(t) \in \h_0(r \geq \eta)$ be the solution to the free wave equation $\p_t^2 u_{R,L} - \Delta u_{R,L}
	= 0,$ $(t,x) \in \R \times \R^5_*$, $\vec u_{R,L}(0) = (u_{0,R}, u_{1,R})$. By finite speed of propagation 
	\begin{align*}
	r \geq R + |t| \implies \vec u_R(t,r) = \vec u_e(t,r).
	\end{align*}  
	By Proposition \ref{p54}, for all $t \geq 0$ or for all $t \leq 0$, 
	\begin{align*}
	\| \pi_R^{\perp} \vec u_{R,L}(0) \|_{\h(r \geq R)} \lesssim \| \vec u_{R,L}(t) \|_{\h(r \geq R + |t|)}. 
	\end{align*}
	Suppose, without loss of generality, that the above bound holds for all $t \geq 0$.  By \eqref{s518} we conclude that for all $t \geq 0$
	\begin{align*}
	\| \vec u_e(t) \|_{\h(r \geq R + |t|)} &\geq \| \vec u_{R,L}(t) \|_{\h(r \geq R + |t|)} - \| \vec u_R(t) - \vec u_{R,L}(t) \|_{\h(r \geq \eta)} \\
	&\geq c \| \pi^{\perp}_R \vec u_{R,L}(0) \|_{\h(r \geq R)} - C \Bigl [ R^{-5/3} \| u_R(0) \|_{\h(r \geq \eta)} + 
	R^{-11/6} \| \vec u_R(0) \|_{\h(r \geq \eta)}^2 + \| \vec u_R(0) \|^3_{\h(r \geq \eta)} \Bigr ]. 
	\end{align*}
	Letting $t \rar \infty$ and using the decay property \eqref{s51} and the definition of $(u_{0,R}, u_{1,R})$, we conclude that 
	\begin{align*}
	\| \pi^{\perp}_R \vec u_e(0) \|_{\h(r \geq R)} \lesssim 
	R^{-5/3} \| u_e(0) \|_{\h(r \geq R)} + 
	R^{-11/6} \| \vec u_e(0) \|_{\h(r \geq R)}^2 + \| \vec u_e(0) \|^3_{\h(r \geq R)}.
	\end{align*}
	Note that $\| \vec u_e(0) \|_{\h(r \geq R)}^2 = \| \pi^{\perp}_R \vec u_e(0) \|_{\h(r \geq R)}^2 + 
	\| \pi_R \vec u_e(0) \|_{\h(r \geq R)}^2$.  Thus, if we take $R_0$ large enough to absorb terms involving 
	$\| \pi^{\perp}_R \vec u_e(0) \|_{\h(r \geq R)}$ into the left hand side in the previous estimate, we obtain for all
	$R \geq R_0$ 
	\begin{align*}
	\| \pi^{\perp}_R \vec u_e(0) \|_{\h(r \geq R)} \lesssim 
	R^{-5/3} \| \pi_R u_e(0) \|_{\h(r \geq R)} + 
	R^{-11/6} \| \pi_R \vec u_e(0) \|_{\h(r \geq R)}^2 + \| \pi_R \vec u_e(0) \|^3_{\h(r \geq R)},
	\end{align*}
	as desired.  This proves Lemma \ref{l55} for $t = 0$. 
	
	For general $t = t_0$ in \eqref{s513}, we first set 
	\begin{align*}
	u_{0,R,t_0} &= 
	\begin{cases}
	u_e(t_0,r) \quad &\mbox{if } r \geq R, \\
	\frac{r - \eta}{R-\eta} u(t_0,R) \quad &\mbox{if } r < R,
	\end{cases} \\
	u_{1,R,t_0} &= 
	\begin{cases}
	\p_t u_e(t_0,r) \quad &\mbox{if } r \geq R, \\
	0 \quad &\mbox{if } r < R.
	\end{cases}
	\end{align*}
	By \eqref{s56} 
	we can find $R_0 = R_0(\de_0,\eta)$ independent of $t_0$ such that for all 
	$R \geq R_0$
	\begin{align*}
	\| (u_{0,R,t_0}, u_{1,R,t_0}) \|_{\h(r \geq \eta)} \lesssim \| \vec u_e(t_0) \|_{\h(r \geq R)} \leq \de_0.
	\end{align*}
	The previous argument for $t_0 = 0$ repeated with obvious modifications yield \eqref{s513a} for $t = t_0$. 
\end{proof}

Before proceeding to the next step, it will be useful to reformulate the conclusion of Lemma \ref{l55}.  Define 
\begin{align}
\lam(t,r) &:= r^3 u_e(t,r), \label{s522}\\
\mu(t,r) &:= r \int_r^\infty \p_t u_e(t,\rho) \rho d\rho. \label{s523}
\end{align}
We denote $\lam(r) := \lam(0,r)$ and $\mu(r) := \mu(0,r)$. By \eqref{proj norm 1} and \eqref{proj norm 2} the functions $\lam$ and $\mu$ arise in the explicit
computation of $\pi_R^{\perp} \vec u(t)$ and $\pi_R \vec u(t)$ as follows:
\begin{align}
\| \pi^{\perp}_R \vec u_e(t) \|^2_{\h(r \geq R)} &= 
\int_R^\infty \left ( \frac{1}{r} \p_r \lam(t,r) \right )^2 dr + \int_R^\infty (\p_r \mu(t,r))^2 dr, \label{s524} \\
\| \pi_R \vec u_e(t) \|^2_{\h(r \geq R)} &= 3 R^{-3} \lam^2(t,R) + R^{-1} \mu^2(t,R). \label{s525} 
\end{align}  
Thus, Lemma \ref{l55} can be restated using $\lam,\mu$ in the following way. 

\begin{lem}\label{l58}
	Let $\mu,\lam$ be as in \eqref{s522} and \eqref{s523}.  Then there exists $R_0 \geq \eta$ such that for all $R > R_0$ and 
	for all $t \in \R$
	\begin{align*}
	\int_R^\infty \left ( \frac{1}{r} \p_r \lam(t,r) \right )^2 dr + \int_R^\infty (\p_r \mu(t,r))^2 dr
	\lesssim R^{-19/3}&  \lam^2(t,R) + R^{-29/3} \lam^4(t,R) + R^{-9} \lam^6(t,R) \\
	+ R^{-13/3}&  \mu^2(t,R) + R^{-17/3} \mu^4(t,R) + R^{-3} \mu^6(t,R). 
	\end{align*} 
\end{lem}


\subsubsection{Step 2: Asymptotics for $\vec u_e(0)$}  

In this step, we prove that $\vec u_e(0)$ has the asymptotic expansions \eqref{s510}, \eqref{s511} which we now formulate as 
a lemma.


\begin{lem}\label{p58}  
	Let $u_e$ be a solution to \eqref{s52e} which satisfies \eqref{s56}.   Let $\vec u_e(0) 
	= (u_{e,0},u_{e,1})$.  Then there exists $\alpha \in \R$ such that 
	\begin{align}
	r^3 u_{e,0}(r) &= \al + O(r^{-1}), \notag \\
	r \int_r^\infty u_{e,1}(\rho) \rho d\rho &= O(r^{-1}), \notag
	\end{align}  
	as $r \rar \infty$.  Equivalently, with $\lam$ and $\mu$ defined as in \eqref{s522} and \eqref{s523}, there exists $\alpha \in \R$ such that 
	\begin{align}
	\lam(r) &= \al + O ( r^{-1} ), \label{s526} \\
	\mu(r) &= O( r^{-1}). \label{s527}
	\end{align}
\end{lem}

The proof of Lemma \ref{p58} is split up into a few further lemmas.  First, we use Lemma \ref{l58} to prove the following difference 
estimate for $\lam$ and $\mu$. 

\begin{lem}\label{l59}
	Let $\de_1 \leq \de_0$ where $\de_0$ is from Lemma \ref{l57}.  Let $R_1 \geq R_0 > 1$ be large enough so that 
	for all $R \geq R_1$ and for all $t \in \R$
	\begin{align*}
	\| \vec u_e(t) \|_{\h(r \geq R)} &\leq \de_1, \\
	R^{-5/3} &\leq \de_1. 
	\end{align*}
	Then for all $r, r'$ with $R_1 \leq r \leq r' \leq 2r$ and for all $t \in \R$
	\begin{align} 
	\begin{split} \label{s528}
	|\lam(t,r) - \lam(t,r')| \lesssim r^{-5/3} &|\lam(t,r)| + r^{-10/3} |\lam(t,r)|^2 + r^{-3} |\lam(t,r)|^3 \\
	+ r^{-2/3} &|\mu(t,r)| + r^{-4/3} |\mu(t,r)|^2 + |\mu(t,r)|^3, \\
	\end{split} \\
	\begin{split} \label{s529}
	|\mu(t,r) - \mu(t,r')| \lesssim r^{-8/3} &|\lam(t,r)| + r^{-13/3} |\lam(t,r)|^2 + r^{-4} |\lam(t,r)|^3 \\
	+ r^{-5/3} &|\mu(t,r)| + r^{-7/3} |\mu(t,r)|^2 + r^{-1} |\mu(t,r)|^3. \\
	\end{split}
	\end{align}
\end{lem}

\begin{proof}
	By the fundamental theorem of calculus and Lemma \ref{l58} we have, for $r, r'$ such that $R_1 \leq r \leq r' \leq 2r$, 
	\begin{align*}
	|\lam(t,r) - \lam(t,r')|^2 &=
	\left ( \int_r^{r'} \p_\rho \lam(t,\rho) d\rho \right )^2 \\
	&\leq \left ( \int_r^{r'} \rho^2 d\rho \right ) \left ( \int_r^{r'} \left ( \frac{1}{\rho} \p_{\rho} \lam(t,\rho) \right )^2 d\rho \right ) \\
	&\lesssim r^3 \Bigl ( 
	r^{-19/3} \lam^2(t,r) + r^{-29/3} \lam^4(t,r) + r^{-9} \lam^6(t,r) \\
	&\:+ r^{-13/3}  \mu^2(t,r) + r^{-17/3} \mu^4(t,r) + r^{-3} \mu^6(t,r) 
	\Bigr )
	\end{align*}
	which proves \eqref{s528}. 
	
	Similarly, we have 
	\begin{align*}
	|\mu(t,r) - \mu(t,r')|^2
	&\leq r \left ( \int_r^{r'} ( \mu(t,\rho) )^2 d\rho \right ) \\
	&\lesssim r \Bigl ( 
	r^{-19/3} \lam^2(t,r) + r^{-29/3} \lam^4(t,r) + r^{-9} \lam^6(t,r) \\
	&\:+ r^{-13/3}  \mu^2(t,r) + r^{-17/3} \mu^4(t,r) + r^{-3} \mu^6(t,r) 
	\Bigr )
	\end{align*}
	which proves \eqref{s529}. 
\end{proof}

A simple consequence of Lemma \ref{l59} is the following.  

\begin{cor}\label{c510}
	Let $R_1$ be as in Lemma \ref{l59}.  Then for all $r,r'$ with $R_1 \leq r \leq r' \leq 2r$ and for all $t \in \R$ 
	\begin{align}
	|\lam(t,r) - \lam(t,r')| &\lesssim \de_1 |\lam(t,r)| + r \de_1 |\mu(t,r)|, \label{s530} \\
	|\mu(t,r) - \mu(t,r')| &\lesssim r^{-1} \de_1 |\lam(t,r)| + \de_1 |\mu(t,r)|. \label{s531}
	\end{align}
\end{cor}

Next we establish the following improved growth rate for $\lam$ and $\mu$. 

\begin{lem}\label{l511}
	For all $t \in \R$, 
	\begin{align}
	|\lam(t,r)| &\lesssim r^{1/6},  \label{s532}\\
	|\mu(t,r)| &\lesssim r^{1/18}. \label{s533}
	\end{align}
\end{lem}

\begin{proof}
	As in the proof of Lemma \ref{l55}, we only consider the case $t = 0$.  Fix $r_0 \geq R_1$.  By Corollary 
	\ref{c510}, 
	\begin{align}
	|\lam(2^{n+1}r)| &\leq  (1 + C_1 \de_1) |\lam(2^n r_0)| + (2^n r_0) C_1 \de_1 |\mu(2^n r_0)|, \label{s534} \\
	|\mu(2^{n+1}r)| &\leq  (1 + C_1 \de_1) |\mu(2^n r_0)| + (2^n r_0)^{-1} C_1 \de_1 |\lam(2^n r_0)|. \label{s535} 
	\end{align}
	If we define $a_n := |\mu(2^n r_0)|$ and $b_n := (2^n r_0)^{-1} |\lam(2^n r_0)|$, then \eqref{s534} and \eqref{s535} imply
	\begin{align*}
	a_{n+1} + b_{n+1} \leq \left ( 1 + 2 C_1 \de_1 \right ) (a_n + b_n).
	\end{align*} 
	By induction
	\begin{align*}
	a_n + b_n \leq \left ( 1 + 2 C_1 \de_1 \right )^n (a_0 + b_0).
	\end{align*}
	Choose $\de_1$ so small so that $1 + 2 C_1 \de_1 < 2^{1/18}$.  We conclude that 
	\begin{align}\label{s536a}
	a_n \leq C (2^n r_0)^{1/18},
	\end{align}
	where $C = C(r_0)$.  This proves \eqref{s533} for $r = 2^n r_0$.  Define
	\begin{align*}
	c_n = |\lam(2^n r_0)| = (2^n r_0) b_n. 
	\end{align*}
	Using \eqref{s536a} and \eqref{s528} we have, for some $C = C(r_0)$, 
	\begin{align*}
	c_{n+1} \leq (1 + C_1 \de_1) c_n + C \de_1 (2^n r_0)^{1/6}.
	\end{align*}
	By induction 
	\begin{align*}
	c_n &\leq (1 + C_1 \de_1)^n c_0 + C r_0^{1/6} \sum_{k = 1}^n (1 + C_1 \de_1)^{n-k} 2^{(k-1)/6} \\
	&\leq C (2^n r_0)^{1/6}.
	\end{align*}
	This proves \eqref{s532} for $r = 2^n r_0$. 
	
	To establish \eqref{s532} and \eqref{s533} for general $r$, let $r \geq r_0$ such that for some $n \geq 0$, $2^n r_0 \leq r \leq 2^{n+1}r_0$.  Then 
	applying \eqref{s528} to the pair $2^n r_0, r$, we conclude that 
	\begin{align*}
	|\lam(r)| &\leq |\lam(2^n r_0)| + |\lam(2^n r_0) - \lam(r)| \\
	&\leq C (2^n r_0)^{1/6} + \Bigl [ (2^n r_0)^{-5/3} (2^n r_0)^{1/6} + (2^n r_0)^{-10/3} (2^n r_0)^{1/3} + (2^n r_0)^{-3}
	(2^n r_0)^{1/2} \Bigr ] \\
	&\leq C (2^n r_0)^{1/6} \\
	&\leq C r^{1/6},
	\end{align*}
	where $C = C(r_0)$.  This proves \eqref{s532}.  A similar argument also establishes the bound \eqref{s533} for all $r \geq r_0$
	as well. 
\end{proof}

We now show that for each $t \in \R$, $\mu(t,r)$ has a limit $\beta(t)$ as $r \rar \infty$.  


\begin{lem}\label{l512}
	For all $t \in \R$, there exists $\beta(t) \in \R$ such that 
	\begin{align}\label{s536}
	|\mu(t,r) - \beta(t) | \leq C r^{-1}. 
	\end{align}
	The constant $C > 0$ is uniform in time.  
\end{lem}

\begin{proof}
	We only consider the case $t = 0$.  The general case follows, again, by using the decay of the trajectory $\vec u_e(t)$
	on exterior regions.  
	Let $R_1 > 1$ be as in Lemma \ref{l59}, and fix $r_0 \geq R_1$.  Then Lemma \ref{l511} and \eqref{s529} yield the estimate 
	\begin{align*}
	|\mu(2^{n+1} r_0) - \mu(2^n r_0)|
	&\lesssim (2^n r_0)^{-8/3} (2^n r_0)^{1/6} + (2^n r_0)^{-13/3}(2^n r_0)^{1/3} +  (2^n r_0)^{-4} (2^n r_0)^{1/2} \\
	&\:+ (2^n r_0)^{-5/3} (2^n r_0)^{1/18} + (2^n r_0)^{-7/3}(2^n r_0)^{1/9} +  (2^n r_0)^{-1} (2^n r_0)^{1/6} \\
	&\lesssim (2^n r_0)^{-5/6} \\
	&\lesssim 2^{-5n/6}, 
	\end{align*}
	where the implied constant is uniform in $r_0$.  Thus,
	\begin{align*}
	\sum_{n \geq 0} |\mu(2^{n+1}r_0) - \mu(2^n r_0)| \lesssim 1,
	\end{align*}
	with a constant uniform in $r_0$.  This implies that there exists $\beta = \beta \in \R$ such that 
	$\lim_{n \rar \infty} \mu(2^n r_0) = \beta$. Moreover, the sequence $\{ \mu(2^n r_0) \}_n$ is bounded by a constant depending only on $r_0$
	since 
	\begin{align*}
	|\mu(2^n r_0)| &\leq |\mu(r_0)| + |\mu(2^n r_0) - \mu(r_0)| \\&= |\mu(r_0)| + \left | \sum_{k = 0}^{n-1} (
	\mu(2^{k+1}r_0) - \mu(2^k r_0) ) \right | \\
	&\leq |\mu(r_0)| + C_1 \sum_{n \geq 0} 2^{-5n/6} \leq C(r_0). 
	\end{align*}
	Inserting this bound into the difference estimate \eqref{s529} improves the previous bound on $|\mu(2^{n+1}r_0) - 
	\mu(2^n r_0)|$ to 
	\begin{align}\label{s536b}
	|\mu(2^{n+1} r_0) - \mu(2^n r_0) | \leq C (2^n r_0)^{-1},
	\end{align} 
	where $C = C(r_0)$.  Now let $r \geq r_0$ such that $2^n r_0 \leq r \leq 2^{n+1} r_0$.  By Lemma \ref{l511}, \eqref{s529},
	and our improved bound \eqref{s536b}, we have that 
	\begin{align*}
	|\mu(r) - \beta| &\leq |\mu(r) - \mu(2^n r_0)| + |\beta - \mu(2^n r_0)| \\
	&= |\mu(r) - \mu(2^n r_0)| + \left | \sum_{k \geq n} (\mu(2^{k+1} r_0) - \mu(2^k r_0)) \right | \\
	&\lesssim (2^n r_0)^{-1} + \sum_{k \geq n} (2^k r_0)^{-1} \\
	&\lesssim (2^n r_0)^{-1} \\
	&\lesssim r^{-1}. 
	\end{align*}   
	This proves \eqref{s536}. 
\end{proof}

We now conclude the proof of the bound \eqref{s527} in Proposition \ref{p58}. 

\begin{lem}\label{l513}
	Let $\beta(t)$ be as in Lemma \ref{l512}.  Then $\beta(t) \equiv 0$. 
\end{lem}

\begin{proof}
	The proof follows in two steps. 
	
	\emph{Step 1.} We first show that $\beta(t)$ is constant in time. By Lemma \ref{l512} and the definition of $\mu$, we have 
	shown that 
	\begin{align*}
	\beta(t) = r \int_r^\infty \p_t u_e(t,\rho) \rho d\rho + O(r^{-1}),
	\end{align*}
	where the $O(r^{-1})$ is uniform in time. Fix $t_1 < t_2$.  Since $u_e$ solves \eqref{s52e}, we have for $R \geq R_1$, 
	\begin{align*}
	\beta(t_2) - \beta(t_1) &= \frac{1}{R} \int_R^{2R} \beta(t_2) - \beta(t_1) ds \\
	&= \frac{1}{R} \int_R^{2R} s \int_s^\infty [ \p_t u_e(t_2,r) - \p_t u_e(t_1,r) ] r dr ds + O(R^{-1}) \\
	&= \frac{1}{R} \int_R^{2R} s \int_s^\infty \int_{t_1}^{t_2} \p^2_t u_e(t, r) dt r dr ds + O(R^{-1}) \\ 
	&= \frac{1}{R} \int_R^{2R} s \int_s^\infty \int_{t_1}^{t_2} r^{-3} \p_r (r^4 \p_r u_e(t,r) )dt dr ds \\
	&\:+ \frac{1}{R} \int_R^{2R} s \int_s^\infty \int_{t_1}^{t_2} [-r V_e(r) u_e(t,r) + r N_e(r,u_e(t,r)) ]dt dr ds + O(R^{-1}) \\
	&=: A + B + O(R^{-1}).
	\end{align*}
	We first estimate $B$.  Recall that $\lam(t,r) = r^3 u_e(t,r)$.  By \eqref{s532}, 
	\begin{align}\label{s537}
	|u_e(t,r)| \lesssim r^{-17/6},
	\end{align}
	uniformly in $t$.  This estimate, \eqref{s57}, \eqref{s58}, and \eqref{s59}, imply that 
	\begin{align*}
	|B| &\lesssim (t_2 - t_1) \frac{1}{R} \int_R^{2R} s \int_s^\infty \left [ r^{-35/6} + r^{-23/3} + r^{-15/2} \right ] dr ds \\
	&\lesssim (t_2 - t_1) \frac{1}{R} \int_R^{2R} s \int_s^\infty r^{-5} dr ds \\
	&\lesssim (t_2 - t_1) R^{-3}. 
	\end{align*}
	For $A$, we repeatedly use integration by parts and use \eqref{s537} to conclude that 
	\begin{align*}
	A &= \frac{3}{R} \int_{t_1}^{t_2} \int_R^{2R} s \int_s^\infty \p_r u_e(t,r) dr ds dt- 
	\frac{1}{R} \int_{t_1}^{t_2} \int_R^{2R} s^2 \p_s u_e(t,s) ds dt \\
	&=-\frac{3}{R} \int_{t_1}^{t_2} \int_R^{2R} s \p_s u_e(t,s) ds dt - 
	\frac{1}{R} \int_{t_1}^{t_2} \int_R^{2R} s^2 \p_s u_e(t,s) ds dt \\
	&= -\frac{1}{R} \int_{t_1}^{t_2} \int_R^{2R} s \p_s u_e(t,s) ds dt + \int_{t_1}^{t_2} [R u_e(t,R) - 2R u_e(t,2R) ] dt\\
	&= O(t_2 - t_1) O(R^{-11/6}). 
	\end{align*}
	In summary, we have that $|A| + |B| \lesssim R^{-1}(t_2 - t_1)$ so that 
	\begin{align*}
	\beta(t_2) - \beta(t_1) = O(t_2 - t_1) O(R^{-1}) + O(R^{-1}).
	\end{align*}
	Letting $R \rar \infty$ implies that $\beta(t_2) = \beta(t_1)$ as desired. This completed Step 1. 
	
	\emph{Step 2.} By Step 1, we have that $\beta(t) = \beta(0) =: \beta$ for all $t \in \R$.  In this step, we show that 
	$\beta = 0$ which concludes the proof of Lemma \ref{l513}.  By Lemma \ref{l512} and Step 1, for all $R \geq R_1$ and for all $t \in \R$ we have
	\begin{align*}
	\beta = R \int_R^\infty \p_t u_e(t,r) r dr + O(R^{-1}), 
	\end{align*}
	where the $O(R^{-1})$ term is uniform in time.  Integrating the previous expression from $0$ to $T$, dividing by $T$, 
	and using \eqref{s537} yield for all $T > 0$ and $R \geq R_1$
	\begin{align*}
	\beta &= \frac{R}{T} \int_R^\infty \int_0^T \p_t u_e(t,r) dt r dr + O(R^{-1}) \\
	&= \frac{R}{T} \int_R^\infty [ u_e(T,r) - u_e(0,r) ] r dr + O(R^{-1}) \\
	&= O \left ( \frac{R^{1/6}}{T} \right ) + O (R^{-1}).
	\end{align*}
	If we now choose $R = T$ and let $T \rar \infty$, we conclude that $\beta = 0$ as desired.  This concludes Step 2 and 
	the proof of Lemma \ref{l513}.
\end{proof}

\begin{lem}\label{l514} 
	There exists $\alpha \in \R$ such that 
	\begin{align*}
	|\lam(r) - \alpha| \lesssim r^{-1}.
	\end{align*}
\end{lem}

\begin{proof}
	The proof of Lemma \ref{l514} is very similar to the proof for Lemma \ref{l512} and we only sketch it. 
	Fix $r_0 \geq R_1$.  By Lemma \ref{l513}, the difference estimate \ref{s528}, and the growth estimate \ref{s532}
	we have 
	\begin{align*}
	|\lam(2^{n+1}r_0) - v_0(2^n r_0) | &\lesssim (2^n r_0)^{-5/3} (2^n r_0)^{1/6} + (2^n r_0)^{-10/3} (2^n r_0)^{1/3} +
	(2^n r_0)^{-3} (2^n r_0)^{1/2} \\
	&\:+ (2^n r_0)^{-2/3} (2^n r_0)^{-1} + (2^n r_0)^{-4/3} (2^n r_0)^{-2} + (2^n r_0)^{-3} \\
	&\lesssim (2^n r_0)^{-3/2}. 
	\end{align*}
	Thus, 
	\begin{align*}
	\sum_{n \geq 0} |\lam(2^{n+1}r_0) - \lam(2^n r_0) | < \infty,  
	\end{align*}
	so that there exists $\alpha \in \R$ such that $\lim_n \lam(2^n r_0) = \alpha$.  As in the proof of Lemma \ref{l512}, we then 
	use the fact that the sequence $\{ \lam(2^n r_0) \}_n$ is bounded and the difference estimate 
	\ref{s532} to conclude that for $r \geq r_0$ 
	\begin{align*}
	|\lam(r) - \alpha| \lesssim r^{-1}
	\end{align*}
	as desired. 
\end{proof}

We have shown that there exists $\alpha \in \R$ such that
\begin{align*}
r^3 u_e(0,r) &= \alpha + O(r^{-1}), \\
r\int_r^\infty \p_t u_e(0,\rho) \rho d\rho &= O(r^{-1}),
\end{align*}
as $r \rar \infty$. In the case $\alpha = 0$, we conclude that 
$\vec u(0) = (0,0)$ on $r \geq \eta$. 

\begin{lem}\label{l515}
	Let $\alpha$ be as in Lemma \ref{l514}.  If $\alpha = 0$, then $\vec u(0,r) = (0,0)$ for $r \geq \eta$. 
\end{lem}

\begin{proof}
	The proof of Lemma \ref{l515} is split into two steps. 
	
	\begin{clm}\label{comp clm}
		Let $\alpha$ be as in Lemma \ref{l514}. If $\alpha = 0$, then 
		$\vec u(0,r)$ is compactly supported in $r$.
	\end{clm}
	
	\begin{proof}[Proof of Claim \ref{comp clm}]
		Since $\alpha = 0$, 
		\begin{align}
		\lam(r) &= O(r^{-1}), \label{s538} \\
		\mu(r) &= O(r^{-1}). \label{s539}
		\end{align}
		Then, for $r_0 \geq R_1$, we have 
		\begin{align}
		|\lam(2^n r_0)| + |\mu(2^n r_0)| \lesssim (2^n r_0)^{-1}. \label{s540}
		\end{align}
		By the difference estimate \eqref{s528} and the growth estimates \eqref{s538}, \eqref{s539}, we conclude that
		\begin{align*}
		|\lam(2^{n+1}r_0)| &\geq (1 - C_1 \de_1) |\lam(2^n r_0)| - C_1 (2^n r_0)^{-2/3} |\mu(2^n r_0)|, \\
		|\mu(2^{n+1} r_0)| &\geq (1 - C_1 \de_1) |\mu(2^n r_0)| - C_1 (2^n r_0)^{-8/3} |\lam(2^n r_0)|.
		\end{align*}
		The constant $C_1$ is independent of $\de_1$ and $r_0$.  Thus
		\begin{align*}
		|\lam(2^{n+1} r_0)| + |\mu(2^{n+1}r_0)| \geq \left (1 - C_1 \de_1 - C_1 r_0^{-2/3} \right )
		\left [ |\lam(2^{n}r_0)| + |\mu(2^{n}r_0)| \right ].
		\end{align*}
		Take $r_0$ large and $\de_1$ small enough so that $C_1(\de_1 + r_0^{-2/3}) < 1/4$.  Then 
		\begin{align*}
		|\lam(2^{n+1} r_0)| + |\mu(2^{n+1}r_0)| \geq \frac{3}{4}
		\left [ |\lam(2^{n}r_0)| + |\mu(2^{n}r_0)| \right ].
		\end{align*}
		Proceeding inductively we obtain 
		\begin{align*}
		|\lam(2^{n+1} r_0)| + |\mu(2^{n+1}r_0)| \geq \left (\frac{3}{4} \right)^n \left [
		|\lam(r_0)| + |\mu(r_0)| \right ].
		\end{align*}
		By \eqref{s540} we conclude that 
		\begin{align*}
		\left (\frac{3}{4} \right)^n \left [
		|\lam(r_0)| + |\mu(r_0)| \right ] \lesssim (2^n r_0)^{-1}
		\end{align*}
		which implies 
		\begin{align*}
		\left (\frac{3}{2} \right)^n \left [
		|\lam(r_0)| + |\mu(r_0)| \right ] \lesssim 1,
		\end{align*}
		where the implied constant is uniform in $n$.  Hence $(\lam(r_0),\mu(r_0)) = (0,0)$.
		By \eqref{s525} $\| \pi_{r_0} \vec u_e(0) \|_{\h(r \geq r_0)} = 0$.  By Lemma \ref{l58} $\| \pi_{r_0}^{\perp} \vec u_e(0) 
		\|_{\h(r \geq r_0)} = 0$.  Hence $\| \vec u_e(0) \|_{\h(r \geq r_0)} = 0$. Since $\lim_{r \rar \infty} u_{e,0}(r) = 0$, we conclude
		that $(u_{e,0}(r),u_{e,1}(r)) = (0,0)$ for $r \geq r_0$.  Since $u(t,r) = 
		r^2 \la r \ra^{-2} u_e(t,r)$, we conclude that $\vec u(0,r) = (0,0)$ on $r \geq r_0$ as well.  This concludes the proof of the claim.
	\end{proof}   
	
	\begin{clm}\label{finish clm}
		If $\vec u(0,r)$ is compactly supported in $(\eta,\infty)$, then $\vec u(t,r) = (0,0)$ on $(\eta,\infty)$. 
	\end{clm}
	
	\begin{proof}[Proof of Claim \ref{finish clm}]
		Suppose not, i.e. $\vec u(0,r)$ is not identically 0 on $(\eta,\infty)$  Then  $(u_{e,0},u_{e,1})$ is not identically 0 on $(\eta,\infty)$. Define
		\begin{align*}
		\rho_0 = \inf \left \{ \rho : \| \vec u_e(0) \|_{\h(r \geq \rho)} = 0 \right \}.
		\end{align*}
		By our assumptions we have that $\eta < \rho < \infty$. Let $\rho_1 = \rho_1(\de_1)$ be so close to $\rho_0$ so that 
		$\eta < \rho_1 < \rho_0$ and 
		\begin{align*}
		0 < \| \vec u_e(0) \|_{\h(r \geq \rho_1)}^2 \leq \de_1^2, 
		\end{align*}
		where $\de_1$ is as in Lemma \ref{l59}.  
		
		By \eqref{s524} and \eqref{s525} and our choice of $\rho_1$, we have that 
		\begin{align}
		\begin{split}\label{s542a} 
		\int_{\rho_1}^\infty &\left ( \frac{1}{r} \p_r \lam(r) \right )^2 dr + \int_{\rho_1}^\infty (\p_r \mu(r))^2 dr \\
		&+ 3 \rho_1^{-3} \lam^2(\rho_1) + \rho_1^{-1} \mu^2(\rho_1) = \| \pi_{\rho_1}^{\perp} \vec u_e(0) \|^2_{\h(r \geq \rho_1)}
		+ \| \pi_{\rho_1} \vec u_e(0) \|^2_{\h(r \geq \rho_1)} < \de_1^2. 
		\end{split}
		\end{align}
		If we define $(u_{0,\rho_1}, u_{1,\rho_1})$ as in \eqref{s519} and \eqref{s520}, 
		we have for $\rho_1$ close to $\rho_0$, 
		\begin{align*}
		\| (u_{0,\rho_1}, u_{1,\rho_1}) \|_{\h(r \geq \eta)} \leq C(\rho_0) \| \vec u_e(0) \|_{\h(r \geq \rho_1)} \leq \de_1.
		\end{align*}
		Thus, by Lemma \ref{l58} we obtain 
		\begin{align}
		\begin{split}\label{s542}
		\int_{\rho_1}^\infty \left ( \frac{1}{r} \p_r \lam(r) \right )^2 dr + \int_{\rho_1}^\infty (\p_r \mu(r))^2 dr
		&\lesssim \rho_1^{-19/3}  \lam^2(\rho_1) + \rho_1^{-29/3} \lam^4(\rho_1) + R^{-\rho_1} \lam^6(\rho_1) \\
		&\:+ \rho_1^{-13/3} \mu^2(\rho_1) + \rho_1^{-17/3} \mu^4(\rho_1) + \rho_1^{-3} \mu^6(\rho_1) \\
		&\leq C(\rho_0) \left [|\lam(\rho_1)|^2 + |\mu(\rho_1)|^2 \right ],
		\end{split}
		\end{align}
		as long as $\rho_1$ is sufficiently close to $\rho_0$. Using the previous estimate and the fact that
		$\lam(\rho_0) = 0$, we argue as in the proof of Lemma 
		\ref{l59} to obtain
		\begin{align*}
		|\lam(\rho_1)|^2 &= |\lam(\rho_1) - \lam(\rho_0)|^2 \\
		&\leq (\rho_0 - \rho_1)^3 \left ( \int_{\rho_1}^{\rho_0} \left ( \frac{1}{r} \p_r \lam(r) \right )^2 dr \right ) \\
		&\leq C(\rho_0) (\rho_0 - \rho_1)^3 \left [|\lam(\rho_1)|^2 + |\mu(\rho_1)|^2 \right ].
		\end{align*}
		Similarly, 
		\begin{align*}
		|\mu(\rho_1)|^2 \leq C(\rho_0) (\rho_0 - \rho_1) \left [|\lam(\rho_1)|^2 + |\mu(\rho_1)|^2 \right ].
		\end{align*}
		We conclude that for all $\rho_1$ close to $\rho_0$,
		\begin{align*}
		|\lam(\rho_1)|^2 + |\mu(\rho_1)|^2 \leq 2 C(\rho_0) (\rho_0 - \rho_1) \left [|\lam(\rho_1)|^2 + |\mu(\rho_1)|^2 \right ]
		\end{align*}
		Thus, $(\lam(\rho_1),\mu(\rho_1)) = (0,0)$ for $\rho_1 < \rho_0$ close to $\rho_0$.  By \eqref{s542a} and \eqref{s542} we conclude that 
		$\| \vec u_e(0) \|_{\h(r \geq
			\rho_1)} = 0$.  This contradicts our definition of $\rho_0$ and the fact that $\rho_1 < \rho_0$. Thus, $\rho_0 = \eta$ and 
		$\| \vec u_e(0) \|_{\h(r \geq \eta)} = 0$ as desired. 
	\end{proof}
	
	Lemma \ref{l515} now follows immediately from Claim \ref{comp clm} and Claim 
	\ref{finish clm}.
\end{proof}

Using the previous arguments we can, in fact, conclude more in the case $\alpha = 0$. 

\begin{lem}\label{allt lem}
	Let $\alpha$ be as in Lemma \ref{l514}.  If $\alpha = 0$, then 
	\begin{align*}
	\vec u(t,r) = (0,0)
	\end{align*}
	for all $t \in \R$ and $r > 0$. 
\end{lem}

\begin{proof}
	By Lemma \ref{l515} we know that if $\alpha = 0$ then $\vec u(0,r) = (0,0)$ 
	on $\{ r \geq \eta\}$.  
	By finite speed of propagation, we conclude that 
	\begin{align}\label{finite speed}
	\vec u(t,r) = (0,0) \quad \mbox{ on } \{ r \geq |t| + \eta \}. 
	\end{align}
	Let $t_0 \in \R$ be arbitrary and define $u_{t_0}(t,r) = u(t+t_0,r)$.  Then $\vec u_{t_0}$ inherits the following compactness property from $\vec u$: 
	\begin{align*}
	\forall R \geq 0, \quad \lim_{|t| \rar \infty} \| \vec u_{t_0}(t) \|_{\h(r \geq R + |t|; (r^2 +1)^2 dr )} &=  0, \\
	\lim_{R \rar \infty} \left [ \sup_{t \in \R} \| \vec u_{t_0}(t) \|_{\h(r \geq R + |t|; (r^2 + 1)^2 dr)} \right ] &= 0,
	\end{align*}
	and by \eqref{finite speed} $\vec u_{t_0}(0,r)$ is supported in $\{ 0 < r 
	\leq \eta + |t_0| \}$. By Claim \ref{finish clm} applied to $\vec u_{t_0}$ we conclude that $\vec u_{t_0}(0,r) = (0,0)$ on $r \geq \eta$.  Since $t_0$ was arbitrary, we conclude that
	\begin{align*}
	\vec u(t_0,r) = (0,0) \quad \mbox{on } \{ r \geq \eta \},
	\end{align*}  
	for any $t_0 \in \R$.  Since $\eta > 0$ was arbitrarily fixed in the beginning of this subsection, we conclude that 
	\begin{align*}
	\vec u(t,r) = (0,0)
	\end{align*}
	for all $t \in \R$ and $r > 0$ as desired. 
\end{proof}

\subsubsection{Step 3: Conclusion of the proof of Proposition \ref{positive lemma}}

We now conclude the proof of Proposition \ref{positive lemma} by proving the following. 

\begin{lem}\label{l516}
	Let $\alpha$ be as in Lemma \ref{l514}. As before, we denote the unique finite 
	energy harmonic map of degree $n$ by $Q$ and recall that there exists a unique 
	$\alpha_n > 0$ such that 
	\begin{align*}
	Q(r) = n\pi - \alpha_n r^{-2} + O(r^{-4}).
	\end{align*}
	Let $Q_{\al - \al_n}$ denote the unique solution to \eqref{ode} with the property that 
	\begin{align}\label{s91}
	Q_{\alpha - \alpha_n}(r) = n\pi + (\alpha - \alpha_n) r^{-2} + O(r^{-4})
	\end{align} 
	as $r \rar \infty$.  Note that $Q_{\al - \al_n}$ exists and is unique by 
	Proposition \ref{prescribe}.  Define a static solution $U_+$ to \eqref{s81} via
	\begin{align*}
	U_+(r) = \la r \ra^{-1} \bigl ( Q_{\alpha - \al_n}(r) - Q(r) \bigr ).
	\end{align*}
	Then 
	\begin{align*}
	\vec u(t,r) = (U_+(r),0)
	\end{align*}
	for all $t \in \R$ and $r > 0$.
\end{lem}

\begin{proof}
	Lemma \ref{l516} follows from the proof of the $\al = 0$ case and a change of variables.  Let $Q_{\alpha - \alpha_n}$ be as in the statement of the lemma.  We define
	\begin{align}
	\begin{split}\label{s92}
	u_{\al}(t,r) &:= u(t,r) - \la r \ra^{-1} \left ( Q_{\alpha - \alpha_n}(r) - Q(r)  \right ) \\
	&= u(t,r) - U_+(r)
	\end{split}
	\end{align}
	and observe that $u_{\al}$ solves 
	\begin{align*}
	\p_t^2 u_{\al} - \p_r^2 u_\al - \frac{4r}{r^2 + 1} \p_r u_\al + V_\al(r) u_\al = N_\al(r,u_\al), 
	\end{align*}
	where the potential $V_\al$ is given by 
	\begin{align}\label{s93}
	V_\al(r) = \la r \ra^{-4} + 2 \la r \ra^{-2} ( \cos 2 Q_{\alpha - \alpha_n} - 1 ), 
	\end{align}
	and $N_\al(r,u) := F_\al(r,u) + G_\al(r,u)$ with 
	\begin{align}
	\begin{split}\label{s94}
	F_\al(r,u) &:= 2 \la r \ra^{-3} \sin^2 (\la r \ra u) \sin 2 Q_{\alpha - \alpha_n} , \\
	G_\al(r,u) &:= \la r \ra^{-3} \left [ 2 \la r \ra u - \sin (2 \la r \ra u) \right ] \cos 2 Q_{\alpha - \alpha_n} .
	\end{split}
	\end{align}
	By \eqref{s91}, the potential $V_\al$ is smooth and satisfies
	\begin{align*}
	V_\al(r) = \la r \ra^{-4} + O ( \la r \ra^{-6} ),
	\end{align*}
	as $r \rar \infty$ and the nonlinearities $F_\al$ and $G_\al$ satisfy 
	\begin{align*}
	|F_\al(r,u)| &\lesssim \la r \ra^{-3} |u|^2, \\
	|G_\al(r,u)| &\lesssim |u|^3, 
	\end{align*}
	for $r \geq 0$.  Moreover, by \eqref{s92} we see that $\vec u_{\alpha}$ inherits the compactness property from $\vec u$:
	\begin{align}
	\begin{split}\label{comp prop al}
	\forall R \geq 0, \quad \lim_{|t| \rar \infty} \| \vec u_{\al}(t) \|_{\h( r \geq R + |t|; (1+r^2)^2 dr)} = 0, \\
	\lim_{R \rar \infty} \left [ \sup_{t \in \R} \| \vec u_{\al}(t) \|_{\h( r \geq R + |t|; (1+r^2)^2 dr)} \right ] = 0.
	\end{split}
	\end{align}
	
	Let $\eta > 0$.  We now define for $r \geq \eta$,
	\begin{align}\label{s96}
	u_{\al,e}(t,r) := \frac{r^2 + 1}{r^2} u_{\al}(t,r)
	\end{align}
	and note that $u_{\al,e}$ satisfies an equation analogous to $u_e$:
	\begin{align}\label{s97} 
	\p_t^2 u_{\al,e} - \p^2_r u_{\al,e} - \frac{4}{r} \p_r u_{\al,e} + V_{\al,e}(r) u_{\al,e} = N_{\al,e} (r,u_{\al,e}), 
	\quad t \in \R, r \geq \eta,
	\end{align}
	where 
	\begin{align*}
	V_{\al,e}(r) = V_\al(r) - \frac{2}{r^2 (r^2 + 1)},
	\end{align*}
	and $N_{\al,e}(r,u_e) = F_{\al,e}(r,u_e) + G_{\al,e}(r,u_e)$ where 
	\begin{align*}
	F_{\al,e}(r,u_{\al,e}) &= \frac{r^2 + 1}{r^2} F_\al \left (r, \frac{r^2}{r^2 + 1} u_{\al,e} \right ), \\
	G_{\al,e}(r,u_{\al,e}) &= \frac{r^2 + 1}{r^2} G_\al \left (r, \frac{r^2}{r^2 + 1} u_{\al,e} \right ).
	\end{align*}
	In particular, we have the analogues of \eqref{s57}, \eqref{s58}, and \eqref{s59}: for all $r \geq \eta$,
	\begin{align} 
	| V_{\al,e}(r) | &\lesssim r^{-4}, \label{s98} \\
	|F_{\al,e}(r,u)| &\lesssim r^{-3} |u|^{2}, \label{s99} \\
	|G_{\al,e}(r,u)| &\lesssim |u|^3. \label{s910}
	\end{align} 
	Moreover, $u_{\al,e}$ inherits the following compactness properties from $u_\alpha$:
	\begin{align}
	\begin{split}\label{s911}
	\forall R \geq \eta, \quad \lim_{|t| \rar \infty} \| \vec u_{\al,e}(t) \|_{\h( r \geq R + |t|; r^4 dr)} = 0, \\
	\lim_{R \rar \infty} \left [ \sup_{t \in \R} \| \vec u_{\al,e}(t) \|_{\h( r \geq R + |t|; r^4 dr)} \right ] = 0.
	\end{split}
	\end{align}
	Finally, by construction we see that
	\begin{align}
	r^3 u_{\al, e,0}(r) &= O(r^{-1}), \label{s912} \\
	r \int_r^\infty u_{\al, e,1}(\rho) \rho d\rho &= O(r^{-1}). \label{s913}
	\end{align} 
	
	Using \eqref{s97}--\eqref{s913}, we may repeat the previous arguments with $u_{e,\al}$ in place
	of $u_e$ to conclude the following analog of Lemma \ref{l515}:
	
	\begin{lem}\label{l515 al}
		$\vec u_{\al}(0,r) = (0,0)$ for $r \geq \eta$. 
	\end{lem}
	
	Finally, we obtain the following analog of Lemma \ref{allt lem}:
	
	\begin{lem}\label{allt lem al}
		We have 
		\begin{align*}
		\vec u_{\alpha}(t,r) = (0,0)
		\end{align*}
		for all $t \in \R$ and $r > 0$. 
	\end{lem}
	
	Equivalently, Lemma \ref{allt lem al} states that 
	\begin{align*}
	\vec u(t,r) = (U_+(r),0) 
	\end{align*}
	for all $t \in \R$ and $r > 0$. This concludes the proof of Lemma \ref{l516} and Proposition \ref{positive lemma}.  
	
\end{proof}

\subsection{Proof of Proposition \ref{p53}}

Using Proposition \ref{positive lemma} and its analog for $r < 0$, we quickly conclude the proof of Proposition \ref{p53}.  Indeed, we know that there exists static solutions $U_{\pm}$ to \eqref{s81} such that 
\begin{align}\label{allt pmr}
\vec u(t,r) = (U_{\pm}(r), 0)
\end{align}
for all $\pm r > 0$ and $t \in \R$.  In particular, $\p_t u(t,r) = 0$, 
$\p_r u(t,r) = \p_r u(0,r)$ and $u(t,r) = u(0,r)$ for all $t$ and almost every $r$.  Let $\psi \in C^\infty_0(\R)$ with 
$\int \psi dt = 1$ and let $\varphi \in C^\infty_0(\R)$.  Then since 
$u$ solves \eqref{s81} in the weak sense, we conclude that 
\begin{align*}
0 &= \int \int \bigl [ \psi'(t) \varphi(r) \p_t u(t,r) + \psi (t)\varphi'(r) \p_r u(t,r) + V(r) \psi(t) \varphi(r) u(t,r) \\&\hspace{1.2 cm} - \psi(t) \varphi(r) N(r,u(t,r))  \bigr ] (r^2 + 1)^2 dr dt \\
&= \int \int \psi(t) \bigl [ \varphi'(r) \p_r u(0,r) + V(r) \varphi(r) u(0,r) - \varphi(r) N(r,u(0,r))  \bigr ] (r^2 + 1)^2 dr dt \\
&= \int \bigl [ \varphi'(r) \p_r u(0,r) + V(r) \varphi(r) u(0,r) - \varphi(r) N(r,u(0,r))  \bigr ] (r^2 + 1)^2 dr.
\end{align*}
Since $\varphi$ was arbitrary, we see that $u(0,r)$ is a weak solution
in $H^1(\R)$ to the static equation $-\p_r^2 u - \frac{4r}{r^2+1} \p_r u + V(r) u = N(r,u)$ on $\R$.  By standard arguments we conclude that $u(0,r)$ is a classical solution.  Thus, $\vec u(t,r) = (U(r),0) := (u(0,r),0)$ for all $t,r \in \R$ as desired. 

\qed

\subsection{Proofs of Proposition \ref{p51} and Theorem \ref{t41}}

We briefly summarize the proofs of Proposition \ref{p51} and Theorem \ref{t41}.  From Proposition \ref{p53}, 
we obtain Proposition \ref{p51}.  

\begin{proof}[Proof of Proposition \ref{p51}]
	By Proposition \ref{p53}, we have that $\vec u = (U, 0)$ for some finite energy static solution $U$ to \eqref{s81}.  Thus, 
	$\psi = Q + \la r \ra U$ is a finite energy static solution to \eqref{s41}, i.e. a harmonic map.  By Proposition \ref{harm},
	the harmonic map $Q$ is the unique finite energy static solution to \eqref{s41} so that $Q = \psi = Q + \la r \ra U$ whence  
	$\vec u = (0,0)$ as desired.  
\end{proof}

Using Proposition \ref{p43} and Proposition \ref{p51}, we conclude the proof of our main result Theorem \ref{t41} (equivalently
Theorem \ref{t01}). 

\begin{proof}[Proof of Theorem \ref{t41}]
	Suppose that Theorem \ref{t41} fails.  Then by Proposition \ref{p43}, there exists a nonzero solution $u_*$ to \eqref{s44} such that 
	the trajectory 
	\begin{align*}
	K := \left \{ \vec u_*(t) : t \in \R \right \},
	\end{align*}
	is precompact in $\h$.  However, by Proposition \ref{p51}, we must have that $\vec u_* = (0,0)$, which contradicts the fact
	that $u_*$ is nonzero.  Thus, Theorem \ref{t41} holds.  
\end{proof}


\begin{thebibliography}{9}

\bibitem{bahger} Bahouri, Hajer; G\'erard, Patrick. High frequency approximation of solutions to critical nonlinear wave equations. \emph{Amer. J. Math.} 121 (1999),
no. 1, 131--175.

\bibitem{biz1} Bizon, Piotr; Chmaj, and M. Maliborski. Equivariant wave
maps exterior to a ball. \emph{Nonlinearity (5)} 25 (2012), 1299--1309.

\bibitem{biz2} Bizon, Piotr; Kahl; Michal. Wave maps on a wormhole. \emph{Preprint}, 12 2014. 

\bibitem{coul} Coulhon, Thierry; Russ, Emmanuel; Tardivel-Nachef, Val\'erie. Sobolev algebras on Lie groups and Riemannian manifolds. 
\emph{Amer. J. Math.} 123 (2001), no. 2, 283--342. 

\bibitem{cks} C\^ote, Rapha\"el; Kenig, Carlos E.; Schlag, Wilhelm. Energy partition for the linear radial wave equation. \emph{Math. Ann.} 358 (2014), no. 3-4, 573--607.

\bibitem{dkm4} Duyckaerts, Thomas; Kenig, Carlos; Merle, Frank. Classification of radial solutions of the focusing, energy critical wave equation. \emph{Cambridge Journal Mathematics} 1 (2013), no. 1, 74--144.  

\bibitem{fjtt} Franklin, Paul; James, Oliver; Thorne, Kip S.; von Tunzelmann, Eugenie.  Visualizing Interstellar's Wormhole. \emph{Preprint}, 04 2015. 

\bibitem{geba1} Geba, Dan-Andrei; Rajeev, Sarada G. A continuity argument for a semilinear Skyrme model. 
\emph{Electron. J. Differential Equations} 2010, No. 86, 9 pp.

\bibitem{geba2} Geba, Dan-Andrei; Rajeev, S. G. Nonconcentration of energy for a semilinear Skyrme model. 
\emph{Ann. Physics} 325 (2010), no. 12, 2697--2706.

\bibitem{ges} Gesztesy, Fritz; Zinchenko, Maxim. On spectral theory for Schr\"odinger operators with strongly singular potentials. 
\emph{Math. Nachr.} 279 (2006), no. 9-10, 1041--1082.

\bibitem{keeltao} Keel, Markus; Tao, Terence. Endpoint Strichartz estimates. \emph{Amer. J. Math.} 120 (1998), no. 5, 955--980. 

\bibitem{km06}  Kenig, Carlos E.; Merle, Frank. Global well--posedness, scattering and blow-up for the energy-critical, focusing,
non-linear Schr\"odinger equation in the radial case. \emph{Invent. Math.} 166 (2006), no. 3, 645-675.

\bibitem{km08} Kenig, Carlos E.; Merle, Frank. Global well-posedness, scattering and blow-up for the energy-critical focusing non-linear wave equation.
\emph{Acta Math.} 201 (2008), no. 2, 147--212.

\bibitem{kls1} Kenig, Carlos E.; Lawrie, Andrew; Schlag, Wilhelm. Relaxation of wave maps exterior to a ball to harmonic maps for all data. \emph{Geom. Funct. Anal.} 24 (2014), no. 2, 610--647.

\bibitem{klls1}  Kenig, Carlos; Lawrie, Andrew; Liu, Baoping; Schlag, Wilhelm. Channels of energy for the linear radial wave equation. \emph{Adv. Math.} 285 (2015), 877--936.

\bibitem{klls2}  Kenig, Carlos; Lawrie, Andrew; Liu, Baoping; Schlag, Wilhelm. Stable soliton resolution for exterior wave maps in all equivariance classes. \emph{Adv. Math.} 285 (2015), 235--300.

\bibitem{los} Lawrie, Andew; Oh, Sung--Jin; Shahshahani, Sohrab.  Equivariant wave maps on the hyperbolic plane with large
energy. \emph{Preprint}, 05 2015.  

\bibitem{ls} Lawrie, Andrew; Schlag, Wilhelm. Scattering for wave maps exterior to a ball. \emph{Adv. Math.} 232 (2013), 57--97.

\bibitem{mctr} McLeod, J. B.; Troy, W. C. The Skyrme model for nucleons under spherical symmetry. \emph{Proc. Roy. Soc. Edinburgh Sect.} A 118 (1991), no. 3--4, 
271--288.

\bibitem{mt} Morris, Michael S.; Thorne, Kip S. Wormholes in spacetime and their use for interstellar travel: 
a tool for teaching general relativity. \emph{Amer. J. Phys.} 56 (1988), no. 5, 395--412.

\bibitem{cr2} Rodriguez, Casey. Soliton resolution for equivariant wave maps on a wormhole: II. \emph{Preprint}, 2016.

\bibitem{sch} Schlag, Wilhelm.  Semilinear wave equations.  ICM Proceedings, 2014.

\bibitem{sss1} Schlag, Wilhelm; Soffer, Avy; Staubach, Wolfgang. Decay for the wave and Schr\"odinger evolutions on manifolds with conical ends. 
I. \emph{Trans. Amer. Math. Soc.} 362 (2010), no. 1, 19--52. 

\bibitem{sss2} Schlag, Wilhelm; Soffer, Avy; Staubach, Wolfgang. Decay for the wave and Schr\"odinger evolutions on manifolds with conical ends.
 II. \emph{Trans. Amer. Math. Soc.} 362 (2010), no. 1, 289--318.

\bibitem{shst} Shatah, Jalal; Struwe, Michael. Geometric wave equations. Courant Lecture Notes in Mathematics, 2. \emph{New York University, Courant Institute of Mathematical Sciences, New York; American Mathematical Society, Providence, RI,} 1998. viii+153 pp.

\bibitem{sogge} Sogge, Christopher D. Lectures on non-linear wave equations. Second edition. \emph{International Press, Boston, MA}, 2008. x+205 pp. 

\bibitem{tao} Tao, Terence. Nonlinear dispersive equations. Local and global analysis. CBMS Regional Conference Series in Mathematics, 106. \emph{Published for the Conference Board of the Mathematical Sciences, Washington, DC; by the American Mathematical Society, Providence, RI}, 2006. xvi+373 pp. 

\bibitem{zhang} Zhang, Junyong. Strichartz estimates and nonlinear wave equation on nontrapping asymptotically conic manifolds. 
\emph{Adv. Math.} 271 (2015), 91--111.

\end{thebibliography}
\end{document}